\documentclass{amsart}
\usepackage[margin=1.25in]{geometry}
\usepackage[utf8]{inputenc}
\usepackage{color}
\usepackage{tikz-cd}
\usepackage[hidelinks]{hyperref}
\usepackage{enumerate}
\usepackage{theoremref}
\usepackage{amssymb, amsthm, amsmath, bbm }
\usepackage{braket}
\usepackage{array} 
\usepackage{etoolbox}

\newcounter{smallroman}

\makeatletter
\preto\env@matrix{}
\makeatother

\newcolumntype{C}[1]{>{\centering\let\newline\\\arraybackslash\hspace{0pt}}m{#1}}

\theoremstyle{theorem}
\newtheorem{theorem}{Theorem}[section]
\newtheorem{corollary}[theorem]{Corollary} 
\newtheorem{lemma}[theorem]{Lemma}
\newtheorem{proposition}[theorem]{Proposition}

\newtheorem{innercustomthm}{Theorem}
\newenvironment{customthm}[1]
  {\renewcommand\theinnercustomthm{#1}\innercustomthm}
  {\endinnercustomthm}

\usepackage{array} 
\usepackage{amssymb}

\numberwithin{equation}{section}

\theoremstyle{definition}

\newtheorem{example}[theorem]{Example}
\newtheorem{definition}[theorem]{Definition}
\newtheorem{remark}[theorem]{Remark}
\newtheorem{notation}[theorem]{Notation}

\newtheorem*{conventionstar}{Conventions}

\theoremstyle{definition}

\def\r{\rho}

\newcommand{\CC}{{\mathbb{C}}}
\newcommand{\RR}{{\mathbb{R}}}

\newcommand{\ZZ}{{\mathbb{Z}}}
\newcommand{\PP}{{\mathbb{P}}}
\newcommand{\GG}{{\mathbb{G}}}

\newcommand{\SL}{\operatorname{SL}}
\newcommand{\GL}{\operatorname{GL}}
\newcommand{\PGL}{\operatorname{PGL}}
\newcommand{\Lie}{\operatorname{Lie}}
\newcommand{\tr}{\operatorname{tr}}

\newcommand{\Mat}{{\rm Mat}}

\newcommand{\zmin}{Z_{\min}}
\newcommand{\xmino}{X_{\min}^0}
\newcommand{\gitq}{/\!/}
\newcommand{\hU}{\widehat{U}}
\newcommand{\hX}{\widehat{X}}
\newcommand{\Stab}{\operatorname{Stab}}

\newcommand{\vmin}{V_{\operatorname{min}}}
\newcommand{\omegamin}{\omega_{\operatorname{min}}}
\newcommand{\Proj}{\operatorname{Proj}}
\newcommand{\Spec}{\operatorname{Spec}}
\newcommand{\Lc}{\mathcal{L}}
\newcommand{\vmino}{V^0_{\operatorname{min}}}
\newcommand{\Rep}{\operatorname{Rep}}
\newcommand{\Aut}{\operatorname{Aut}}

\newcommand{\HMss}[2]{{#1}^{{#2}-HMss}}

\newcommand{\HMs}[2]{{#1}^{{#2}-HMs}}
\newcommand{\Iss}[2]{{#1}^{{#2}-Iss}}

\newcommand{\Qss}[2]{{#1}^{{#2}-Qss}}

\newcommand{\justss}[2]{{#1}^{{#2}-ss}}
\newcommand{\justs}[2]{{#1}^{{#2}-s}}
\newcommand{\justssors}[2]{{#1}^{{#2}-(s)s}}

\newcommand{\conv}{\operatorname{conv}}

\newcommand{\rar}{\rightarrow}
\newcommand{\Gm}{\mathbb{G}_m}
\newcommand{\mc}{\mathcal}
\title[Affine Non-Reductive GIT and quivers with multiplicities]{Affine Non-Reductive GIT and moduli of representations of quivers with multiplicities}
\author{Eloise Hamilton, Victoria Hoskins and Joshua Jackson}
\date{}

\begin{document}

\maketitle

\vspace{-0.7cm} 
\begin{abstract}
     We give an explicit approach to quotienting affine varieties by linear actions of linear algebraic groups with graded unipotent radical, using results from projective Non-Reductive GIT. Our quotients come with explicit projective completions, whose boundaries we interpret in terms of the original action. As an application we construct moduli spaces of semistable representations of quivers with multiplicities subject to certain conditions, which always hold in the toric case for a generic stability condition. 
\end{abstract}
\setcounter{tocdepth}{1}
\tableofcontents

\section{Introduction}
For an action of a linear algebraic group $H$ on an algebraic variety $Y$, Geometric Invariant Theory (GIT) is concerned with constructing a quotient variety. When $H$ is reductive and the action can be linearised, a solution is provided by Mumford's classical GIT \cite{Mumford1994}, which now lies at the heart of moduli theory in algebraic geometry. When $H$ is non-reductive, the problem of constructing a quotient variety is far more difficult. Nevertheless, a recent programme of research \cite{Berczi2016, Berczi2018,Berczi2023}, led principally by B\'{e}rczi, Doran, Hawes and Kirwan (BDHK), has shown that many desirable properties of Mumford's GIT can be extended to the non-reductive setting, provided $H$ has \emph{graded unipotent radical} (see \thref{def internally graded unipotent radical}) and $Y$ is projective. Under certain assumptions on the unipotent stabilisers, their work yields a projective quotient of an explicit open subset of $Y$; if these assumptions are not satisfied then they can be achieved by performing a sequence of blow-ups of $Y$.  This provides a viable Non-Reductive GIT (NRGIT), which has already led to many applications \cite{Bunnett2021,Jackson2021,Hoskins2021,Qiao2022b,Cooper2023,Berczi2024}. 

This paper is in part a response to a major conceptual difference between the reductive and non-reductive theories as they presently stand. While reductive GIT quotients are obtained locally by gluing affine GIT quotients, NRGIT quotients have only been constructed globally for projective varieties. There are good reasons for this: the challenges of defining non-reductive quotients (non-finite generation of invariants, non-separation of closed orbits and non-surjectivity of the quotient map) present themselves affine-locally, so affine non-reductive GIT quotients cannot be defined in the same way as affine reductive GIT quotients (i.e.\ as the spectrum of invariant functions). The strategy of BDHK is to assume that $Y$ is projective, and to use this global structure to overcome the obstacles. This approach means that NRGIT in its current form does not work if $Y$ is affine. 

There are two ways to develop an affine version of NRGIT: via an equivariant embedding into a projective space or via a more abstract relative affine approach. This paper takes the former approach, as it is more explicit and provides projective completions of our quotients, whereas in upcoming work \cite{HHJ} we develop the latter approach which applies more broadly, but does not yield natural projective completions. More precisely, for an $H$-action on an affine variety $Y$, we take equivariant embeddings $Y \hookrightarrow V \cong  \mathbb{A}^n \hookrightarrow X:=\mathbb{P}(V \oplus k)$ and construct a quotient, provided $H$ has graded unipotent radical and certain stabiliser assumptions are satisfied, by applying the results of BDHK. Implementing this idea involves carefully interpreting the relevant stabiliser conditions and semistable sets for $X$ in terms of the $H$-representation $V$. Moreover, there is a trichotomy in the results we obtain: the process looks different depending on whether the minimal weight in $V$ for the grading multiplicative group is positive, zero, or negative. As the linear action always fixes the origin $0 \in V \subset X$, in two of these cases one must blow-up this origin to achieve the unipotent stabiliser assumptions required by projective NRGIT. In each case, we obtain a quotient of an open semistable set in $Y$ with an explicit projective completion.

Our results also cover the case where a character $\rho$ of $H$ is used to twist the linearisation, as considered by King in the reductive setting for the purpose of constructing moduli spaces of quiver representations \cite{King1994}. In the reductive case, Halic \cite{Halic2004, Halic2010} produces a linearisation $\Lc_{\rho}$ on the projective space $X = \PP( V \oplus k)$ with the property that the corresponding semistable locus $X^{ss}(\Lc_{\rho}) \subseteq X$ intersected with $V$ is King's $\rho$-semistable locus $V^{ss}(\rho) \subseteq V$, i.e.\  $$ V^{ss}(\rho) = V \cap X^{ss}(\Lc_{\rho}). $$ Thus the $\rho$-twisted affine quotient of $V$ can be constructed as an open subset of the projective GIT quotient of $X$. In the non-reductive case, the analogous semistable locus $V^{ss}(\rho)$ can be defined invariant-theoretically, but this locus may not admit an explicit description or a good $H$-quotient. Our strategy for developing twisted affine NRGIT is to take the equation above as the \emph{definition} of $V^{ss}(\rho)$ for non-reductive groups, and to use projective NRGIT to obtain a quotient. The difficulty lies in describing the intersection and its quotient in terms of information solely about $V$.

To illustrate our results, we include some simple worked examples involving two-by-two matrices. We conclude with a novel application: the construction of moduli spaces of representations of quivers with multiplicities. These are representations of a quiver in the category of modules over a truncated polynomial ring, as considered in \cite{Ringel2011} for the dual numbers, and \cite{Geiss2014} 
for an arbitrary truncated polynomial ring to build realisations of symmetrisable Kac-Moody algebras. In \cite{Wyss2017,Hausel2018,Vernet2023}, the polynomial behaviour and positivity properties of counts of absolutely indecomposable representations of a quiver with multiplicities over finite fields are studied, which partially extend corresponding results for classical representations \cite{Kac80,Crawley2004,Hausel2013}.

~\\
\noindent \textbf{Main results.} As explained above, for an $H$-action on an affine variety $Y$, we construct a quotient by considering equivariant embeddings $Y \hookrightarrow V \cong \mathbb{A}^n \hookrightarrow X:= \PP(V \oplus k)$ and appropriate linearisations. As in projective NRGIT, we assume that the unipotent radical $U$ of $H$ is graded by a copy of the multiplicative group $\mathbb{G}_m$ (see \thref{def internally graded unipotent radical}) and construct a quotient in stages: 
first by the graded unipotent group $\hU:= U \rtimes \mathbb{G}_m$, then by the residual reductive group $H/\hU$. Our first two results are affine versions of the projective NRGIT results for quotienting by $\hU$ and by $H$. Here we simply state our results for quotients of the affine space $V$ and refer to Section \ref{subsec:fromvstoaffine} for the case of the affine variety $Y$; note that unlike in the reductive case, the extension from affine spaces to affine varieties is not immediate, as taking invariants is not in general exact for non-reductive groups. Fortunately, the BDHK approach to NRGIT yields a theory which is functorial with respect to equivariant closed embeddings that respect the linearisations.

\begin{customthm}{A}[$\hU$-Theorem for affine spaces, see \S \ref{subsec:fromvstoaffine} for affine varieties] \thlabel{firstresult} Let $\hU= U \rtimes \GG_m$ be a graded unipotent group acting linearly on a vector space $V$. Suppose that 
    $\Stab_U(v) = \{e\}$ for all non-zero vectors $v$ in the minimal $\GG_m$-weight space $\vmin$. Then there exists a quasi-projective geometric quotient $$ {V}^{\hU-ss} \to V \gitq \hU : =  {V}^{\hU-ss} / \hU$$ for the $\hU$-action on an explicit \emph{semistable locus} ${V}^{\hU-ss}$, and an explicit projective completion of  $V \gitq \hU$ constructed as a projective NRGIT quotient (see \thref{quotienting}).
  \end{customthm}  

  We illustrate \thref{firstresult} with two examples involving two-by-two matrices in Section \ref{sec:2by2ex}.

\begin{customthm}{B}[Twisted NRGIT for affine spaces, see \S \ref{subsec:fromvstoaffine} for affine varieties] \thlabel{secondresult} 
   Let $H = U \rtimes R$ be a linear algebraic group with graded unipotent radical acting linearly on a vector space $V$, and let $\rho$ be a character of $H$ that is trivial on the grading $\GG_m$. Suppose that the action of $H$ on $V$ satisfies the conditions \ref{Rcondaffine} and \ref{Usscondaffine}. Then there exists a quasi-projective geometric quotient $$ \justss{V}{H}(\rho) \to V \gitq_{\hspace{-2pt} \rho \hspace{2pt}} H  : = \justss{V}{H}(\rho) / H$$ for the action of $H$ on an explicit  $\rho$-twisted semistable locus $\justss{V}{H}(\rho)$, and an explicit projective completion of $V \gitq_{\hspace{-2pt} \rho \hspace{2pt}} H $ defined as a projective NRGIT quotient (see \thref{thmQsslocusforH,Hquotient}). 
\end{customthm}

The conditions \ref{Rcondaffine} and \ref{Usscondaffine} appearing in \thref{secondresult} correspond roughly to stabiliser assumptions for the actions of $R$ and of $U$ respectively. 
\thref{secondresult} generalises \thref{firstresult} in two ways: it enables the construction of $H$-quotients rather than just $\hU$-quotients, and allows for twisting by a character as in \cite{King1994}, which ultimately enables us to construct moduli spaces for representations of quivers with multiplicities.  

A representation of a quiver $Q$ with multiplicity $m \in \mathbb{N}$, also called a representation of $(Q,m)$, is a representation of $Q$ in the category of free finite rank modules over the truncated polynomial ring $k[\epsilon]/(\epsilon^{m+1})$. If $m=0$, then a representation of $(Q,m)$ is simply a classical quiver representation, i.e.\ a representation of $Q$ in the category of vector spaces over $k$. The isomorphism classes of representations of $(Q,m)$ of fixed rank vector $d$ are in bijection with the orbits of a non-reductive group acting on an affine space $V = \Rep(Q,m,d)$.  As in the classical case studied by King, we use a character to twist the linearisation, and there is a diagonally embedded group $\Delta_m$ acting trivially. The projection onto the minimal weight space $V \cong \Rep(Q,m,d) \rightarrow V_{\min} \cong \Rep(Q,d)$ sends a representation $\Phi$ to its \emph{classical truncation} $\Phi^0$, where $\Rep(Q,d):= \Rep(Q,0,d)$.

\begin{customthm}{C}[Moduli for representations of quivers with multiplicities] \thlabel{thirdresult}
   Fix a quiver $Q$, a multiplicity $m \in \mathbb{N}$ and a rank vector ${d} = (d_v)_{v \in Q_0}$. Let $\rho \in \mathbb{Z}^{Q_0}$ denote a non-zero stability condition. Suppose that the following two conditions are satisfied: \begin{enumerate}[(i)]
       \item any $\rho$-semistable classical representation of $Q$ of dimension $d$ is $\rho$-stable;    \label{assumption1thm} 
       \item any representation $\Phi$ of $(Q,m)$ of rank $d$ whose classical trunctation $\Phi^0$ is $\rho$-stable has the property that $\Aut_{(Q,m)}({\Phi^0}) / \Delta_m$ is reductive.  \label{assumption2thm}
    \end{enumerate} 
   Then there exists a quasi-projective coarse moduli space $\mathcal{M}^s_{\rho}(Q,m,d)$ for rank $d$ representations ${\Phi}$ of $(Q, m)$ such that ${\Phi}^0$ is $\rho$-stable as a representation of $Q$, with an explicit projective completion. Moreover, if $\rho$ is generic (see \thref{genericstab}) with respect to a toric rank vector ${d}=(1,\hdots, 1)$, then \eqref{assumption1thm} and \eqref{assumption2thm} are always satisfied, so there is a quasi-projective coarse moduli space for toric representations $\Phi$ of $(Q,m)$ such that $\Phi^0$ is $\rho$-stable, with an explicit projective completion.
\end{customthm}

Key to proving \thref{thirdresult} is showing that \eqref{assumption1thm} and \eqref{assumption2thm} ensure that \ref{Rcondaffine} and \ref{Usscondaffine} are satisfied in the corresponding NRGIT set-up for the construction of the moduli space. \thref{thirdresult} suggests a NRGIT notion of $\rho$-stability for representations of quivers with multiplicities (see \thref{stabwithmult}), which reduces to the classical notion of $\rho$-stability introduced in \cite{King1994} when $m = 0$, and moreover implies a natural Rudakov-type stability notion for representations of $(Q,m)$ that involves verifying an inequality for all subrepresentations (see \thref{lem comparing two stab defs for quivers}).

\subsection*{Structure of the paper.} We review GIT for reductive groups and show how affine GIT can be recovered from projective GIT in Section \ref{sec:redreview}. We review projective GIT for non-reductive groups in Section \ref{sec:projnrgit}. We build an affine version of this theory in the case of graded unipotent group actions in Section \ref{sec:affineuhat}. Section \ref{sec:2by2ex} illustrates this theory through two examples involving two-by-two matrices. We extend the theory to internally or externally graded linear algebraic group actions in Section \ref{sec:affinegitfromprojgitforH}, and we apply this theory to representations of quivers with multiplicities in Section \ref{sec:repqwm}.

\begin{conventionstar} 
We work throughout over an algebraically closed field $k$ of characteristic zero. By variety we mean a reduced scheme of finite type over $k$, not necessarily irreducible. Given a vector space $V$ and a subvariety $V' \subseteq V$, we define $ \PP(V') : = \{ [v] \in \PP(V) \ | \ v \in V' \setminus \{ 0 \} \}$  to be the image of $V' \setminus \{0\}$ under the quotient map $V \setminus \{0\} \to \PP(V)$.
\end{conventionstar}

\subsection*{Acknowledgments} EH thanks Emmanuel Letellier and Gergely B\'{e}rczi for useful discussions on representations of quivers with multiplicities.

\section{From affine GIT to projective GIT and back for reductive groups} \label{sec:redreview}

Throughout this section we let $G$ denote a reductive linear algebraic group. 

\subsection{Affine GIT} \label{subsec:untwistedaffine} 
If $G$ acts on an affine variety $Y = \operatorname{Spec} A$, then the algebra of invariants $A^G$ is finitely generated, as $G$ is reductive. The inclusion of the invariant ring $A^G \subseteq A$ induces a good quotient (see \cite{Newstead1978} for the definition of good and geometric quotient) to the \emph{affine GIT quotient} $$ \pi: Y \to Y \gitq G:= \operatorname{Spec} A^G,$$ which identifies any two orbits whose closures meet. The restriction of $\pi$ to the \emph{stable} locus $$\justs{Y}{G} : = \{ y \in Y\ | \ G \cdot y \text{ is closed and } \operatorname{Stab}_G(y) \text{ is finite} \}$$ is a geometric quotient $\justs{Y}{G}/G$, which is an open subvariety of $Y \gitq G$. 

We recall the definition of the null cone, which will play an important role from Section \ref{sec:affineuhat}. 

\begin{definition}[Null cone] \label{defn null cone}
      For a linear $G$-action on a vector space $V$, the \emph{null cone} $\mathcal{N}$ is the locus in $V$ where all non-constant invariants vanish, i.e.\ \[\mathcal{N} : = \pi^{-1}(\overline{0})\] where $\pi: V \rar V\gitq G$ is the affine GIT quotient and $\overline{0} : = \pi(0)$. \end{definition}

\subsection{Projective GIT} \label{subsec:projgit}
To construct a GIT quotient of an action of $G$ on a projective variety $X$, we need to assume that the $G$-action lifts to an ample line bundle $L$ on $X$; this is called a \emph{linearisation} and denoted $\mathcal{L}$ to indicate the data of the underlying line bundle $L$ together with its $G$-equivariant structure. By taking a tensor power of $L$ if necessary, we can assume that $L$ is very ample, so that there is a $G$-equivariant embedding $X \subseteq \PP(H^0(X,\mathcal{L})^{\ast})$, where here we use the notation $\mathcal{L}$ to emphasise that its $G$-equivariant structure determines the $G$-action on $H^0(X,\mathcal{L})$.
Conversely, a representation $G \to \GL(V)$ on a vector space $V$ such that $X \subseteq \PP(V)$ is $G$-invariant determines a linearisation $\mathcal{O}(1)|_{X}$ on $X$. For this reason, we may sometimes specify the  linearisation for a $G$-action on $X \subseteq \PP(V)$ by giving a representation $G \to \GL(V)$ instead.

Since $G$ is reductive, the algebra of invariants $\bigoplus_{k \geq 0} H^0(X, \mathcal{L}^{\otimes k})^G$ is finitely generated. The inclusion of the invariants induces a rational map to a projective variety \begin{equation} X \dashrightarrow X \gitq_{\hspace{-2pt} \mathcal{L} \hspace{2pt}} G : = \Proj  \bigoplus_{k \geq 0} H^0(X, \mathcal{L}^{\otimes k})^G,\label{projrationalmap}
\end{equation} called the \emph{projective GIT quotient} for the $G$-action on $X$ with linearisation $\Lc$. 

We now define three semistable loci, which all coincide for reductive GIT, but may not in the non-reductive case. One of the central theoretical questions of Non-Reductive GIT is to determine the circumstances under which the three loci coincide.

\begin{definition}[Semistable loci]
For a $G$-action on a projective variety $X$ with linearisation $\Lc$,  
\begin{enumerate}[(i)]
    \item the \emph{quotienting semistable locus} $\Qss{X}{H}(\Lc)$ is the domain of definition of the map in \eqref{projrationalmap}.
    \item the \emph{invariant-theoretic semistable locus} is $$\Iss{X}{G} (\mathcal{L}): = \{x \in X \ | \  \exists  \ \sigma \in H^0(X,\Lc^{\otimes k})^G \text{ for some $k > 0$ such that $\sigma(x) \neq 0$} \};$$ 
    \item the \emph{Hilbert-Mumford semistable locus} is $$\HMss{X}{G}(\mathcal{L}) : = \{x \in X \ | \ \forall \text{ one-parameter subgroups (1PSs) $\lambda$ of $G$, $\operatorname{wt}_{\lambda(\GG_m)} (\overline{x}) \leq 0$}\},$$ where $\operatorname{wt}_{\lambda(\GG_m)} (\overline{x})$ denotes the $\lambda(\GG_m)$-weight on the fibre of $L$ over $\overline{x} : = \lim_{t \to 0} \lambda(t) \cdot x$. The \emph{Hilbert-Mumford stable locus} is defined with a strict inequality $\operatorname{wt}_{\lambda(\GG_m)} (\overline{x}) < 0$ instead. 
    \end{enumerate} 
\end{definition}

Since the rational map in \eqref{projrationalmap} is undefined at $x$ if and only if all non-constant invariants vanish at $x$, we have  $\Qss{X}{G}(\mathcal{L}) = \Iss{X}{G}(\mathcal{L})$. The Hilbert-Mumford criterion implies that $\HMss{X}{G}(\mathcal{L}) = \Iss{X}{G}(\mathcal{L})$, so semistability can be described without computing invariants \cite[Thm 2.1]{Mumford1994}. 

\begin{proposition} [Mumford]\thlabel{projallequal}
    For a $G$-action on $X$ with linearisation $\Lc$, we have: $$ \Qss{X}{G}(\mathcal{L}) = \Iss{X}{G}(\mathcal{L}) = \HMss{X}{G}(\mathcal{L}).$$
\end{proposition}

In light of \thref{projallequal}, we make the following definition.

\begin{definition}[Semistable locus for reductive projective GIT]
     The \emph{semistable locus} for the action of $G$ on  a projective variety $X$ with linearisation $\Lc$ is $$ \justss{X}{G}(\Lc) : = \Qss{X}{G}(\Lc) = \Iss{X}{G}(\Lc) = \HMss{X}{G}(\Lc).$$ 
\end{definition}

The rational map \eqref{projrationalmap} restricts to a good quotient $ \justss{X}{G}(\mathcal{L}) \to X \gitq_{\hspace{-2pt} \mathcal{L} \hspace{2pt}} G$, as it is locally given by affine GIT quotients $X_f \to  X_f \gitq G$ for an invariant section $f$ of a positive power of $\mathcal{L}$.  Moreover, the restriction to the \emph{stable locus} $$ \justs{X}{G}(\Lc) := \left\{ x \in X  \ \middle\vert \begin{array}{c}  \exists  \ f \in H^0(X,\Lc^{\otimes k})^G \text{ for some $k > 0$ such that $f(x) \neq 0$,} \\
\Stab_G(x) \text{ is finite and the } G\text{-action on }  X_f \text{ is closed} \end{array} \right\}$$
is a geometric quasi-projective  quotient $ \justs{X}{G}(\Lc) / G$ with projective completion $X \gitq_{\hspace{-2pt} \Lc \hspace{2pt}} G$.

    If the linearisation is clear, we may write $\justss{X}{G}$ and  $X \gitq G$ instead of $\justss{X}{G}(\Lc) $ and $X \gitq_{\hspace{-2pt} \Lc \hspace{2pt}} G $. In particular, if $\mathcal{L}$ is pulled back from $\mathcal{O}(1)$ on $\mathbb{P}(V)$ for a representation $G \to \GL(V)$, we typically omit $\mathcal{L}$ from the notation. 
    If moreover the group $G$ is clear, we may also write $X^{ss}$ for $\justss{X}{G}$.

The Hilbert-Mumford criterion also gives an equality $\HMs{X}{G}(\Lc) = \justs{X}{G}(\Lc)$. In general the Hilbert-Mumford criterion is sometimes formulated for a maximal torus $T \subset G$ as an equality \begin{equation} \justssors{X}{G}(\Lc) = \bigcap_{g \in G} g \justssors{X}{T}(\Lc).\label{reducingtotorus} \end{equation}  
In turn, semistability and stability for the $T$-action can be interpreted combinatorially as follows. Assume $L$ is very ample and consider the $T$-representation $V : = H^0(X,\Lc)^{\ast}$ such that $X \subseteq \PP(V)$. Let $V = \bigoplus_{\chi \in \chi(T)} V_{\chi}$ denote its $T$-weight space decomposition indexed by the character group $\chi(T)$. For $x=[v] \in X$, the \emph{weight polytope of $v$}, denoted $\operatorname{conv}_T(v)$, is the convex hull in $\chi(T)_{\RR}$ of the set of weights $\chi$ such that $v_{\chi} \neq 0$, where $v_{\chi}$ denotes the component of $v$ in $V_{\chi}$. 

\begin{theorem}[Torus (semi)stability, {\cite[Thm 9.2]{Dolgachev2003}}] \thlabel{torussemistabilityproj}
Suppose that a torus $T$ acts on a projective variety $X \subseteq \PP(V)$ via a representation $T \to \GL(V)$. Then given $x =[v] \in X$, we have that $x \in \justss{X}{T}(\Lc)$ if and only if $0 \in \operatorname{conv}_T(v)$, and $x \in \justs{X}{T}(\Lc)$ if and only if $0$ is contained in the interior of $\operatorname{conv}_T(v)$.
\end{theorem} 

\begin{remark}\label{remark:nullconeinstability}
    Via the topological Hilbert--Mumford criterion \cite[Prop 2.2]{Mumford1994}, for a representation $V$ of $G$, the associated linearised $G$-action on $\mathbb{P}(V)$ satisfies
    $ \mathbb{P}(V)^{ss} = \mathbb{P}(V) \setminus \mathbb{P}(\mathcal{N}) = \PP(V \setminus \mathcal{N})$
    where $\mathcal{N} = \pi^{-1}(\overline{0})$ is the nullcone for the affine GIT quotient $\pi : V \rightarrow V \gitq G$ (see Remark \ref{defn null cone}).
\end{remark}

\subsection{Affine GIT from projective GIT}  \label{subsec:affinefromproj} 

In this section we explain how affine GIT can be recovered from projective GIT. By \cite[Lem 1.1]{Kempf1978}, if a reductive group $G$ acts on an affine variety $Y$, then there is a $G$-representation $V$ and a $G$-equivariant inclusion $Y \hookrightarrow V$. Since $G$ is reductive, the affine GIT quotient $Y \gitq G$ can be recovered by restricting the good quotient $V \to V \gitq G$ to $Y$. Thus we can restrict our attention to linear actions on vector spaces without loss of generality. The following result is well-known (for example, see \cite[\S 1]{Schmitt1998}).

\begin{proposition}[Affine GIT from projective GIT] \thlabel{affinefromproj} 
Suppose that $G$ acts on a vector space $V$ via a representation $G \to \GL(V)$. Consider the induced linearised $G$-actions on $\PP(V)$ and $X:= \PP(V \oplus k)$, with the latter determined by the induced representation $G \to \GL(V \oplus k)$ where $G$ acts trivially on $k$. Let $q: X^{ss} \to X \gitq G$ denote the associated GIT quotient and identify $V$ as an open subvariety of $X$ via the map $v \mapsto [v:1]$. Then there is an equality $V \cap X^{ss} = V,   $ and $q$ restricts to  a good quotient $V \rightarrow q(V)$ that is affine, with $q(V) \cong V \gitq G$. Moreover, the complement of $q(V)$ in its projective completion $X \gitq G$ is isomorphic to the GIT quotient $\PP(V) \gitq G$.
\end{proposition} 

\begin{proof}
By the topological Hilbert-Mumford criterion \cite[Prop 2.2]{Mumford1994}, a point $[v:w] \in X$ is semistable if and only if $0 \notin \overline{G \cdot (v,w)}$. Since $G$ acts trivially on $k$, any point $[v:1]$ satisfies $0 \notin \overline{G \cdot (v,1)}$, so that all points in $V$ are in $X^{ss}$. Therefore $V \cap X^{ss} = V$. 

Since good quotients restrict to good quotients on saturated open subvarieties \cite[\S 1.1]{Birula1991}, $q$ restricts to a good quotient on the open saturated subvariety $V \cap X^{ss} = V$. The quotient $q(V)$ is affine as it is given by the non-vanishing locus of the $G$-invariant section $w$ on $X$ which picks out the coordinate at infinity. By uniqueness of good quotients,  $q(V)$ is isomorphic to $V \gitq G$. 

By writing $X$ as $X = V \sqcup \PP(V)$, we can identify the complement of $V \gitq G$ inside $X \gitq G$ as the image of $\PP(V) \cap X^{ss}$ under $q$, which is a good quotient for the action on $\PP(V) \cap X^{ss}$. For the linear $G$-action on $\PP(V)$ induced by the representation of $G$ on $V$, we have $ \PP(V) \cap X^{ss} =\PP(V)^{ss}$ by the topological Hilbert-Mumford criterion. Indeed, $\overline{G \cdot (v,0)} $ does not contain $(0,0)$ if and only if $\overline{G \cdot v}$ does not contain $0$. Therefore $\PP(V) \gitq G$ is a good quotient of $\PP(V)^{ss} = \PP(V) \cap X^{ss}$, so that $X \gitq G \setminus V \gitq G$ is isomorphic to $\PP(V) \gitq G$ by uniqueness of good quotients. 
\end{proof}

\subsection{Twisted affine GIT} \label{subsec:twistedaffine}  
Suppose that $G$ acts linearly on a vector space $V$ and let $\rho$ denote a character of $G$. As in \cite{King1994} we will use $\rho$ to twist the induced action of $G$ on the trival line bundle $ V \times k$ and produce a so-called twisted affine GIT quotient.

A function $f \in k[V]$ is a \emph{semi-invariant with respect to $\rho$} if there is an $n \in \mathbb{N}$ such that $f(g \cdot v) = \rho(g)^n f(v)$ for each $v \in V$ and $g \in G$. In this case $n$ is the \emph{weight} of $f$ and we denote it by $\operatorname{wt}_{\rho} ( f)$. 
The ring $k[V]^{G,\rho}$ of semi-invariants (with respect to $\rho$) is graded by the weight:
 $$k[V]^{G,\rho} = \bigoplus_{n \geq 0} k[V]^G_{\rho^n},$$ where $k[V]^G_{\rho^n}$ denotes semi-invariants of weight $n$.  Each graded piece can be identified as the space of invariant sections for a $G$-equivariant line bundle on $V$ in the following way. Let $\mathcal{L}_p$ denote the linearisation on the trivial line bundle $L:= V \times k$ with $G$-action $g \cdot (v,z) = (g \cdot v, \rho(g) z)$. Then an invariant section of ${\Lc}^{\otimes n}_{\rho}$ is a function $f z^n \in k[V \times k]=k[V][z]$ such that $f \in k[V]$ is a semi-invariant (with respect to $\rho$) of weight $n$. Note that there is another grading of $k[V]^{G, \rho}$, given by the degree of its elements as polynomials on $V$.  

Using the equality $k[V]^{G,\rho} = (k[V][z])^G$, where $G$ acts on $z$ via the character $\rho^{-1}$, the inclusion of invariants $k[V]^{G,\rho} \subseteq k[V][z]$ induces a rational morphism \begin{equation} \Proj k[V][z] = \Spec k[V] = V \dashrightarrow  V \gitq_{\hspace{-2pt} \rho} \hspace{1pt} G := \Proj k[V]^{G,\rho} \label{rationalmap}
\end{equation} 
to the \emph{twisted affine GIT quotient of $V$ by $G$ (with respect to $\rho$)}. By definition, $V \gitq_{\hspace{-2pt} \rho} \hspace{1pt} G$ admits a projective morphism to $\operatorname{Spec} k[V]_{\rho^0}^G = \operatorname{Spec} k[V]^G$. It follows that the twisted affine GIT quotient $V \gitq_{\hspace{-2pt} \rho \hspace{1pt}} G$ is projective over the `untwisted' affine GIT quotient $V \gitq G$. In particular, if $k[V]^G  = k$ so that $V \gitq G$ is a point, then the twisted affine GIT quotient is projective.

As in the projective case, we define three semistable loci for the action of $G$ on $V$ twisted by $\rho$. Since $G$ is reductive, we will see again that all three semistable loci coincide.

\begin{definition}[$\rho$-semistable loci in $V$] \thlabel{semistableloci}
For a linear $G$-action on a vector space $V$ twisted by $\rho$, 
\begin{enumerate}[(i)]
    \item  the \emph{quotienting $\rho$-semistable locus} is the domain of definition of the rational map in \eqref{rationalmap};
    \item the \emph{invariant-theoretic $\rho$-semistable locus} is $$\Iss{V}{G}(\rho) : = \{v \in V \ | \ \exists \ \sigma \in k[V]^{G}_{\rho^n} \text{ for some $n > 0$ such that $\sigma(v) \neq 0$} \};$$ 
    \item the \emph{Hilbert-Mumford $\rho$-semistable locus} is $$\HMss{V}{G}(\rho) : = \{ v \in V \ | \ \langle \rho, \lambda \rangle \geq 0 \text{ for all 1PSs $\lambda$ of $G$ such that $\lim_{t \to 0} \lambda(t) \cdot x$ exists}\},$$ where $\langle \rho, \lambda \rangle$ is the natural pairing between characters and co-characters of $G$, given by $\langle \rho, \lambda \rangle = r$ when $\rho(\lambda(t)) = t^r$.  The \emph{Hilbert-Mumford $\rho$-stable locus} is defined analogously, with a strict inequality $\langle \rho, \lambda \rangle > 0$ instead. 
\end{enumerate}
\end{definition}

It is clear that $\Qss{V}{G}(\rho) = \Iss{V}{G}(\rho)$, since the morphism in \eqref{rationalmap} is undefined at $v$ if and only if all semi-invariants of strictly positive weight vanish at $v$. Moreover, $\Iss{V}{G}(\rho) = \Qss{V}{G}(\rho)$ by King's Hilbert-Mumford criterion \cite{King1994}. To summarise, we have the following result.

\begin{proposition}[King] \thlabel{allequal}
    For a $G$-action on a vector space $V$ twisted by $\rho$, we have: $$ \Qss{V}{G}(\rho) = \Iss{V}{G}(\rho) = \HMss{V}{G}(\rho).$$
\end{proposition}

The above proposition justifies the following definition. 

\begin{definition}[Semistable locus for reductive twisted affine GIT]
    The \emph{$\rho$-semistable locus} for the linear action of a reductive group $G$ on a vector space $V$ twisted by a character $\rho$ of $G$ is $$\justss{V}{G}(\rho) : = \Qss{V}{G}(\rho) = \Iss{V}{G}(\rho) = \HMss{V}{G}(\rho).$$ 
\end{definition}

 If $\rho$ is the trivial character, then $\justss{V}{G}(\rho) = V$, which we can view as the `untwisted' semistable locus.  If the group is clear, we may write $V^{ss}(\rho)$ instead of $\justss{V}{G}(\rho)$.

The rational map in \eqref{rationalmap} restricts to a good quotient $\Qss{V}{G}(\rho) \to V \gitq_{\hspace{-2pt} \rho} \hspace{1pt} G$. Moreover, this good quotient restricts to a geometric quotient on the \emph{$\rho$-stable locus}  $$ \justs{V}{G}(\rho) := \left\{ v \in V \ \middle\vert \begin{array}{c}  \exists  \ \sigma \in k[V]^{G}_{\rho^n} \text{ for some $n > 0$ such that $\sigma(v) \neq 0$,} \\
\Stab_G(v) \text{ is finite and the } G\text{-action on }  V_{\sigma} \text{ is closed} \end{array} \right\},$$ where $V_{\sigma} = \{v' \in V \ | \ \sigma(v') \neq 0\}.$ Note that $\justs{V}{G}(\rho) = \justs{V}{G}$ if $\rho$ is trivial.

The Hilbert-Mumford criterion also gives an equality $\HMs{X}{G}(\rho) = \justs{V}{G}(\rho)$. The Hilbert-Mumford criterion is also sometimes formulated for a maximal torus $T \subseteq G$ as equalities \begin{equation} \justssors{V}{G}(\rho) = \bigcap_{g \in G} g \justssors{V}{T}(\rho). \label{hm} \end{equation} 
Note that the notation ${(s)s}$ in the the above equation means that the equality holds both for the semistable and stable loci, and we will use this notation  throughout the remainder of the paper where convenient. 
In turn $\rho$-(semi)stability for tori can be characterised geometrically using cones of allowable 1PSs (see \cite[Prop 2.7]{Hoskins2014}), which combined with \eqref{hm} gives a combinatorial characterisation of $\rho$-(semi)stability for any reductive group. In \thref{hmlemma} below we draw attention to another geometric characterisation of $\rho$-semistability for $T$-representations, which is closer in spirit to the characterisation of $T$-semistability in the projective case (see \thref{torussemistabilityproj}), in that it is also formulated using weight polytopes. We will use this characterisation in Section \ref{subsec:twistedaffinefromproj} to provide simpler proofs of Halic's results  \cite{Halic2004,Halic2010}.  

\begin{lemma} \thlabel{hmlemma} 
    Suppose that a torus $T$ acts linearly on a vector space $V$ and let $\rho$ be a character of $T$. Then $v \in \Qss{V}{T}(\rho)$ if and only if the weight polytope of $v$ contains a non-zero point along the positive ray through $\rho$. Moreover, $v \in \justs{V}{T}(\rho)$ if and only if the interior of the weight polytope of $v$ contains a non-zero point along the positive ray through $\rho$.
\end{lemma}

\begin{proof} 
Choose coordinates on $V = \Spec k[x_1,\dots , x_n]$ which diagonalise the $T$-action, so $x_i$ has $T$-weight $\chi_i$. Given $v = (v_1,\dots v_n) \in V$ in these coordinates, let $I = \{i\mid v_i \neq 0\}$ be the support of $v$. Then there exists a semi-invariant $f$ with respect to $\rho$ of weight one with $f(v) \neq 0$ if and only if there is an expression $\rho =  \sum_{i\in I} \alpha_i \chi_i$ with $\alpha_i \in \ZZ_{\geq 0}$, because any such semi-invariant may without loss of generality be taken to be a monomial in the $x_i$'s. Hence, by appropriate scaling by a rational number, there exists a non-vanishing semi-invariant with respect to $\rho$ of weight $n$ for some $n>0$ at $v$ if and only if there exists an expression $q\rho = \sum_{i\in I} r_i\chi_i $ where $q,r_i \in \mathbb{Q}_{\geq 0}$ and $\sum_{i\in I} r_i = 1$. This is equivalent to the condition that some positive rational multiple of $\rho$ lies in the convex hull of the weights $\chi_i$ for $i \in I$, which is exactly the weight polytope of $v$.

To prove the statement for stability, let $v \in V$. If the ray through $\rho$ intersects the boundary of the weight polytope of $v$ and not its interior, then since the origin lies on that ray, there is a cocharacter $\mu: \Gm \rar T$ pairing to zero with all the weights on the face of the boundary of the weight polytope intersecting the ray through $\rho$. If this is not a proper face of the weight polytope, then $\mu(\Gm) \subseteq \Stab_T(v)$, in which case $v$ has infinite stabiliser so is not stable. We can therefore assume that this is a proper face of the weight polytope and that $\mu(\GG_m)$ does not fix $v$. Then possibly after replacing $\mu$ with $\mu^{-1}$, the limit point $\lim_{t \rar 0} \cdot \mu(t)v$ is a point whose weight polytope is exactly the aforementioned face. Such a point is $\rho$-semistable by the above paragraph,  which shows that the orbit of $v$ cannot be closed in the $\rho$-semistable locus. Therefore $v$ cannot be stable. Conversely, if the ray through $\rho$ intersects the weight polytope of $v \in V$ in some interior point, then in particular the interior is non-empty, which implies that the $T$-stabiliser of $v$ is finite. The closure of the torus orbit $Tv$ in $V$ consists of points whose weight polytopes are faces of the weight polytope (which can have any codimension) containing the origin. If the orbit is not closed in the $\r$-semistable locus, then there exists such a limit point which is $\rho$-semistable, i.e.\ a face of the weight polytope containing both the origin and a point on the ray. If this is the case, the intersection of the ray and the weight polytope must be contained in that face, and therefore cannot intersect the interior of the weight polytope, contradicting our assumption. \end{proof}

\subsection{Twisted affine GIT from projective GIT}  \label{subsec:twistedaffinefromproj}

In this section we show how twisted affine GIT can be recovered from projective GIT, providing new proofs of results in \cite{Halic2004,Halic2010}.

\subsubsection{The naive approach}  \label{subsubsec:firstattempt} 
A natural first attempt to recover twisted affine GIT from projective GIT is to use the same set-up as in Section \ref{subsec:affinefromproj} and consider a linearised $G$-action on $X : = \PP(V \oplus k)$ where $G$ acts trivially on $k$.  A reasonable guess for the linearisation is to twist the $G$-action on $\mathcal{O}(1)$ by the character $\rho$. This corresponds to the modified $G$-representation on $V \oplus k$ given by \begin{equation} g \cdot_{\rho} (v,w) := (\rho(g)^{-1} g \cdot v, \rho(g)^{-1} w). \label{linearisationrho} \end{equation}   We denote this linearisation by $\mathcal{O}(1)_{\rho}$. Contrary to the untwisted case where $V \cap X^{ss}(\mathcal{O}(1)) = V=V^{ss}(0)$ (here $0$ denotes the trivial character), we will see that the intersection $V \cap X^{ss}(\mathcal{O}(1)_{\rho})$ may not coincide with the $\rho$-semistable locus $V^{ss}(\rho)$.

\begin{proposition} \thlabel{inclusionvsequality} 
For a linear $G$-action on a vector space $V$ twisted by $\rho$, there are inclusions \[V \cap X^{(s)s}(\mathcal{O}(1)_{\rho}) \subseteq  V^{(s)s}(\rho).\] Moreover: \begin{enumerate}[(i)]
\item  if the inclusion in the semistable case is an equality then any $v \in V^{ss}(\rho)$ admits a non-vanishing semi-invariant $f$ with respect to $\rho$ such that $ \operatorname{deg}  f  \leq \operatorname{wt}_{\rho}(f)$; \label{inequalitydegree} 
\item \label{secondpart} if for each $v \in V^{ss}(\rho)$ there is a non-vanishing semi-invariant with respect to $\rho$ such that the inequality in (i) is strict, then the inclusion in the semistable case is an equality. 
\end{enumerate} 
\end{proposition}

\begin{proof} 
If $[v_0:1] \in V \cap X^{ss}(\mathcal{O}(1)_{\rho})$, then there is a $G$-invariant section of $(\mathcal{O}(1)_{\rho})^{\otimes n}$ for some $n \geq 1$ which does not vanish at $[v_0:1]$. This section can be thought of as a homogeneous polynomial $F$ of degree $n$ such that $F(v_0,1) \neq 0$ and $F(g \cdot v, w) = \rho(g)^n F(v,w)$ for each $(v,w) \in V \oplus k$ and $g \in G$. From $F$, we define $f \in k[V]$ by $f(v) : = F(v,1)$ for each $v \in V$. Then \begin{equation} f(g \cdot v) = F( g \cdot v, 1) = \rho(g)^n F(v,1) = \rho(g)^n f(v) \label{equivariant} 
\end{equation}  for each $v \in V$ and $g \in G$, so $f$ is a weight $n$ semi-invariant, with $f(v_0) = \rho(g)^n F(v_0,1) \neq 0$ as $F(v_0,1) \neq 0$. Thus $v_0 \in V^{ss}(\rho)$. In fact, as $V_f = V \cap X_F$, we see that if $[v_0:1] \in X^{s}(\mathcal{O}(1)_\rho)$, then $v_0 \in V^s(\rho)$ as the stabilisers also coincide. Since $\operatorname{deg} f  \leq \operatorname{deg} F = \operatorname{wt}_{\rho} (f)$, we have also shown \eqref{inequalitydegree}: if $v_0 \in V^{ss}(\rho) = V \cap X^{ss}(\mathcal{O}(1)_{\rho})$ then we can choose $F$ and construct $f$ as above. 

To show \eqref{secondpart}, suppose that $v_0 \in V^{ss}(\rho)$ and $f \in k[V]^G_{\rho^n}$ satisfies $f(v_0) \neq 0$ and $\operatorname{deg} f < n$. Let $F'$ denote the homogenisation of $f$ with respect to the variable $z$, where $k[V \oplus k ] = k[V][z]$, and let $F : = F' z^{n - \operatorname{deg} f} \in k[V][z]$. Then $F$ satisfies \eqref{equivariant} and $F(v,0) = 0$ as $\deg f < n$. Thus $F$ gives a $G$-invariant section of $(\mathcal{O}(1)_{\rho})^{\otimes n}$ which does not vanish at $[v_0:1]$, so $[v_0:1] \in X^{ss}(\mathcal{O}(1)_{\rho})$. 
\end{proof}

We give an example of a non-vanishing semi-invariant $f$ with $\deg f > \operatorname{wt}_{\rho}(f)$, which by \thref{inclusionvsequality} \eqref{inequalitydegree} shows that the inclusion $V \cap X^{ss}(\mathcal{O}(1)_{\rho}) \subseteq  V^{ss}(\rho)$ is not always an equality. 
\begin{example} \thlabel{weightdegree}
Consider the $\GG_m^2$-action on $V = \operatorname{Spec} k[x,y]$ with weights $(1,0)$ and $(0,1)$. For $\rho = (1,1)$, the only $\rho$-semi-invariants are $(xy)^n$ for $n \geq 1$. As $\operatorname{deg} ((xy)^n) = 2n > \operatorname{wt}_{\rho}((xy)^n) = n$, the inclusion $V \cap X^{ss}(\mathcal{O}(1)_{\rho}) \subseteq V^{ss}(\rho)$ must be strict by \thref{inclusionvsequality}. This can be seen directly: $V^{ss}(\rho) = \mathbb{A}^2 \setminus \{xy=0\}$ while $X^{ss}(\mathcal{O}(1)_\rho)  = \emptyset$. 
\end{example}

\subsubsection{New proofs of Halic's approach} \label{subsec:secondattempt} 

We now show that the inclusions in \thref{inclusionvsequality}  become an equality if we linearise the $G$-action on $X$ with respect to a sufficiently high tensor power of $\mathcal{O}(1)$ rather than with respect to $\mathcal{O}(1)$. This recovers and extends results of Halic from \cite{Halic2004,Halic2010}. Thus the twisted affine GIT quotient of $V$ can be recovered from a projective GIT quotient of $X = \PP(V \oplus k)$. The paper \cite{Halic2010} considers only the case where $k[V]^G = k$, while the earlier unpublished paper \cite{Halic2004} also considers the case where $k[V]^G \neq k$. In this section we provide new simpler proofs in both cases using a geometric rather than invariant-theoretic perspective.  
We strengthen Halic's results, by showing that the first equality of \cite[Prop 2.4]{Halic2004} holds also for the stable loci. This strengthening is necessary for proving \thref{ss=sforzmin} in Section \ref{sec:affinegitfromprojgitforH}.  

Let $d \in \mathbb{N}$ and consider the line bundle $\mathcal{O}(d)$ on $X$, which we identify geometrically with the quotient of $ ( (V \times \mathbb{A}^1 ) \setminus \{0\} \times \mathbb{A}^1 ) $ by the $\GG_m$-action $ s \cdot (v,w,z) = (sv, sw, s^d z)$. We let $[v,w,z]_d$ denote the equivalence class of $(v,w,z) \in (V \times \mathbb{A}^1 ) \setminus \{0\} \times \mathbb{A}^1$ after taking the $\GG_m$-quotient. Let $\mathcal{O}(d)_{\rho}$ denote the $G$-linearisation on the line bundle $\mathcal{O}(d)$ given by the following $G$-action  \begin{equation} g \cdot [v,w,z]_d : = [g \cdot v, w, \rho(g) z]_d. \label{actiononod} \end{equation} Note that this is not the linearisation $(\mathcal{O}(1)_{\rho})^{\otimes d} = \mathcal{O}(d)_{\rho^d}$ induced from $G$-linearisation $\mathcal{O}(1)_\rho$ defined in Section \ref{subsubsec:firstattempt} above. Indeed the latter is given by $g \cdot [v,w,z]_d = [g \cdot v, w, \rho(g)^d z]_d = [\rho(g)^{-1} g \cdot v, \rho(g)^{-1} w, z]_d$. In particular, the weight polytopes for the linearisations $\mathcal{O}(1)$ and $\mathcal{O}(d)_{\rho}$ are related as follows:
\begin{equation}\label{compareweightpolytopes} \operatorname{conv}_T^{\mathcal{O}(d)_{\rho}}(v,w) = d \operatorname{conv}_T^{\mathcal{O}(1)}(v,w) - \rho. \end{equation}

\begin{proposition} \thlabel{equalityonV}
Writing $X = V \sqcup \PP(V)$, for $d \gg 1$ we have $$ V \cap X^{ss}(\mathcal{O}(d)_{\rho}) = V^{ss}(\rho) \quad \text{and} \quad  V \cap X^{s}(\mathcal{O}(d)_{\rho}) = V^{s}(\rho). $$ 
\end{proposition} 

\begin{proof}
 By the equalities  \eqref{reducingtotorus} and \eqref{hm},  both sides of each equality can be written as intersections over $G$-translates of (semi)stable loci for a maximal torus $T \subseteq G$, so we may assume $G=T$. 

For the semistable case, let $v \in V^{ss}(\rho)$. By \thref{hmlemma}, the weight polytope $\conv_T(v)$ of $v$ for the $T$-representation $V$ contains a point on the ray through $\rho$. Since $\conv^{\mathcal{O}(1)}_T(v,1)$ is the convex hull of $\{ 0 \} \cup \conv_T(v)$, there is some open segment of the ray through $\rho$, beginning at the origin, which is contained in $\conv_T(v,1)$. It follows that for $d \gg 1$ we have $\rho^{1/d} \in \conv^{\mathcal{O}(1)}_T(v,1)$, which by the equality of weight polytopes in \eqref{compareweightpolytopes} is exactly the condition of being semistable with respect to $\mathcal{O}(d)_{\rho}$. The opposite inclusion follows by the same argument as for \thref{inclusionvsequality}.

For the stable case, take $v\in V^{s}(\rho)$. Then by \thref{hmlemma}, the interior of the weight polytope $\conv_T(v)$ of $v$ for the representation $V$ must contain a point on the ray through $\rho$. The point $[v:1] \in X= \PP(V \oplus k)$ also has $0 \in \conv_T(v,1)$ for the induced $T$-representation on $V \oplus k$ (corresponding to the natural linearisation on $\mathcal{O}_X(1)$), again since $w$ is invariant. Hence there is a small open ray from the origin and in the direction of $\rho$ which is contained in the interior of the weight polytope of $(v,1)$. Thus if $d\gg 1$, we have $[v:1] \in V\cap X^{s}(\mc{O}(d)_{\rho})$. The reverse inclusion is provided by the same argument as for \thref{inclusionvsequality}, just as in the semistable case.\end{proof}

\thref{equalityonV} above describes the intersection of the (semi)stable loci in $X = V \sqcup \PP(V)$ with $V$. In \thref{equalityonPV} below we describe the intersection of the (semi)stable loci in $X$ with $\PP(V)$.  The argument is almost the same as \thref{equalityonV}: the difference is that since we now consider points $[v:0] \in X$ rather than $[v:1] \in X$, we have $\conv_T(v,0) = \conv_T(v)$. That is, we do not automatically have $0 \in \conv_T(v,0)$, which is why we must also delete the null cone. 

\begin{proposition} \thlabel{equalityonPV}
Writing $X = V \sqcup \PP(V)$, for $d \gg 1$ there are equalities $$\PP(V)  \cap X^{(s)s}(\mathcal{O}(d)_{\rho}) =  \PP \left( V^{(s)s}(\rho) \setminus \mathcal{N} \right),$$ where $\mathcal{N}$ is the null cone as defined in Definition \ref{defn null cone}.
\end{proposition} 

\begin{proof}
To begin with, assume that $G=T$ is a torus. By standard facts in reductive GIT, the left-hand side is independent of $d$ for $d\gg 1$, both in the stable and semistable cases. 

For the semistable case, if $[v:0] \in X^{ss}(\mc{O}(d)_{\rho})$ then we must have $\frac{\rho}{d}  \in \conv_T(v,0)$, and if this is true for all $d\gg 1$, then a small open-interval ray based at the origin pointing in the $\rho$ direction must be contained in $\conv_T(v,0)$. This implies that $0\in \conv_T(v,0)$, since otherwise an open ball around the origin would be disjoint from $\conv_T(v,0)$. By Remark \ref{remark:nullconeinstability}, it follows that $v \notin \mathcal{N}$. Additionally, the containment of this open-interval ray tells us that $v \in V^{ss}(\rho)$, by \thref{hmlemma}. This gives the forwards containment in the semistable case. Conversely, if $v \notin \mathcal{N}$ then $0 \in \conv_T(v,0)$, and if additionally $v \in V^{ss}(\rho)$, then a small open ray based at the origin and pointing in the $\r$ direction is contained in $\conv_T(v,0)$, meaning that for all $d\gg 1$ we have $v \in \PP(V)\cap X^{ss}(\mc{O}(d)_{\rho})$, which gives the opposite containment in the semistable case.

   For the stable case, if $[v:0] \in X^{s}(\mc{O}(d)_{\rho})$ then by \eqref{compareweightpolytopes} we see that $\frac{\rho}{d}$ lies in the interior of $\conv_T(v,0) = \conv_T(v)$ and we deduce $v \in V^{ss}(\rho)$, by \thref{hmlemma}. If this holds for any $d\gg 1$, the same argument as in the semistable case then tells us that $0\in \conv_T(v,0)$. This proves the forwards containment in the stable case. Conversely, if $v \notin \mathcal{N} $, then $0\in \conv_T(v) = \conv_T(v,0)$, and if additionally $v \in  V^{s}(\rho) $ then by \thref{hmlemma} some point $a\rho$ for $a\in \mathbb{Q}_{>0}$ of the positive ray pointing towards $\rho$ is contained in the interior of $\conv_T(v).$ Thus $\conv_T(v)$ contains some open ball $ B$ around $a \rho$, and $\conv(B \cup \{0\}) \subseteq \conv_T(v,0)$. As the interior of $\conv(B \cup \{0\})$ contains an open ray pointing from $0$ towards $\rho$, for $d \gg 1$, we conclude that $\frac{\rho}{d}$ lies in the interior of $\conv_T(v,0)$, which implies $[v:0] \in X^{s}(\mc{O}(d)_{\rho})$ and completes the proof for tori.

For an arbitrary reductive group $G$, observe that as in \thref{equalityonV} both sides of the desired equation are equal to the intersection over all maximal tori $T \subseteq G$ of the corresponding loci for the action of $T$, using Remark \ref{remark:nullconeinstability} for the part concerning the null cone for $T$ and $G$.\end{proof}

Combining \thref{equalityonV,equalityonPV}, we can recover the twisted affine GIT quotient $V \gitq_{\hspace{-2pt} \rho}  \hspace{2pt} G$ from a projective GIT quotient.

\begin{corollary}[Twisted affine GIT from projective GIT] \thlabel{summary} 
Suppose that $G$ acts linearly on a vector space $V$ and let $\rho$ denote a character $G$. 
Then for $d \gg 1$, the projective GIT quotient morphism $q: X^{ss}(\mathcal{O}(d)_{\rho}) \rightarrow X \gitq_{\hspace{-2pt} \mathcal{O}(d)_{\rho}} G$ restricts to a quotient 
\[ V^{ss}(\rho) = V \cap X^{ss}(\mathcal{O}(d)_{\rho}) \rightarrow q(V \cap X^{ss}(\mathcal{O}(d)_{\rho})) \cong  V \gitq_{\hspace{-2pt} \rho \hspace{2pt}} G   \]
with boundary $\mathbb{P}(V) \gitq_{\hspace{-2pt} \mathcal{O}(d)_{\rho}} G$.  If $k[V]^G=k$, then this boundary is empty and $V \gitq_{\hspace{-2pt} \rho}  \hspace{2pt} G$ is projective. 
\end{corollary}

\begin{remark}
In our version of Halic's results, we use the linearisation $\mathcal{O}(d)_\rho$ with $d$ sufficiently large -- how large $d$ needs to be can be determined from the weights for the action of a maximal torus on $V$. Halic works invariant-theoretically and instead considers a linearisation $\mathcal{O}(d)_{\rho^c}$, where first one replaces $\rho$ by $\rho^c$ for $c \gg 1$ so the semi-invariants for $\rho^c$ are generated by weight one semi-invariants, and then one takes $d$ larger than the degrees of a fixed set of generating semi-invariants.
\end{remark}

\section{Projective GIT for non-reductive groups} \label{sec:projnrgit}

This section summarises projective Non-Reductive GIT (NRGIT) as described in \cite{Berczi2023}.

\subsection{Graded unipotent groups} \label{subsec:gradedunip}

In this section we explain NRGIT for graded unipotent groups.

\begin{definition} \thlabel{gradedunip}
    A semi-direct product $\hU : = U \rtimes \GG_m$ of a unipotent group $U$ with the multiplicative group $\GG_m$ is \emph{positively graded} if the induced $\GG_m$-action on $\Lie U$ has strictly positive weights. In this case we call $\hU$ a \emph{graded unipotent group} and we call $\GG_m$ the \emph{grading} $\GG_m$.
\end{definition}

We fix a graded unipotent group $\hU$ throughout this section. If $\hU$ acting on a projective variety $X$ with respect to an ample linearisation, the $\hU$-invariants may not be finitely generated, so we cannot directly construct a rational map whose domain of definition should be the quotienting semistable locus. Nevertheless, we can define the invariant-theoretic and Hilbert-Mumford semistable loci.

\begin{definition}[Semistable loci for $\hU$-actions on projective varieties] \thlabel{hmdef}
For a $\hU$-action on a projective variety $X$ with linearisation $\Lc$, \begin{enumerate}[(i)]
     \item the  \emph{invariant-theoretic semistable locus} is $$\Iss{X}{\hU}(\Lc) : = \{x \in X \ | \ \exists \ \sigma \in H^0(X,\mathcal{L}^{\otimes k})^{\hU} \text{ for some $k > 0$ such that } \sigma(x) \neq 0\};$$ 
    \item the \emph{Hilbert-Mumford semistable locus} is $$ \HMss{X}{\hU}(\Lc) := \bigcap_{u \in U} u X^{\GG_m-ss}(\Lc).$$   
    \end{enumerate} 
\end{definition}

We define the following subvarieties of $X$ that appear in the NRGIT quotient by $\hU$. 

\begin{notation}[The varieties $\xmino$ and $\zmin$] \thlabel{xominandzmin}
    Suppose that $\hU$ acts on a projective variety $X$ with very ample linearisation $\Lc$, so that the action is induced via a representation $\hU \to \GL(V)$ where $X \subseteq \PP(V)$ and $V = H^0(X,L)^{\ast}$. 
    Let $\vmin$ denote the minimal weight space in $V$ for the restricted representation $\GG_m \to \GL(V)$, and define $$ \zmin := X \cap \PP(\vmin) \text{ and } \xmino := \{ x \in X \ | \ \lim_{t \to 0} t \cdot x \in \zmin\}.$$ 
    Let $p: \xmino \to \zmin$ denote the natural retraction map given by $x \mapsto \lim_{t \to 0} t \cdot x$.
\end{notation}

Assuming that $X$ is embedded in $\PP(V)$ via a very ample line bundle ensures that $\zmin$ and $\xmino$ are non-empty, as then $X$ cannot be contained in a proper linear subspace. In general if $\hU$ acts on $X \subseteq \PP(V)$ via a representation $\hU \to \GL(V)$, then $\zmin \cap \PP(\vmin)$ may be empty. 

\begin{remark} \thlabel{independent} 
    A priori $\xmino$ and $\zmin$ depend on the linearisation, but a posteriori they depend only on the $\GG_m$-action: the subvariety $\xmino$ is the open Bialynicki-Birula stratum for the $\GG_m$-action on $X$, and $\zmin$ its image under the map $x \mapsto \lim_{t \to 0} t \cdot x$.     
\end{remark}

While reductive GIT quotients can use any (ample) linearisation, in NRGIT one must  modify the linearisation to make it `well-adapted' and force it to lie in a certain VGIT chamber for $\GG_m$.

\begin{definition}[Adapted and well-adapted linearisations and characters]
Suppose that a multiplicative group $\GG_m$ acts linearly on a projective variety $X \subseteq \PP(V)$ via a representation $\GG_m \to \GL(V)$. Let $\omegamin < \omega_1 < \cdots$ denote the weights for this action. The linearisation is \emph{adapted} if $$\omegamin < 0 < \omega_1.$$  The linearisation is \emph{well-adapted} if $- \epsilon < \omegamin < 0 < \omega_1$ for $0 < \epsilon \ll 1.$ A $\hU$-linearisation is \emph{adapted} (resp.\ \emph{well-adapted}) if its restriction to the grading $\GG_m$ is adapted (resp.\ well-adapted). 
\end{definition}

How small $\epsilon$ needs to be for a linearisation to be deemed well-adapted depends on what property one wants to establish using well-adaptedness. To say that it is well-adapted means that there exists an $\epsilon$ such that the desired property holds. Any linear $\GG_m$-action on a projective variety $X$ can be made to be adapted or well-adapted by twisting the linearisation by a suitable character of $\GG_m$. The same is true for any linear $\hU$-action on a projective variety $X$, by choosing a suitable character of $\GG_m$ and extending it to a character of $\hU$.

The following $\hU$-quotient is constructed in stages (under suitable assumptions): first one obtains a geometric $U$-quotient $q_U : \xmino \rightarrow \xmino/U$ using the grading $\GG_m$ to locally construct slices, and then one takes a $\GG_m$-quotient of a projective completion $\overline{\xmino/U}$. The \emph{quotienting semistable locus} is the preimage under $q_U$ of the $\GG_m$-semistable locus in this projective completion.

\begin{theorem}[$\hU$-Theorem] \thlabel{uhatthm}
    Suppose that $\hU$ acts on a projective variety $X$ with well-adapted linearisation $\Lc$ such that \begin{equation} 
    \Stab_U(z) = \{e\} \text{ for all $z \in \zmin$}. \tag*{$[U]_{\mathrm{proj}}$} \label{Ucondproj} 
    \end{equation}  
    Then $\Qss{X}{\hU}(\Lc) = \xmino \setminus U \zmin$ and there is a projective geometric $\hU$-quotient $$ \Qss{X}{\hU}(\Lc)  \to X \gitq_{\hspace{-2pt} \Lc \hspace{2pt}}  \hU :=   \Qss{X}{\hU}(\Lc) / \hU.  $$ Moreover, if $U \zmin \neq \xmino$ so that $\Qss{X}{\hU}(\Lc)$ is non-empty, then for $n$ sufficiently large the algebra of invariants $\bigoplus_{l \geq 0} H^0(X, \Lc^{\otimes nl})^{\hU}$ is finitely generated and the quotient map is that induced by the inclusion of invariants, so that $$ X \gitq_{\hspace{-2pt} \Lc \hspace{2pt}} \hU \cong \Proj \bigoplus_{l \geq 0} H^0(X, \Lc^{\otimes nl})^{\hU}.$$  
\end{theorem}

\begin{remark}[When \ref{Ucondproj} is not satisfied] \thlabel{uzmin=xmino} 
    If \ref{Ucondproj} is not satisfied for the $\hU$-action on $X$, a projective quotient can still be constructed by performing a sequence of $\hU$-equivariant blow-ups of $X$ to obtain a new variety $\hX$ with a $\hU$-action such that \ref{Ucondproj} holds. This construction is the non-reductive analogue of the partial desingularisation construction in classical GIT \cite{Kirwan1985}, but involves many more subtleties, see \cite{Hoskins2021,Qiao2022}. In this paper we have chosen to give an exposition of NRGIT under the assumption that unipotent stabilisers are trivial on $\zmin$, as this is the core of the theory and already involves many technicalities. Nevertheless, in Section \ref{subsubsec:setup} we have to blow-up $X=\PP(V \oplus k)$ at the origin, which is a special case of the blow-up construction.  
\end{remark}

Under the conditions of \thref{uhatthm}, the invariant-theoretic and Hilbert-Mumford semistable loci coincide with the quotienting semistable locus, as in the reductive case.

\begin{corollary}[Equality of semistable loci] \thlabel{equalityofssloci}
    Suppose that $\hU$ acts on a projective variety $X$ with well-adapted linearisation $\Lc$. Then if $U \zmin \neq \xmino$ and \ref{Ucondproj} holds, there are equalities  $$ \Qss{X}{\hU}(\Lc) = \HMss{X}{\hU}(\Lc) = \Iss{X}{\hU}(\Lc) = \xmino \setminus U \zmin.$$ 
\end{corollary}

\begin{proof}
    The Hilbert-Mumford criterion for $\GG_m$ gives $\justss{X}{\GG_m}(\Lc) = \xmino \setminus \zmin$ as $\Lc$ is adapted. By  \thref{hmdef,uhatthm}, $\HMss{X}{\hU}(\Lc) = \xmino \setminus U \zmin = \Qss{X}{\hU}$.  Finally, $\Qss{X}{\hU}(\Lc) = \Iss{X}{\hU}(\Lc)$ follows from the second part of \thref{uhatthm} and the observation that $\Iss{X}{\hU}(\Lc)$ is the domain of definition of the rational map $X \dashrightarrow X \gitq \hU$ induced by the inclusion of the invariants. 
\end{proof}

In light of \thref{equalityofssloci} and \thref{independent}, we introduce the following notation for $\xmino \setminus U \zmin$. 

\begin{definition}[$\hU$-semistable locus] \thlabel{uhatssdef}
    The \emph{semistable locus} for a $\hU$-action on a projective variety $X$ with well-adapted linearisation for which  \ref{Ucondproj} holds is $X^{\hU-ss}(\Lc) : = \xmino \setminus U \zmin$. 
\end{definition}

\subsection{Restricting $\hU$-quotients to closed subvarieties} \label{subsec:restrictiontoclosed}

Good quotients for reductive group actions restrict to good quotients on closed invariant subvarieties, by virtue of invariants extending for reductive groups.  
By contrast, invariants do not necessarily extend for non-reductive group actions (see \thref{invtsnotextending} for a counter-example). Nevertheless, the $\hU$-quotients from \thref{uhatthm} do have the property that they restrict to good quotients on closed invariant subvarieties, a crucial property for enabling us to pass from vector spaces to affine varieties in Section \ref{subsec:fromvstoaffine}.

\begin{proposition}[Restricting $\hU$-quotients to closed subvarieties] \thlabel{proprestriction}
    Suppose that a graded unipotent group $\hU$ acts on a projective variety $X$ with well-adapted linearisation $\Lc$ such that \ref{Ucondproj} holds. Then, for any closed $\hU$-invariant subvariety $Y \subseteq X$, the good $\hU$-quotient $q_X :\Qss{X}{\hU} \to X \gitq_{\hspace{-2pt}  \Lc \hspace{2pt}} \hU$ restricts to a good $\hU$-quotient $Y \cap \Qss{X}{\hU} \to q_X(Y)$ .   
\end{proposition}

\begin{proof} Recall from the discussion proceeding \thref{uhatthm} that the quotient $\Qss{X}{\hU}(\Lc) \to X \gitq \hU$ is constructed in stages, first by taking a quotient of $\xmino$ by $U$ and then by taking a quotient by the residual $\GG_m$. It suffices to show that the first quotient restricts to a geometric $U$-quotient on $Y \cap \xmino$, because the second quotient is reductive. The statement about the $U$-quotient is local, so we may assume that $X^0_{\min} = \Spec A$ and that the $U$-quotient is trivial \cite[Prop 4.25]{Berczi2023}. With these assumptions, we need only observe that the trivial $U$-bundle $q_U$ restricts to a trivial $U$-bundle on the closed subvariety $q_U(Y \cap \xmino)$. Alternatively, one can prove this by observing that the $U$-quotient is constructed as a sequence of $\GG_a$-quotients,  each constructed using slices, where a $\GG_a$-action on $X^0_{\min}=\Spec A$ is the same thing as a locally nilpotent derivation $D:A\rar A$  and a slice is an element $s \in A$ with $D(s)=1$ \cite{Freudenburg2017}. From this perspective it is clear that a slice restricts to a slice on a closed invariant subvariety.   \end{proof}

\begin{remark}[Exactness of $U$-invariants]
   The fact that slices restrict to slices on closed subvarieties implies that taking $U$-invariants for a unipotent group $U$ is exact in this setting. For a $\GG_a$-action on $\Spec A$ with derivation $D$ and an ideal $I \subseteq A$ with $D(I) \subseteq I$ corresponding to a $\GG_a$-invariant closed subscheme $\Spec A/I$, we have $A/I = (A/I)^{\GG_a}[s]$ if $s$ is a slice for $D : A \rightarrow A$. This can be iterated to obtain surjectivity of the natural map $A^{U} \rar (A/I)^{U}$ for any $U$.     
\end{remark}

\begin{remark}[When do $\hU$-quotients commute with closed embeddings?]
It is straightforward to show that in the setting of \thref{proprestriction}, the equality $Y \cap \Qss{X}{\hU}(\Lc) = \Qss{Y}{\hU}(\Lc|_{Y})$  holds if and only if $Y \cap \xmino \neq \emptyset$ and $U Z(Y)_{\min} \neq Y^0_{\min}.$ Otherwise, $Y \cap \Qss{X}{\hU}(\Lc)$ is empty. 
\end{remark}

\subsection{Unipotent groups} \label{subsec:unip}
\thref{uhatthm} can be used to construct quotients by unipotent groups rather than graded unipotent groups, under certain assumptions (see \cite[$\S$9]{Berczi2016}).

\begin{corollary} \thlabel{projgitforU}
    Suppose that a projective variety $X$ admits a well-adapted linearised action of a unipotent group $U$ which extends to a graded unipotent group $\hU = U \rtimes \GG_m$.   Then if \ref{Ucondproj} holds, there exists a quasi-projective geometric $U$-quotient $$ \xmino \to \xmino / U =: X \gitq U$$
    with projective completion $( X \times \PP^1) \gitq \hU$. This projective completion  is obtained by letting $\hU$ act linearly on $\PP^1 $ via the $\hU$-representation on $\mathbb{A}^2$ defined by $(u,t) \cdot (x,y) = (tx,y)$, and taking a tensor product of the linearisation on $X$ with a large tensor power of the linearisation on $\PP^1$. 
\end{corollary}

\begin{proof}
    If \ref{Ucondproj} holds for the action of $\hU$ on $X$, then it also holds for the action of $\hU$ on $X \times \PP^1$. Therefore by \thref{uhatthm} there exists a projective geometric $\hU$-quotient $ (X \times \PP^1) \gitq \hU$ of $$\justss{(X \times \PP^1)}{\hU} =  (X \times \PP^1)^0_{\min} \setminus U Z(X \times \PP^1)_{\min} = ( \xmino \times \mathbb{A}^1 \setminus \{0\}) \sqcup (\justss{X}{\hU} \times \{0\}),$$
    where $\mathbb{A}^{1} \setminus \{ 0 \} = \{ [x: 1 ] \in \mathbb{P}^1 : x \neq 0 \}$. Thus $ (X \times \PP^1) \gitq \hU  =  (\xmino \times \mathbb{A}^1 \setminus \{0\})/ \hU \sqcup (\justss{X}{\hU} \times \{0\}) / \hU$  contains a geometric $\hU$-quotient $ (\xmino \times \mathbb{A}^1 \setminus \{0\})/ \hU$ as an open subvariety, since geometric quotients restrict to geometric quotients on preimages of open subvarieties \cite[Prop 3.10]{Newstead1978}. 
    
    We claim that the composition $\phi: \xmino \to \xmino \times \mathbb{A}^1 \setminus \{0\} \to (\xmino \times \mathbb{A}^1 \setminus \{0\}) /\hU$ is a geometric $U$-quotient, where the first map is the inclusion $x \mapsto (x,1)$. As the fibres of $\phi$ are $U$-orbits and $\phi$ satisfies all the topological properties required to be a good quotient, it suffices to show that functions on  the target are $U$-invariant functions on $\xmino$. Since the linearisation is well-adapted, the quotient $(X \times \PP^1) \gitq \hU$ is locally given by the spectrum of invariants, and as the question is local we may assume that $\phi$ is of the form $\Spec A \to \Spec A[x,x^{-1}] \to \Spec A[x,x^{-1}]^{\hU}$. Then $A[x,x^{-1}]^{\hU} \cong A^U$, as $U$ acts trivially on $x$ and $\GG_m$ acts by scaling on $x$. This shows $\phi$ is a geometric $U$-quotient. \end{proof}

 We next consider the more general case of graded linear algebraic group actions and split the exposition into two cases, depending on whether the group is internally or externally graded.  

\subsection{Internally graded linear algebraic groups} \label{subsec:internallygraded}

Internally graded linear algebraic groups are a natural generalisation of graded unipotent groups.

\begin{definition}[Internally graded linear algebraic groups] \thlabel{def internally graded unipotent radical}
  A linear algebraic group $H = U \rtimes R$ with unipotent radical $U$ is \emph{internally graded} if it contains a (\emph{grading}) 1PS $\lambda: \GG_m \to Z(R)$ such that the adjoint action of $\lambda(\GG_m)$ on $\Lie U$ has strictly positive weights.   
\end{definition}

Throughout this section we fix an internally graded group $H = U \rtimes R$ with grading 1PS $\lambda$. We also choose $R' \subseteq R$ such that $R\cong R' \times \lambda(\Gm)$ up to a finite group. This splitting allows us to lift characters of $\lambda(\Gm)$ to characters of $R$ and $H$, by pulling back along the quotient maps $H \rar R \rar \lambda(\Gm)$.

Similarly to the case for $\hU$, we can directly define the invariant-theoretic and Hilbert-Mumford semistable loci for $H$ as in the reductive case. The quotienting semistable locus will be defined later, as the invariants may not be finitely generated and so the quotient is constructed differently.

\begin{definition} \thlabel{sslociforH}
        For an $H$-action on a projective variety $X$ with linearisation $\Lc$, 
    \begin{enumerate}[(i)]
    \item the \emph{invariant-theoretic semistable locus} is $$ \Iss{X}{H}(\Lc) : =  \{x \in X \ | \ \exists \  \sigma \in H^0(X,\Lc^{\otimes k})^H \text{ for some $k > 0$ such that $\sigma(x) \neq 0$} \};$$ 
    \item the \emph{Hilbert-Mumford semistable locus} is $$ \HMss{X}{H}(\Lc) := \bigcap_{h \in H} h \justss{X}{T}(\Lc)$$ where $T \subseteq R$ is any maximal torus.  
    \end{enumerate} 
    \end{definition}

\thref{projHthm} below extends \thref{uhatthm} from graded unipotent groups to internally graded groups. It produces an  explicit open subset of $X$ admitting a projective $H$-quotient, under certain assumptions on the linearisation and on stabiliser groups; see \cite[Thm 4.28]{Berczi2023} under the stronger assumption \ref{Ucondproj}, and \cite[Thm 2.29]{Hoskins2021} and \cite[Rk 2.7]{Berczi2016} under the weaker assumption \ref{Usscondproj}. This $H$-quotient is constructed in two stages: first one obtains a quasi-projective geometric $\hU$-quotient $q_{\hU}$ of the locus $\{x \in \xmino \setminus U \zmin \ | \ \Stab_U(p(x)) =\{e\}\}$ in $X$ (via a generalisation of the proof of \thref{uhatthm}), then one takes an $R/\lambda(\GG_m)$-quotient of a projective completion of the $\hU$-quotient. The \emph{quotienting semistable locus} is the preimage under $q_{\hU}$ of the $R/\lambda(\GG_m)$-semistable locus in this projective completion.  

\begin{notation}   Given an $H$-action on a projective variety $X$ with linearisation $\Lc$, we let $\Lc'$ denote the linearisation for the $R'$ action on $\zmin$ obtained by restricting linearisation $\Lc$. 
\end{notation}

\begin{theorem} \thlabel{projHthm}     Suppose that an internally graded group $H$ acts on a projective variety $X$ with well-adapted linearisation $\Lc$ such that the conditions \begin{equation} \Stab_U(z) = \{e\} \text{ for all } z \in \zmin^{R'-ss}(\Lc')  \label{Usscondproj} \tag*{$[U;{ss}]_{\mathrm{proj}}$}\end{equation}  and \begin{equation} 
    \zmin^{R'-ss}(\Lc') = \zmin^{R'-s}(\Lc'). \label{Rcondproj} \tag*{$[R]_{\mathrm{proj}}$}
    \end{equation} hold. 
    Then $\Qss{X}{H}(\Lc) = p^{-1}(\zmin^{R'-ss}(\Lc')) \setminus U \zmin$ and there exists a projective geometric quotient $$ \Qss{X}{H}(\Lc)  \to  X \gitq_{\hspace{-2pt} \Lc \hspace{2pt}} H $$ for the action of $H$ on $ \Qss{X}{H} (\Lc).$ Moreover, if $\Qss{X}{H}(\Lc)$ is non-empty, then for $n$ sufficiently large the algebra of invariants $\bigoplus_{k \geq 0} H^0(X, \Lc^{\otimes nk})^H$ is finitely generated and the geometric quotient above is the map induced by the inclusion of invariants, so that  $$ X \gitq_{\hspace{-2pt} \Lc \hspace{2pt}} H \cong \Proj \bigoplus_{k \geq 0} H^0(X,\Lc^{\otimes kn})^H.$$ 
\end{theorem}

\begin{remark}[Restricting $H$-quotients to closed subvarieties] \thlabel{restrictingHquotients}
The fact that $\hU$-quotients restrict to $\hU$-quotients on closed subvarieties (\thref{proprestriction}) implies that the same is true for $H$-quotients. Indeed an $H$-quotient is obtained by first quotienting by $\hU$, then by $\overline{R} = R/ \lambda(\GG_m)$ which is reductive and hence also has the desired restriction property. 
\end{remark}

Under the conditions of \thref{projHthm}, the invariant-theoretic and Hilbert-Mumford semistable loci coincide with the quotienting semistable locus (see \cite[Thm 2.28]{Hoskins2021}).

\begin{corollary}[Equality of semistable loci] \thlabel{thm:hmcritforH} 
    Suppose that $H$ acts on a projective variety $X$ with well-adapted linearisation $\Lc$ such that \ref{Usscondproj} and \ref{Rcondproj} hold. Then we have: $$ \Qss{X}{H}(\Lc)  = \Iss{X}{H}(\Lc) = \HMss{X}{H}(\Lc) = p^{-1}(\zmin^{R'-ss}(\Lc')) \setminus U \zmin.$$
\end{corollary}

In light of \thref{thm:hmcritforH}, we introduce the following notation and terminology.  

 \begin{definition}[$H$-semistable locus]
     The \emph{semistable locus} for an action of $H$ on a projective variety $X$ with well-adapted linearisation $\Lc$ which satisfies \ref{Usscondproj} and \ref{Rcondproj} is $$X^{H-ss}(\Lc) : = p^{-1}(\zmin^{R'-ss}(\Lc')) \setminus U \zmin.$$
 \end{definition}

\subsection{Externally graded linear algebraic groups} \label{subsec:externallygraded}
If a linear algebraic group $H$ is not internally graded, but there is a multiplicative group $\GG_m$ such that $\widehat{H} = H \rtimes \GG_m$ is internally graded by $\GG_m$ and the action of $H$ extends to $\widehat{H}$, then the results of Section \ref{subsec:internallygraded} can be applied to construct a quotient for the action of $H$ on $X$ under suitable assumptions.

\begin{definition}[Externally graded linear algebraic groups]
    A linear algebraic group $H=U \rtimes R$ is \emph{externally graded by a multiplicative group $\GG_m$} if $\widehat{H}: =H \rtimes \GG_m$ is internally graded by this $\GG_m$, i.e.\ if the adjoint action of $\GG_m$ on $\Lie U$ has strictly positive weights. In this case $\GG_m$ is called the \emph{externally grading $\GG_m$}.  
\end{definition}

If $H$ is externally graded by $\GG_m$, and the linearised action of $H$ on a projective variety $X$ satisfying \ref{Usscondproj} and \ref{Rcondproj} extends to a linearised action of $\widehat{H} = H \rtimes \GG_m$ on $X$, then an $H$-quotient can be constructed from an $\widehat{H}$-quotient of $X \times \PP^1$, where the linearisation $\widetilde{\Lc}$ for the latter action is as given in \thref{projgitforU}, with $R$ acting trivially on the line bundle on $\PP^1$. That is, \thref{projHthm} can be used to obtain a quotient for the $\widehat{H}$-action on $\justss{(X \times \PP^1)}{\widehat{H}}(\widetilde{\Lc})$, which can then be restricted to $ \justss{X}{H}(\Lc) : = X \cap \justss{(X \times \PP^1)}{\widehat{H}}(\widetilde{\Lc})$, where the inclusion of $X$ in $X \times \PP^1$ is given by $x \mapsto (x,[1:1]).$

\begin{corollary}[Quotients by externally graded groups] \thlabel{externallygraded} 
    Suppose that a linear algebraic group $H$ acts on a projective variety $X$ with linearisation $\Lc$, extending to a linearised action of $\widehat{H} : = H \rtimes \GG_m$ on $X$, where $H$ is externally graded by $\GG_m$. If this linearisation is well-adapted and satisfies \ref{Usscondproj} and \ref{Rcondproj}, then $\justss{X}{H}(\Lc) = p^{-1}(Z_{\operatorname{min}}^{R'-ss}(\Lc'))$ and there exists a quasi-projective geometric $H$-quotient $$ \Qss{X}{H}(\Lc) \to \Qss{X}{H}(\Lc) / H,$$ with projective completion $ (X \times \PP^1) \gitq_{\hspace{-2pt} \widetilde{\Lc} \hspace{2pt}} \widehat{H}$. 
\end{corollary}

\begin{proof}
    The conditions on the linearised $\widehat{H}$-action on $X$ ensure that the analogous conditions hold for the linearised $\widehat{H}$-action on $X \times \PP^1$. Hence by \thref{projHthm} there exists a projective geometric $\widehat{H}$-quotient $q$ of $ \Qss{(X \times \PP^1)}{\widehat{H}} = p^{-1} ( Z(X \times \PP^1)_{\min}^{R'-ss} )  \setminus U Z(X \times \PP^1)_{\min}. $ The intersection of $\Qss{(X \times \PP^1)}{\widehat{H}}$ with $X \times \mathbb{A}^1 \setminus \{0\}$ is open in $\Qss{(X \times \PP^1)}{\widehat{H}}$ and given by $ p^{-1} ( \zmin^{R'-ss}(\Lc') ) \times \mathbb{A}^1 \setminus \{0\}.$ The restriction of $q$ to this open subvariety is a quasi-projective geometric $\widehat{H}$-quotient, open inside $(X \times \PP^1) \gitq \widehat{H}$.  By the same argument as in the proof of \thref{projgitforU}, restricting this quotient to $X$ gives a geometric $H$-quotient of $X \cap ( p^{-1} ( \zmin^{R'-ss}(\Lc') ) \times \mathbb{A}^1 \setminus \{0\})  = p^{-1} ( \zmin^{R'-ss}(\Lc') ).$
\end{proof}

\section{Affine GIT from projective GIT for 
graded unipotent groups} \label{sec:affineuhat}

In this section we establish an affine analogue of the results of Sections \ref{subsec:gradedunip} and \ref{subsec:unip}. Given a linear action of a graded unipotent group $\hU : = U \rtimes \GG_m$ on a vector space $V$, we construct a quotient for the $\hU$-action on an open subset of $V$ using \thref{uhatthm} applied to an induced linearised $\hU$-action on $X: = \PP(V \oplus k)$ or on a blow-up of $X$.

\subsection{Set-up} \label{subsubsec:setup}

Let $\omega_{\min}$ denote the minimal weight for the action of $\GG_m$ on $V$ and $V_{\min}$ the associated weight space. Let $p_{\min}: V \to V_{\min}$ denote the natural projection and define $$V^0_{\min} := \{v \in V \ | \ p_{\min} (v) \neq 0\}.$$ 

\begin{remark}[Comparing $V_{\operatorname{min}}$ and $V^0_{\operatorname{min}}$ to $\zmin$ and $\xmino$] While $V_{\operatorname{min}}$ and $V^0_{\operatorname{min}}$ should be thought of as affine analogues of $\zmin$ and $\xmino$, the maps $p_{\min}: V^0_{\operatorname{min}} \to V_{\operatorname{min}}$ and $p: \xmino \to \zmin$ do not arise in the same way: the former is the restriction of the projection $ V \to V_{\operatorname{min}}  $ while the latter takes the limit at $0$ under the grading 1PS. For $v \in V^0_{\operatorname{min}}$, we have $p_{\min}(v) = \lim_{t \to 0} t \cdot v$ if and only if $\omega_{\operatorname{min}} = 0$. If $\omegamin < 0 $, then $\lim_{t \to 0} t \cdot v$ does not exist, while if $\omegamin > 0$ then $\lim_{t \to 0} t \cdot v= 0.$ Nevertheless, we do have that $[p_{\min}(v)] = p([v])$ in $\PP(V)$ for $v \in V^0_{\min}$. 
\end{remark}

We wish to construct a quotient for the action of $\hU$ on $V$ from a quotient for the linearised action of $\hU$ on $X: = \PP(V \oplus k)$ induced  by the representation $\hU \to V$, where $\hU$ acts trivially on $k$. By \thref{uhatthm}, a quotient for the latter exists if the action of $\hU$ on $X$ satisfies \ref{Ucondproj}. Thus our aim is to identify a condition for the $\hU$-action on $V$ which ensures that \ref{Ucondproj} is satisfied for the action of $\hU$ on $X$. As we will see in \thref{relatingUconds},  such a condition exists as long as we are willing to work with the following blow-up of $X$ in the case where $\omega_{\min} \geq 0$. 

\begin{notation}[The variety $\hX$] \label{notation def of Xhat}
Let $\psi: \hX \to X$ be the blow-up of $X$ at $[0:1]$, which we identify as a subvariety of $X \times \PP(V)$ consisting of $([v : w],[v']) $ such that $[v'] = [v]$ if $v \neq 0$. We let $E \cong \PP(V)$ denote the exceptional divisor. Let $\widehat{L}$ be the linearisation on $\hX$ given by tensoring the pullback of a large power of the linearisation $L$ with the natural linearisation $\mathcal{O}_{\PP(V)}(1)$ on $E \cong \PP(V)$.
\end{notation}  

A special case of this blow-up is considered in \cite{Berczi2024}, where $V = \operatorname{Hom}(\CC^k, \CC^n)$ is the fibre over $y \in Y$ of the bundle $J_k(Y) \to Y$ of $k$-jets of germs of parametrised curves in a complex projective variety $Y$ of dimension $n$; in this example $\omega_{\min}$ is strictly positive.

\begin{proposition} \thlabel{relatingUconds}
 If the linear action of $\hU$ on $V$ satisfies \begin{equation} \operatorname{Stab}_U(v) = \{e\} \text{ for all $v \in V_{\operatorname{min}} \setminus \{0\}, $} \tag*{$[U]_{\mathrm{aff}}$} \label{Ucondaffine} \end{equation} then \ref{Ucondproj} holds for the linearised action of $\hU$ on $X$ if and only if $\omega_{\min} < 0$. If $\omega_{\min} \geq 0$, then \ref{Ucondproj} holds for the linearised action of $\hU$ on $\hX$. 
\end{proposition}

\begin{remark} \thlabel{independence}
    While \thref{relatingUconds} considers specific linearisations for the $\hU$-action on $X$ and on $\hX$, the result holds true for any linearisation of the $\hU$-actions on $X$ and on $\hX$,  since the quotienting semistable locus is independent of the linearisation by \thref{equalityofssloci}. 
    
\end{remark}

To prove \thref{relatingUconds}, we first give explicit descriptions of $\zmin$, $\xmino$ and their analogues for $\hX$, denoted $\widehat{Z}_{\operatorname{min}}$ and $\hX^0_{\operatorname{min}}$ respectively. We use the decomposition $X = V \sqcup \PP(V)$ where $V \subseteq X$ is given by $v \mapsto [v:1]$ and $\PP(V) \subseteq X$ is given by $ [v] \mapsto [v:0].$ As we will only consider the blow-up $\hX$ when $\omega_{\min} \geq 0$, we only include a description of $\widehat{Z}_{\operatorname{min}}$ and $\hX^0_{\operatorname{min}}$ in that case.

\begin{lemma}[Describing $\zmin$ and $\xmino$ and their analogues for $\hX$] \thlabel{zminandxomin}
Given the linear action of $\hU$ on $V$, the subvarieties $(V \oplus k)_{\min}$, $\xmino$ and $\zmin$ are as follows:   
\begin{center} \begin{tabular}{c c c c }
Case & $(V \oplus k)_{\min}$ &  $\zmin \subseteq V \sqcup \PP(V) $ & $\xmino \subseteq V \sqcup \PP(V)$  \\ 
\hline 
$\omega_{\operatorname{min}} < 0 $ & $V_{\min}\oplus 0$ & $ \emptyset \sqcup \PP(V_{\operatorname{min}})$ & $ V^0_{\operatorname{min}} \sqcup \PP(V^0_{\operatorname{min}})$    \\ 
$\omega_{\operatorname{min}} = 0 $ & $V_{\min}\oplus k$ & $V_{\operatorname{min}}  \sqcup \PP(V_{\operatorname{min}})$ & $ V \sqcup \PP(V^0_{\operatorname{min}}) $   \\
$\omega_{\operatorname{min}} > 0 $  & $0\oplus k$ & $ \{[0:1]\} \sqcup \emptyset$ & $V \sqcup \emptyset $ 
\end{tabular} 
\end{center}
The subvarieties $\widehat{Z}_{\operatorname{min}}$ and $\hX^0_{\operatorname{min}}$ are as follows: 
\begin{center}
\begin{tabular}{c  c  c}
     Case & $\widehat{Z}_{\operatorname{min}} \subseteq X \times \PP(V)$  & $\hX^0_{\operatorname{min}} \subseteq X \times \PP(V) $   \\ \hline 
    $\omega_{\operatorname{min}} = 0 $ &   $\{ ([v:w],[v']) \in \hX \  |  \ v' \in V_{\operatorname{min}}\}  $  & $\{ ([v:w],[v']) \in \hX  \ | \ v' \in V^0_{\min} \} $ \\ 
    $\omega_{\operatorname{min}}  > 0$ & $  \{[0:1]\} \times \PP(V_{\operatorname{min}}) $  &  $ \{([v:w],[v']) \in \hX \ | \ w \neq 0 \text{ and } v' \in V^0_{\operatorname{min}}\}$.
\end{tabular}
\end{center}
\end{lemma}

\begin{proof}
The descriptions in the first table follow directly from definitions, so we focus on those in the second table. We have that $\widehat{Z}_{\min} = \hX \cap Z(X \times \PP(V)) = \hX \cap (\zmin \times \PP(\vmin))$. If $\omegamin = 0$, then $\zmin = \vmin \sqcup \PP(\vmin)$ from the first table. Therefore $Z(X \times \PP(V))_{\min} = \{([v:w],[v']) \in X \times \PP(V) \ | \ v,v' \in \vmin\}$. If we intersect this with $\hX$, the condition that $v \in \vmin$ is redundant. If $([v:w],[v']) \in \hX$, then $\lim_{t \to 0} t \cdot ([v:w],[v'])$ lies in  $\widehat{Z}_{\min}$ if and only if either $p_{\min}(v) \neq 0$ (in which case $p_{\min}(v')$ is also non-zero), or $v = 0$ and $p_{\min}(v') \neq 0$. Equivalently, the limit lies in $\widehat{Z}_{\min}$ if and only if $p_{\min}(v') \neq 0$. If $\omegamin > 0$, then $\zmin = \{[0:1]\} \sqcup \emptyset$ from the first table. Therefore $Z(X \times \PP(V))_{\min} = \{[0:1]\} \times \PP(\vmin)$, which is contained in $\hX$. If $([v:w],[v']) \in \hX$, then $\lim_{t \to 0} t \cdot ([v:w],[v'])$ lies in $\widehat{Z}_{\min}$ if and only if $w \neq 0$ and $p_{\min}(v') \neq 0$.
\end{proof}

We can now prove \thref{relatingUconds}.

\begin{proof}[Proof of \thref{relatingUconds}]
 We start by showing that \ref{Ucondaffine} is equivalent to  \ref{Ucondproj} holding for the induced linear action of $\hU$ on $\PP(V)$.  As a first step, note that the projection $V \setminus \{0\} \rar \PP(V)$ is $U$-equivariant, and, because unipotent groups have no characters, induces an isomorphism on stabilisers $\Stab_U(v) \cong \Stab_U([v])$. Since $Z(\PP(V))_{\min} = \PP(\vmin)$, the desired equivalence follows.

Combining this equivalence with the case-by-case description of $\zmin$ given in the first table in \thref{zminandxomin}, we see that if \ref{Ucondaffine} holds for the $\hU$-action on $V$, then \ref{Ucondproj} holds for the $\hU$-action on $X$ if and only if $\omega_{\operatorname{min}} < 0$. Indeed, if $\omega_{\operatorname{min}} \geq 0$, then  the point $[0:1] \in \zmin$ is fixed by $U$. 

For $\omega_{\min} \geq 0$, it remains to show that \ref{Ucondproj} holds for the $\hU$-action on $\hX$ instead. If $\omega_{\operatorname{min}} = 0$ and $([v:w],[v']) \in \widehat{Z}_{\operatorname{min}}$, then $v' \in V_{\operatorname{min}}$ by the description of $\widehat{Z}_{\operatorname{min}}$ given in the table above. Therefore $\operatorname{Stab}_U(([v:w],[v'])) \subset \operatorname{Stab}_U([v']) = \Stab_U(v') = \{ e \}$, where the middle equality is the equivalence proved in the first paragraph and the final equality is by assumption. 
If $\omega_{\operatorname{min}} > 0$, the proof is exactly the same, as long as we replace the point $([v:w],[v])$ by $([0:1],[v])$.
\end{proof}

\subsection{Constructing the quotients} \label{subsec:constructionthequotients}

We construct quotients for the $\hU$-action on $V$ from quotients for the $\hU$-action on $X$ or on $\hX$, depending on the sign of $\omega_{\operatorname{min}}$. \thref{relatingUconds} shows that \ref{Ucondaffine} gives \ref{Ucondproj} for $X$ if $\omega_{\min} <0$ and for $\hX$ if $\omega_{\min} \geq 0$. Hence, provided we twist the linearisations for the $\hU$-actions on $X$ and on $\hX$ to make them well-adapted, we can make the following definition. 

\begin{definition}[$\hU$-quotienting semistable locus in $V$] \thlabel{QssforV} 
The \emph{quotienting semistable locus} for the linear action of $\hU$ on a vector space $V$ satisfying \ref{Ucondaffine} is 
$$ \Qss{V}{\hU} : = \begin{cases} V \cap \Qss{X}{\hU} & \text{ if $\omega_{\operatorname{min}} < 0$} \\
V \cap \psi(\Qss{\hX}{\hU} \setminus E) & \text{ if $\omega_{\operatorname{min}} \geq 0$.}
\end{cases} $$
\end{definition}

\begin{theorem}\thlabel{quotienting}
Suppose that a graded unipotent group $\hU = U \rtimes \GG_m$ acts linearly on a vector space $V$ such that \ref{Ucondaffine} holds. Then there is an explicit description of the quotienting semistable locus
\begin{equation} \label{explicitdescripeq}
    \Qss{V}{\hU} = \begin{cases} V^0_{\operatorname{min}} & \text{ if $\omega_{\operatorname{min}} \neq 0$}\\
V^0_{\operatorname{min}}  \setminus U V_{\operatorname{min}} & \text{ if $\omega_{\operatorname{min}} = 0$}.
\end{cases}
\end{equation} 
and a quasi-projective geometric $\hU$-quotient $$ \Qss{V}{\hU}  \to  V \gitq \hU : = \Qss{V}{\hU} / \hU$$  with an explicit projective completion given by $X \gitq \hU$ if $\omega_{\operatorname{min}} < 0$ and by $\hX \gitq \hU$ if  $\omega_{\operatorname{min}} \geq 0$.
\end{theorem} 

\begin{proof} 
If $\omega_{\operatorname{min}} < 0$, then $\Qss{V}{\hU}  = V \cap \Qss{X}{\hU}$ is an open subset of $\Qss{V}{\hU}$, and so the projective geometric quotient $X^{ss} \to X \gitq \hU = \Qss{X}{\hU} / \hU$ restricts to give a quasi-projective geometric quotient $ \Qss{V}{\hU}  \to \Qss{V}{\hU} / \hU,$ with $X \gitq \hU$ as an explicit projective completion. If $\omega_{\operatorname{min}} \geq 0$, then $\Qss{V}{\hU} = V \cap \psi(\Qss{\hX}{\hU}  \setminus E)$ is isomorphic via $\psi$ to an open subset of $\Qss{\hX}{\hU}$, since $\psi$ is an isomorphism away from $E$. Restricting the projective geometric quotient $\Qss{\hX}{\hU} \to \hX \gitq \hU = \Qss{\hX}{\hU} / \hU$ to this open subset then gives a quasi-projective geometric quotient $ \Qss{V}{\hU}  \to \Qss{V}{\hU} / \hU$ with $\hX \gitq \hU$ as an explicit projective completion. 

It remains to prove the explicit description \eqref{explicitdescripeq} of the semistable locus. For $\omega_{\operatorname{min}} < 0$, as $\justss{X}{\hU} = \xmino \setminus U \zmin$, we can use \thref{zminandxomin} to see that the intersection with $V$ is as claimed. For $\omega_{\operatorname{min}} \geq 0$, we use $\justss{\hX}{\hU} =  \widehat{X}^0_{\min} \setminus U \widehat{Z}_{\min}$ and \thref{zminandxomin} to identify $V \cap \psi(\Qss{\hX}{\hU} \setminus E)$.
\end{proof}

\begin{remark}[Comparison with the reductive case] \thlabel{comparison} 
In the reductive case, the quotient $V \gitq G$ is always affine, by construction. This can also be seen from the perspective of Section \ref{subsec:affinefromproj}: the image $V \gitq G$ of the restriction to $V$ of the quotient $X^{ss} \to X \gitq G$ coincides with the non-vanishing locus of the function picking out the coordinate $w$ at infinity, which is $G$-invariant. This argument no longer works in the non-reductive case. Although the quotient $V \gitq \hU$ inside $X \gitq \hU$ is similarly given by the non-vanishing of $w$, this function is not $\hU$-invariant, because of the twist to the linearisation required to make it well-adapted. 
 In fact, the explicit description \eqref{explicitdescripeq} of $\Qss{V}{\hU}$ implies that  the quotient $V \gitq \hU$ cannot always be affine. Indeed, if it were affine then the domain $\Qss{V}{\hU}$ of the quotient map would need to be affine, because good quotients are affine morphisms. But $V^0_{\operatorname{min}}$ and $V^0_{\operatorname{min}} \setminus U V_{\operatorname{min}}$ are not necessarily affine. Section \ref{subsec:leftmult} gives a simple example of a non-affine quotient.
\end{remark}

\subsection{Boundary of the quotient} \label{subsec:describing}

We now describe the boundary  of the projective completion of $V \gitq \hU$ given in \thref{quotienting}, in terms of information coming only from $V$.

\begin{proposition}[Description of the boundary] \thlabel{boundaryprop} The boundary of the projective completion of $V \gitq \hU$ given in \thref{quotienting} is as follows:  
\begin{enumerate}[(i)]
\item if $\omega_{\operatorname{min}} < 0 $  then there is an isomorphism $ ( X \gitq \hU)  \setminus (V \gitq \hU )\cong \PP(V) \gitq \hU$ induced by the inclusion $\PP(V) \subseteq X$ given by $[v] \mapsto [v:0]$;
\item if $\omega_{\operatorname{min}} = 0$ then there is an isomorphism $( \hX \gitq \hU)  \setminus (V \gitq \hU ) \cong \PP(V) \gitq \hU \sqcup \PP(V) \gitq \hU$ induced by two inclusions $\PP(V) \hookrightarrow \hX$: the first given by $[v] \mapsto ([v:0],[v])$ and the second by $[v] \mapsto ([0:1],[v])$;
\item if $\omega_{\operatorname{min}} > 0$ then there is an isomorphism $( \hX \gitq \hU)  \setminus (V \gitq \hU ) \cong \PP(V) \gitq \hU$ induced by the inclusion $\PP(V) \hookrightarrow \hX$ given by $[v] \mapsto ([0:1],[v])$. 
\end{enumerate}
\end{proposition}

\begin{proof} 
Suppose that $\omega_{\operatorname{min}} < 0$. Then $$\Qss{X}{\hU} \setminus \Qss{V}{\hU} = \{[v:0] \in X \ | \ v \in V^0_{\operatorname{min}} \setminus U V_{\operatorname{min}}\} \cong \PP(V)^0_{\operatorname{min}} \setminus U \PP(V_{\operatorname{min}}).$$    If \ref{Ucondaffine} holds for the $\hU$-action on $V$, then \ref{Ucondproj} holds for the induced linearised $\hU$-action on $\PP(V)$ by the proof of \thref{relatingUconds}. Thus by \thref{uhatthm} the locus $\Qss{\PP(V)}{\hU}=\PP(V)^0_{\operatorname{min}} \setminus U \PP(V_{\operatorname{min}})$
admits a projective geometric quotient. The restriction of the quotient map $\Qss{X}{\hU} \to X \gitq \hU$ to $\Qss{X}{\hU}  \setminus V^{\operatorname{Qss}}$ must therefore coincide with this quotient. 

Suppose that $\omega_{\operatorname{min}}  = 0$. Using the description \eqref{explicitdescripeq} of $\Qss{V}{\hU}$, we obtain that $$ \psi^{-1}(\Qss{V}{\hU}) = \{([v:1],[v]) \in \hX \ | \ v \in V^0_{\operatorname{min}} \setminus U V_{\operatorname{min}}  \}. $$ By \thref{zminandxomin}, we can describe  $\Qss{\hX}{\hU} = \hX_{\min}^0 \setminus U \widehat{Z}_{\min}$, and it follows that $$   \Qss{\hX}{\hU} \setminus \psi^{-1}(\Qss{V}{\hU})  =  \{([v:0],[v]) \in \hX \ | \ v \in V^0_{\operatorname{min}} \setminus U V_{\operatorname{min}} \}    \sqcup \{ ([0:1],[v]) \in \hX   \ | \ v \in V^0_{\operatorname{min}} \setminus U V_{\operatorname{min}}\}. $$ Both factors of the disjoint union are isomorphic to $\PP(V)^0_{\operatorname{min}} \setminus U \PP(V_{\operatorname{min}})$, which by the same argument as in the above paragraph coincides with $\Qss{\PP(V)}{\hU}$ and admits a projective geometric quotient $\PP(V) \gitq \hU$. Hence the restriction of $\Qss{\hX}{\hU} \to X \gitq \hU$ to $\Qss{\hX}{\hU} \setminus \psi^{-1}( \Qss{V}{\hU})$ coincides on each factor of the disjoint union with the projective $\hU$-quotient $\Qss{\PP(V)}{\hU} \to \PP(V) \gitq \hU$.

Finally for $\omega_{\operatorname{min}} > 0$, we have that $  \psi^{-1}(\Qss{V}{\hU}) = \{([v:1],[v]) \in \hX \ | \ v \in V^0_{\operatorname{min}} \setminus U V_{\operatorname{min}}  \} $ by the description \eqref{explicitdescripeq} of $\Qss{V}{\hU}$. It follows then from  \thref{zminandxomin} that $$ \Qss{\hX}{\hU} \setminus \psi^{-1}(\Qss{V}{\hU}) = \{ ([0:1],[v]) \ | \ v \in V^0_{\operatorname{min}} \setminus U V_{\operatorname{min}} \} \cong \PP(V)^0_{\operatorname{min}} \setminus U \PP(V_{\operatorname{min}})$$
and we conclude the proof in the same way as in the previous two cases.
\end{proof} 

 \thref{quotienting} and \thref{boundaryprop} together complete the proof of \thref{firstresult}. In Section \ref{subsec:fromvstoaffine} we extend this result to linear $\hU$-actions on affine varieties rather than just vector spaces.

\subsection{Quotients of vector spaces by unipotent group actions}

Just as the projective $\hU$-theorem (\thref{uhatthm}) can be used to construct quotients for unipotent rather than graded unipotent actions on projective varieties (see \thref{projgitforU}), the affine $\hU$-theorem can be used to construct quotients for unipotent rather than graded unipotent actions.   

\begin{corollary} \thlabel{affinebyunipotent}
    Suppose that a unipotent group $U$ acts linearly on a vector space $V$ and that this action extends to the linear action of a graded unipotent group $\hU = U \rtimes \GG_m$ such that \ref{Ucondaffine} holds. Then there exists a quasi-projective geometric quotient $$\Qss{V}{U} : = \vmino \to \Qss{V}{U} / U =: V \gitq U$$ with a projective completion given by $ (X \times \PP^1 ) \gitq \hU$ if $\omega_{\min} < 0$ and by $(\hX \times \PP^1) \gitq \hU$ if $\omega_{\min} \geq 0$, where the linearisations are as given in \thref{projgitforU}. 
\end{corollary}

\begin{proof}
\thref{relatingUconds} shows \ref{Ucondproj} holds for the  $\hU$-action on $X$ if $\omega_{\min} < 0$, and on $\hX$ if $\omega_{\min} \geq 0$. Thus we can apply \thref{projgitforU} to obtain a quasi-projective geometric $U$-quotient of $\xmino$ if $\omega_{\min} < 0$, and of $\widehat{X}^0_{\min}$ if $\omega_{\min}  \geq 0$ with projective completions $(X \times \PP^1) \gitq \hU$ and $(\hX \times \PP^1) \gitq \hU$ respectively.   If $\omega_{\min} < 0$, then restricticting to $V \hookrightarrow X$ via $v \mapsto [v:1]$, we obtain a quasi-projective geometric quotient for the $U$-action on $V \cap \xmino = \vmino$. If $\omega_{\min} \geq 0$, then we instead obtain a quasi-projective geometric quotient for the $U$-action on $V \cap \psi(\widehat{X}^0_{\min} \setminus E) = \vmino$.  
\end{proof}

\subsection{Twisted affine GIT for graded unipotent group actions}  \label{subsec:kingss} 
So far in Section \ref{sec:affineuhat} we have developed affine GIT for linear actions of graded unipotent groups on vector spaces. It is reasonable to ask whether a `twisted' version of this theory can be developed, analogously to the reductive case. We provide such a theory in this section. To this end, let $\rho$ denote a character of the graded unipotent group $\hU = U \rtimes \GG_m$, fixed for the remainder of this section.

Suppose that $\hU$ acts linearly on a vector space $V$. In the reductive case, the $\rho$-twisted quotienting semistable locus is defined as the locus of definition of the rational map from $V$ to the projective spectrum of the semi-invariants. This definition is not suitable in the non-reductive case, as the semi-invariants may not be finitely generated. However, we can instead use the equality $$ \Qss{V}{G}(\rho) = V \cap \Qss{X}{G}(\mathcal{O}(d)_{\rho})$$ from \thref{equalityonV}, valid for $d \gg 1$, to \emph{define} the $\rho$-twisted quotienting semistable locus in $V$ for the action of $\hU$. The issue with this approach in the non-reductive case is that $\Qss{X}{\hU}(\mathcal{O}(d)_{\rho})$ is only well-defined (and admits a quotient) if the linearisation is well-adapted and \ref{Ucondproj} holds. Regarding well-adaptedness, the linearisations $\Lc : = \mathcal{O}(d)_{\rho}$ and $\widehat{\Lc} : = \psi^{\ast} \mathcal{O}(d)_{\rho}^{\otimes N } \otimes \mathcal{O}(-E)$ with $N \gg 1$ for the $\hU$-actions on $X$ and on $\hX$ can be twisted by a character of $\hU$ to make them well-adapted. Call these well-adapted linearisation $\Lc'$ and $\widehat{\Lc'}$.  Regarding the condition \ref{Ucondproj}, by \thref{equalityofssloci} and \thref{independence} we know that if \ref{Ucondaffine} holds for the linear action of $\hU$ on $V$, then \ref{Ucondproj} holds for any linearised action of $\hU$ on $X$ if $\omega_{\min} < 0$, and on $\hX$ if $\omega_{\min } \geq 0$. Thus we may define $\Qss{V}{\hU}(\rho)$ analogously to $\Qss{V}{\hU}$ in \thref{QssforV}, using the new linearisations $\Lc'$ and $\widehat{\Lc'}$: $$ \Qss{V}{\hU}(\rho)  := \begin{cases}
V \cap \Qss{X}{\hU}(\Lc') & \text{ if $\omega_{\min} < 0$}\\
V \cap \psi( \Qss{\hX}{\hU}(\widehat{\Lc}') \setminus E) & \text{ if $\omega_{\min} \geq 0$}.
\end{cases}.$$   As we have seen in \thref{equalityofssloci}, the quotienting semistable loci in $X$ and $\hX$ are independent of the choice of well-adapted linearisation. Therefore $\Qss{V}{\hU}(\rho) = \Qss{V}{\hU}$ for any $\rho$ and so we do not get a new quotienting semistable locus by twisting by a character $\rho$ of $\hU$. This is not surprising, since achieving well-adapatedness of the linearisations on $X$ and $\hX$ requries twisting by a character of $\hU$, which cancels out the effect of any initial twist by $\rho$.

However, we can get new Hilbert-Mumford and invariant-theoretic semistable loci by twisting by $\rho$; these can be defined in analogy with the the reductive case.

\begin{definition} \thlabel{Qssforuhattwisted}
    Given the linear action of a graded unipotent group $\hU = U \rtimes \GG_m$ on a vector space $V$, and a character $\rho$ of $\hU$, 
    \begin{enumerate}[(i)]

\item the \emph{$\rho$-twisted Hilbert-Mumford semistable locus} is 
    $$  \HMss{V}{\hU}(\rho) := \bigcap_{u \in U} u \justss{V}{\GG_m}(\rho);$$

\item the \emph{$\rho$-twisted invariant-theoretic semistable locus} is $$\Iss{V}{G}(\rho) : = \{v \in V \ | \ \text{there exists } \sigma \in k[V]^{\hU}_{\rho^n} \text{ for some $n > 0$ such that $\sigma(v) \neq 0$} \},$$ where $k[V]^{\hU}_{\rho^n}$ denotes the ring of weight $n$ semi-invariants, i.e.\ the ring of functions $f \in k[V]$ such that $f ( h \cdot v) = \rho(h)^n f(v)$ for each $v \in V$ and $h \in \hU$. 
    \end{enumerate}
\end{definition}

\thref{equalityuhat} below provides conditions under which the $\rho$-twisted quotienting, invariant-theoretic and Hilbert-Mumford loci coincide.  

\begin{proposition} \thlabel{equalityuhat}
    Suppose that $\hU$ acts linearly on $V$ such that \ref{Ucondaffine} holds, and let $\rho$ denote a character of $\hU$, identified with an integer so it is given by $t \rightarrow t^\rho$ on $\GG_m$. Then there are equalities $$ \Qss{V}{\hU} = \HMss{V}{\hU}(\rho)  = \Iss{V}{\hU}(\rho) = \justss{V} {\GG_m}(\rho)$$ if and only if $\rho < 0$, $\omega_{\min}< 0$ and $\omega_{\operatorname{next}} \geq 0$, where $\omega_{\operatorname{next}}$ denotes the next minimal weight for the $\GG_m$ action on $V$ after $\omega_{\min}$. 
\end{proposition}

To prove \thref{equalityuhat}, we first give an explicit description of $\HMss{V}{\hU}(\rho)$.

\begin{notation}
    Let $V = \bigoplus_{i \in \ZZ} V_i$ denote the $\GG_m$-weight space decomposition for the linear action of $\GG_m$ on $V$. Let $V_{\geq 0} = \bigoplus_{i \geq 0} V_i$ and $V_{\leq 0} = \bigoplus_{i \leq 0} V_i.$
\end{notation}

\begin{lemma}[Description of the Hilbert-Mumford $\rho$-semistable locus] \thlabel{describingHMrhoss}
    Suppose that a graded unipotent group $\hU$ acts linearly on a vector space $V$. Then we have: $$ \HMss{V}{\hU}(\rho)  = \begin{cases} V \setminus  V_{\geq 0} & \text{ if $\rho < 0$} \\
    V & \text{ if $\rho = 0$} \\
    V \setminus U V_{\leq 0} & \text{ if $\rho > 0$.}
    \end{cases}$$ 
\end{lemma} 

\begin{proof}
Suppose first that $\rho < 0$. Define $\lambda_n(t) := t^n$. If $v \in V \setminus V_{\geq 0}$, then $\lim_{t \rightarrow 0}\lambda_n(t) \cdot (v)$ exists if and only if $n \leq 0$. In this case we have that $\langle \rho, \lambda_n \rangle \geq 0$, therefore $v \in \HMss{V}{\GG_m}(\rho)$, see \thref{semistableloci}. If $v \in V_{\geq 0}$, then $\lim_{t \rightarrow 0}\lambda_n(t) \cdot (v)$ exists if and only if $n \geq 0$. But if $n >0$ then $\langle \rho, \lambda_n \rangle < 0$, so $v \notin \HMss{V}{\GG_m}(\rho)$.  Therefore $\Qss{V}{\GG_m}(\rho) = V \setminus V_{\geq 0}.$  By the grading property of $\hU$, we have that $U V_i \subseteq \bigoplus_{j > i} V_j$, so that $V_{\geq 0}$ is $U$-invariant. Therefore $$\Qss{V}{\hU}(\rho) = \bigcap_{u \in U} u \Qss{V}{\GG_m}(\rho) = V \setminus U V_{\geq 0} = V \setminus V_{\geq 0}.$$

 Suppose that $\rho > 0$, then one similarly shows $\Qss{V}{\GG_m}(\rho) = V \setminus V_{\leq 0}$.  Hence $$\Qss{V}{\hU}(\rho) = \bigcap_{u \in U} u \Qss{V}{\GG_m}(\rho) = \bigcap_{u \in U} V \setminus V_{\leq 0} = V \setminus U V_{\leq 0}. $$ 
  
  Finally if $\rho = 0$, then $\Qss{V}{\GG_m}(\rho) = \justss{V}{\GG_m} = V$, so that $\Qss{V}{\hU}(\rho) = \bigcap_{u \in U} u V = V.$\end{proof}

We can now prove \thref{equalityuhat}.
\begin{proof}[Proof of \thref{equalityuhat}]
        First suppose that the equalities of semistable loci in \thref{equalityuhat} hold for $\rho$. 
    If $\rho=0$, then $\HMss{V}{\hU}(\rho) = V$ by \thref{describingHMrhoss}, yet the explicit description  \eqref{explicitdescripeq} shows $\Qss{V}{\hU}$ is never equal to $V$, so $\rho$ cannot be trivial.   If $\rho > 0$, then $\HMss{V}{\hU}(\rho) = V \setminus U V_{\leq 0}$. If $\omega_{\operatorname{min}} < 0$, then $\Qss{V}{\hU} = V^0_{\operatorname{min}}$, which can never equal $V \setminus U V_{\leq 0}$ as $U ( V_{\operatorname{min}} \setminus \{0\} ) \subseteq V^0_{\operatorname{min}} $ while $U ( V_{\operatorname{min}} \setminus \{0\} ) \nsubseteq V \setminus U V_{\leq 0}$. If $\omega_{\operatorname{min}} = 0$, then $\Qss{V}{\hU} = V^0_{\operatorname{min}}  \setminus U V_{\operatorname{min}}$, while $\HMss{V}{\hU}(\rho) = V \setminus U V_{\operatorname{min}}$ so the two are not equal. Finally if $\omega_{\operatorname{min}} > 0$ then $\Qss{V}{\hU} = V^0_{\operatorname{min}}$ while $\HMss{V}{\hU,\rho} = V$, so again the two are not equal. Thus we have $\rho < 0$ and $\HMss{V}{\hU}(\rho) = V \setminus V_{\geq 0}$, which coincides with $\Qss{V}{\hU}(\rho)$ if and only if $V \setminus V_{\geq 0}=V^0_{\operatorname{min}}$. This happens if and only if $\omega_{\operatorname{min}} < 0 $ and $\omega_{\operatorname{next}}  \geq 0$. 

    Conversely, suppose that $\rho < 0$, $\omega_{\min} < 0$ and $\omega_{\operatorname{next}} \geq 0$.  In this case the proof of \thref{describingHMrhoss} shows that $V^0_{\operatorname{min}}  = \justss{V}{\GG_m}(\rho) $. By the discussion proceeding \thref{Qssforuhattwisted}, we have that $\Qss{V}{\hU}(\rho)= \Qss{V}{\hU} = \vmino$. To show that  $\Iss{V}{\hU}(\rho) = V^0_{\min}$, it is enough to show that any $\Gm$ semi-invariant with respect to $\rho$ is $U$-invariant. This can be seen by the fact that the infinitesimal $U$-action on the dual representation $V^\ast$ must move the $-\omega_i$ weight space to a strictly higher weight space, because $U$ is graded by the $\Gm$. We then use the fact that $(V_{\min})^\ast = (V^\ast)_{\max}$. 
\end{proof}

\section{Simple examples with two-by-two matrices}  \label{sec:2by2ex}

In this section we illustrate the results of Section \ref{sec:affineuhat} through two examples involving two-by-two matrices. We consider the Borel subgroup $$ \hU = \GG_a \rtimes \GG_m = \left\{ \left. \begin{pmatrix} t & u \\
0 & t^{-1} \end{pmatrix} \  \right|  \ t \in k^{\ast}, u \in k \right\} \subseteq \SL_2(k),$$ which is graded by the 1PS consisting of diagonal matrices $\operatorname{diag}(t, t^{-1})$ for $t \in k^{\ast}$, acting on $V = \operatorname{Mat}_{2 \times 2}(k)$ by conjugation in Section \ref{subsec:2by2}, and by left multiplication in Section \ref{subsec:leftmult}.  In Table \ref{examples1} we summarise the semistable loci and quotients obtained from applying the results of Section \ref{sec:affineuhat}.

\begin{table}[h!]
\begin{tabular}{|c |   c |  c|}
\hline 
$\hU \to \GL(V)$ & Conjugation & Left multiplication   \\
\hline 
$\omega_{\min}$ & -2  & -1 \\ \hline 
$\Qss{V}{\hU} \to V \gitq \hU $  &  $\{M \in V \ | \ M_{21} \neq 0\} \longrightarrow \mathbb{A}^2$ \quad  &  $\{M \in V  \ | \ M_{21} \text{ or } M_{22} \neq 0\} \longrightarrow \mathbb{A}^1 \times \mathbb{P}^1$ \\  
 & \quad \quad \quad $ M \mapsto (\tr M, \det M) $  & \quad \quad \quad \quad $M \mapsto (\det M, [M_{21}: M_{22}])$  \\ \hline 
$X \gitq \hU$ & $\PP(1,1,2)$ & $\PP(2,1) \times \PP^1$  \\ \hline 
$X \gitq  \hU \setminus V \gitq \hU$ & $\PP(1,2)$   & $\PP^1$ \\ \hline 
$\Qss{V}{U} $ & $\{M \in V \ | \ M_{21} \neq 0\}$  &  $\{M \in V  \ | \ M_{21} \text{ or } M_{22} \neq 0\}$ \\ \hline 
$V \gitq U$ & $\mathbb{A}^1 \setminus \{0\} \times \mathbb{A}^2 $  & $\mathbb{A}^3 \setminus \{(x,y,z) \ | \ x=y=0\}$ \\ \hline
$\Qss{V}{U} \to V \gitq U $ & $M \mapsto (M_{21},\tr M, \det M) $  &  $M \mapsto (M_{21},M_{22},\det M)$ \\ \hline
 \end{tabular}
 \vspace{0.3cm}
 \caption{Explicit description of semistable loci and quotients for linear actions of the Borel subgroup $\hU$ in $\SL_2(k)$ and of its unipotent radical $U$ on $V = \operatorname{Mat}_{2 \times 2}(k)$. Given a matrix $M$, the $(i,j)$-th entry is denoted by $M_{ij}$.}
 \label{examples1}
\end{table}

\subsection{Conjugation action}  \label{subsec:2by2} 
Let $V = \operatorname{Mat}_{2 \times 2}(k)$ and consider the conjugation action of $\hU$ on $V$. The $\GG_m$-weights on $V$ are $\omega_{\min} = -2, \omega_{\operatorname{next}} = 0$ and $\omega_{\max} = 2$. Moreover, we have: $$ V_{\min} = \left\{   \begin{pmatrix} 0 & 0 \\
c & 0  \end{pmatrix} \in V \right\} \text{ and } V^0_{\min} = \left\{ \left.  \begin{pmatrix} a & b \\
c & d \end{pmatrix} \in V \ \right| \ c \neq 0   \right\}.$$ 
A simple calculation shows that \ref{Ucondaffine} holds for the $\hU$-action on $V$. Hence \thref{quotienting} provides a quasi-projective $\hU$-geometric quotient of $\vmino$, and a projective completion given by the projective geometric quotient $\Qss{X}{\hU} \to X \gitq \hU$ where $X = \PP(V \oplus k)$. 
By \thref{boundaryprop}, the boundary of the projective completion coincides with the projective geometric quotient $ \Qss{\PP(V)}{\hU} \to \PP(V) \gitq \hU$,  with the inclusion $\Qss{\PP(V)}{\hU} \hookrightarrow \Qss{X}{\hU}$ given by $[M] \mapsto [M:0]$. By  \thref{affinebyunipotent} there exists a quasi-projective geometric quotient for the action of $U = \GG_a$ on $\vmino$, with $(X \times \PP^1) \gitq \hU$ as a projective completion. The above semistable loci and quotients can each be described explicitly.

\begin{proposition}[Quotients for conjugation action on two-by-two matrices] \thlabel{rk2matrices} 
We have: 
\begin{enumerate}[(i)]
    \item \label{V} the geometric quotient for the action of $\hU$ on $V^0_{\operatorname{min}}$ is, up to isomorphism, the map $$ V^0_{\operatorname{min}}  = \left\{ \left.  M \in V \ \right| \ M_{21} \neq 0   \right\}    \to V \gitq \hU \cong \mathbb{A}^2, \quad M   \mapsto (\operatorname{tr} M, \operatorname{det} M);$$ 
   \item \label{PC} the projective geometric quotient for the action of $\hU$ on $\Qss{X}{\hU}$ is, up to isomorphism, the map $$ \Qss{X}{\hU}  = \left\{ \left.  \left[ M : w \right]  \in X \ \right| \ M_{21} \neq 0 \text{ and if } w = 0 \text{ then }  (\operatorname{det} M, \operatorname{tr} M) \neq 0  \right\}   \to X \gitq \hU \cong \PP(1,1,2) $$ given by $[M:w]  \mapsto [w:\operatorname{tr} M : \operatorname{det} M ];$
    \item \label{boundary} the projective geometric quotient for the action of $\hU$ on $\Qss{\PP(V)}{\hU}$  is, up to isomorphism, the map $$  \Qss{\PP(V)}{\hU} =   \left\{ \left.  [M] \in \PP(V)  \ \right| \ M_{21} \neq 0  \text{ and } \operatorname{det} M, \operatorname{tr} M \neq 0 \right\}    \to \PP(1,2) \cong \PP(V) \gitq \hU $$ given by $    [M]  \mapsto [\operatorname{tr} M : \operatorname{det} M ];$
    \item the quasi-projective geometric quotient for the action of $U$ on $\Qss{V}{U} = \vmino$ is, up to isomorphism, the map $$ \vmino \to \mathbb{A}^1 \setminus \{0\} \times \mathbb{A}^2, \quad M \mapsto (M_{21}, \tr M, \det M)$$ where $M_{21}$ denotes the bottom left entry of $M$. 
\end{enumerate}

\end{proposition}

\begin{proof}
For \eqref{V}, it follows from a simple calculation that the map is surjective and $\hU$-invariant. If $M$ and $M'$ both have the same trace and determinant, then it can be shown by direct calculation that there is an element of $\hU$ which conjugates $M$ to $M'$. Therefore the fibres of the map correspond to $\hU$-orbits. Since $V^0_{\operatorname{min}}$ is irreducible and $\mathbb{A}^2$ is normal, by \cite[Cor 25.3.4]{Tauvel2005} the given map is a geometric quotient. The result then follows from uniqueness of geometric quotients. For \eqref{PC}, the description of $\Qss{X}{\hU}$ follows from the observation that $[M:w]$ lies in $U \zmin$ if and only if $w=0$ and $\det M = \tr M = 0$. Since $\PP(1,1,2)$ is normal and $\Qss{X}{\hU}$ irreducible, to show that the given map is a geometric quotient, again it suffices to show that it is surjective and has fibres corresponding to $\hU$-orbits. But this follows from the exact same argument as in the previous case, as $\hU$ acts trivially on the coordinate $w$. The proofs of the last two items are analogous. 
\end{proof}  

We conclude this example with three remarks. 

\begin{remark}[Comparison with affine GIT quotient of $V$ by $\SL_2(k)$]
The affine GIT quotient $V \gitq \SL_2(k)$ for the conjugation action of $\SL_2(k)$ on $V = \Mat_{2 \times 2}(k)$ is the map $\pi: V \to k^2$ defined by $ M \mapsto (\operatorname{tr} M, \operatorname{det} M).$ This map is a good $\SL_2(k)$-quotient but not a geometric $\SL_2(k)$-quotient, nor is it a geometric $\SL_2(k)$-quotient when restricted to $V^0_{\operatorname{min}}$. However, by \thref{rk2matrices} \eqref{V} this restriction is a geometric $\hU$-quotient. In particular the fibres of $\pi|_{V^0_{\operatorname{min}}}$ are precisely the $\hU$-orbits. 
\end{remark}

\begin{remark}[Invariants do not necessarily extend] \thlabel{invtsnotextending}
    We can use the conjugation action of $\hU$ on $V$ to produce an example of invariants on a closed subvariey not extending to the ambient variety.     Consider the subvariety $V'$ of $V$ consisting of upper triangular matrices. The two functions that pick out the diagonal entries of a matrix are invariant for the $\hU$-action on $V'$, but do not extend to invariants on $V$. 
\end{remark}

\begin{remark}[Link with Section \ref{subsec:kingss}]
   Since $\omega_{\min} < 0 $ and $\omega_{\operatorname{next}} \geq 0$ in this example, by \thref{equalityuhat} for any character $\rho < 0$ of $\hU$ we have $\Qss{V}{\hU}(\rho) = \HMss{V}{\hU}(\rho) = \Iss{V}{\hU}(\rho)$. Since $\Qss{V}{\hU}(\rho) = \Qss{V}{\hU} = \vmino$, it follows that $\Iss{V}{\hU}(\rho) = \vmino$. We can see this equality directly, since the only way for a function on $V$ to be a semi-invariant of negative weight is if it is divisible by the function which picks out the bottom-left coordinate. 
\end{remark}

\subsection{Left multiplication} \label{subsec:leftmult}

Let $V = \Mat_{2 \times 2}(k)$ and consider $\hU$ acting on $V$ by left multiplication. Then $\omegamin = -1$ and $\vmin$ consists of matrices with zero top row, so that $\vmino$ consists of matrices with non-zero bottom row. Again a simple calculation shows that \ref{Ucondaffine} holds, so we can apply \thref{quotienting,affinebyunipotent}. In this example, we can also  construct these quotients by hand.

More precisely, using the same argument as in the proof of \thref{rk2matrices}, we can show that $$ \vmino \to \mathbb{A}^1 \times \PP^1, \quad M \mapsto (\det M, [M_{21}:M_{22}])$$ is a geometric $\hU$-quotient with projective completion is given by the $\hU$-geometric quotient $$\Qss{X}{\hU}  \to \PP(1,2) \times \PP^1, \quad [M:w] \mapsto ([ w: \det M ],[M_{21}:M_{22}]),$$ where $ \Qss{X}{\hU} = \{[M:w] \in X \ | \ M_{21} \text{ or } M_{22} \neq 0 \text{ and if $w=0$ then $\det M \neq 0$}\}.$ The boundary $X \gitq \hU \setminus V \gitq \hU = \PP(V) \gitq \hU$ is given by the geometric $\hU$-quotient $$ \Qss{\PP(V)}{\hU} \to \PP^1, \quad [M] \mapsto [M_{21}:M_{22}],$$ where $\Qss{\PP(V)}{\hU} = \{[M] \in \PP(V)  \ | \ \text{$\det M \neq 0$}\} = \PGL_2(k).$  By \thref{uhatthm}, the boundary $X \gitq \hU  \setminus V \gitq \hU$ is isomorphic to $\PP(V) \gitq \hU = \PGL_2(k) / \hU$. Finally, the geometric $U$-quotient for the action of $U$ on $\vmino$ is given by $$ \vmino \to D(x) \cup D(y) \subseteq \mathbb{A}^3, \quad M \mapsto (M_{21}, M_{22}, \det M)$$ where $D(x)$ and $D(y)$ denote the non-vanishing loci of the first and second coordinates on $\mathbb{A}^3$. 

\begin{remark}[Link with Section \ref{subsec:kingss}]
   As in the previous example, applying \thref{equalityuhat} we obtain that for any character $\rho < 0$ of $\hU$, there are equalities $\Qss{V}{\hU}(\rho) = \HMss{V}{\hU}(\rho) = \Iss{V}{\hU}(\rho)$. Since $\Qss{V}{\hU}(\rho) = \Qss{V}{\hU} = \vmino$, it follows that $\Iss{V}{\hU}(\rho) = \vmino$. We can see this equality directly, since the semi-invariants of negative weights are polynomials in the functions picking out the bottom left and bottom right entries respectively. 
\end{remark}

In both this example and the previous, the $\hU$-invariants on $V$ are finitely generated (generated by the trace and determinant in the case of conjugation, and by the determinant in the case of left multiplication), so that the untwisted quotient $\Spec k[V]^{\hU}$ is well-defined as a variety. In both examples we can see that there is a projective morphism $V \gitq \hU \to \Spec k[V]^{\hU}$: the identity map in the conjugation example, and the projection $\mathbb{A}^1 \times \PP^1 \to \mathbb{A}^1$ in the left multiplication example. It is natural to ask whether in the non-reductive setting, whenever $\omega_{\min} < 0$, $\omega_{\operatorname{next}} \geq 0$ and  $k[V]^{\hU}$ is finitely generated, the quotient $V \gitq  \hU$  (which can be thought of as a $\rho$-twisted quotient for any $\rho < 0$ by \thref{equalityuhat}) admits a projective morphism to the untwisted quotient $\Spec k[V]^{\hU}$, as in the reductive case. This question is more easily addressed from the point of view of what we will call relative non-reductive GIT, a perspective that we will explore in forthcoming work.

\section{Affine GIT from projective GIT for graded linear algebraic groups}  \label{sec:affinegitfromprojgitforH}

In this section we explain how to construct quotients for linear actions of internally or externally graded linear algebraic group on affine varieties twisted by a character.

\subsection{Internally graded linear algebraic group actions} \label{subsec:affineinternal}
Suppose that an internally graded linear algebraic group $H = U \rtimes R$ acts linearly on a vector space $V$. Throughout this section we assume that $R$ is of the form $R' \times \lambda(\GG_m)$ -- this is true up to a finite group up so the assumption will not alter any of the semistable or stable loci.  

Our aim is to define a quotienting semistable locus for the action of $H$ on $V$ twisted by a character $\rho$ of $H$, which we can show admits a good $H$-quotient under certain assumptions, by constructing a quotient of the $H$-action on $X:= \mathbb{P}(V \oplus k)$, or on the blow-up $\hX$ of $X$ as defined in Notation \ref{notation def of Xhat}. Since $\Qss{V}{H}(\rho)$ is independent of $\rho$ when $H = \hU$ is a graded unipotent group (see Section \ref{subsec:kingss}), 
 we only consider characters of $\rho$ that are trivial on the grading group $\lambda(\GG_m)$, which we identify with characters of $R'$.

\subsubsection{The $\rho$-twisted $H$-semistable locus in $V$}

For a reductive group $G$, the $\rho$-twisted semistable locus $\justss{V}{G}(\rho)$ is the domain of definition of the map $V \dashrightarrow \Proj k[V]^{G, \rho}$ in \eqref{rationalmap}. We will not use an analogous definition for $H$, as $k[V]^{H, \rho}$ may not be finitely generated. Instead, recalling from \thref{equalityonV} the equality for $G$ reductive   $$ \justss{V}{G}(\rho) = V \cap \justss{X}{G}(\mathcal{O}(d)_{\rho}) \quad\quad  \text{for $d \gg 1$},$$ our approach is to take this equality as the definition of the $\rho$-twisted semistable locus for $H$. It will be convenient for us to work with fractional linearisations, and so we will instead work with the linearisation $\mathcal{O}(1)_{\frac{\rho}{d}}$ for $d \gg 1$, or equivalently $\mathcal{O}(1)_{\varepsilon \rho}$ for $0 < \varepsilon \ll 1$. This will not affect the semistable loci or quotients.

To specify how small to take $\varepsilon$, we  consider the linear $R'$-action on $\vmin$. From here on we will denote by $\mathcal{N}$ the null cone for this action, i.e.\ $\mathcal{N} = \pi^{-1}(\overline{0})$ where $\pi: \vmin \to \vmin \gitq R'$ is the affine GIT quotient for the $R'$-action on $\vmin$ and $\overline{0} = \pi(0)$.
Let $\rho$ denote the restriction to $R'$ of $\rho$. By \thref{equalityonV,equalityonPV} applied to the $R$-action on $V \subseteq X:= \PP(V \oplus k)$ and to the $R'$-action on $\vmin \subseteq \PP(\vmin \oplus k)$, for all $0<\varepsilon \ll 1$ the following equalities are satisfied: 
\begin{align}
    \vmin^{R'-(s)s}(\rho) & = \vmin \cap \justssors{\PP(\vmin \oplus k)}{R'}(\mathcal{O}_{\PP(\vmin \oplus k)}(1)_{\varepsilon\rho}) \label{equality2}  \\
    \PP ( \vmin^{R'-(s)s}(\rho) \setminus \mathcal{N})  & = \PP(\vmin) \cap \PP(\vmin \oplus k)^{R'-(s)s}(\mathcal{O}_{\PP(\vmin \oplus k)}(1)_{\varepsilon\rho}). \label{equality3} \end{align}

We will define a $\rho$-twisted $H$-semistable locus in $V$ by intersecting $V$ with the $H$-semistable locus in $X = \PP( V \oplus k)$ (respectively in its blow-up $\hX$) if $\omega_{\min} < 0$ (respectively if $\omega_{\min} \geq 0$). Describing the latter requires introducing notation for various linearisations. 

\begin{notation}[Linearisations]
   Choose $N \gg 1$ such that $\psi^*\mathcal{O}_X(N) \otimes \mathcal{O}(-E)$ is ample on $\hX$. Then choose $d\gg0 $ such that \eqref{equality2} and \eqref{equality3} are satisfied with $\varepsilon = N/d$.   \begin{enumerate}[(i)] 
    \item  Consider the linearisation $\Lc:=\mathcal{O}_X(1) _{\frac{\rho}{d}}$ for the $H$-action on $X$ and the linearisation $\widehat{\Lc} := \psi^*\mathcal{L}^{\otimes N} \otimes \mathcal{O}(-E)$ for the $H$-action on $\hX$.
    \item Let $\Lc_+$ (respectively \ $\widehat{\mathcal{L}}_+$) denote the well-adapted $H$-linearisation on $X$ (respectively\ $\hX$) obtained by twisting the $H$-linearisation $\mathcal{L}$ (respectively\ $\widehat{\mathcal{L}}$) by a suitable rational character of $\lambda(\GG_m)$.
    \item Let $\Lc'$ (respectively \  $\widehat{\Lc'}$) denote the $R'$-linearisation  on $\zmin$ (respectively \ $\widehat{Z}_{\min}$) obtained by restricting the $R$-linearisation $\Lc$ on $X$ (respectively \ $\widehat{\Lc}$ on $\hX$).
    \end{enumerate}
\end{notation}

\begin{definition}[$\rho$-twisted semistable locus for the action of $H$ on $V$ twisted by $\rho$]\thlabel{definition:rhotwistedsemistablelocus}
   For an internally graded unipotent group $H = U \rtimes R$ acting linearly on a vector space $V$ and a character $\rho$ of $H$ that is trivial on the grading group $\lambda(\GG_m)$, the \emph{$\rho$-twisted semistable locus} for the $H$-action on $V$ is $$ \justss{V}{H}(\rho) : = \begin{cases} V \cap \justss{X}{H}(\Lc_+) & \text{if $\omega_{\min} < 0$} \\
    V \cap \psi(\justss{\hX}{H}(\widehat{\Lc}_+) \setminus E) & \text{if $\omega_{\min} \geq 0$}, 
    \end{cases}$$ 
    assuming \ref{Usscondproj} and \ref{Rcondproj} hold for the action on $X$ if $\omega_{\min} < 0$ and on $\hX$ if $\omega_{\min} \geq 0$. 
\end{definition}

Our aim is to impose conditions on the $H$-action on $V$ such that $\Qss{V}{H}(\rho)$ is well-defined, i.e.\ such that \ref{Usscondproj} and \ref{Rcondproj} hold for the action on $X$ if $\omega_{\min} < 0$ and on $\hX$ if $\omega_{\min} \geq 0$, since by doing so we can obtain a quotient of $V$ by restricting a quotient of $X$ or $\hX$. Once we have identified such conditions, we will see that $\justss{V}{H}(\rho)$ is independent of $d$ for $d \gg 1$, even if this is not obvious a priori from the definition (see \thref{thmQsslocusforH}). As both \ref{Usscondproj} and \ref{Rcondproj} relate to the $R'$-semistable points in $\zmin$ (if $\omega_{\min} < 0$) and in $\widehat{Z}_{\min}$ (if $\omega_{\min} \geq 0$), we first describe these loci in terms of points in $V$. 

\subsubsection{Describing $\zmin^{R'-ss}(\Lc')$ and $\widehat{Z}_{\min}^{R'-ss}(\widehat{\Lc}')$ in terms of $V$} \label{subsubsec:zminss}

\begin{proposition} \thlabel{zminssvminss}
    Suppose that an internally graded group $H =  U \rtimes R$ acts linearly on a vector space $V$. Let $\rho$ denote a character of $H$ that is trivial on the grading multiplicative group $\lambda(\GG_m)$ of $H$. Then we have: \begin{enumerate}[(i)]
    \item if $\omega_{\min} < 0$, then $ \zmin^{R'-(s)s}(\Lc') = \PP ( \vmin^{R'-(s)s}(\rho) \setminus \mathcal{N} ) \subseteq \PP(\vmin),$ where $\PP(\vmin) \subseteq X$ consists of the points at infinity; \label{case1} 
    \item if $\omega_{\min} > 0$, then $ \widehat{Z}_{\min}^{R'-(s)s}(\widehat{\mathcal{L}}') = \PP (\vmin^{R'-(s)s}(\rho) \setminus \mathcal{N} ) \subseteq \PP(\vmin) ,$ where $\PP(\vmin)$ is identified as a subvariety of $E$; 
    \label{case2}
    \item if $\omega_{\min} = 0$ and $\rho$ is trivial, then $ \widehat{Z}_{\min}^{R'-(s)s}(\widehat{\mathcal{L}}') = \{ ([v:w],[v']) \in \widehat{Z}_{\min} \ | \ v' \in \vmin^{R'-(s)s}(0) \setminus \mathcal{N}\},$ where $\vmin^{R'-ss}(0) = \vmin;$  \label{case4} 
    \item \label{case3} if $\omega_{\min} = 0$ and $\rho$ is non-trivial, then $$ \widehat{Z}_{\min}^{R'-(s)s}(\widehat{\mathcal{L}}') = \{ ([v:w],[v']) \in \widehat{Z}_{\min} \ | \ v' \in \vmin^{R'-(s)s}(\rho),\text{ and } v' \notin \mathcal{N} \text{ if $w=0$} \}.$$  
    \end{enumerate} 
\end{proposition}

\begin{proof}[Proof of \thref{zminssvminss} \eqref{case1}]
      Writing $X$ as $V \sqcup \PP(V)$, we can identify $\zmin$ with $\PP(\vmin) \subseteq \PP(V)$ as in \thref{zminandxomin}. Under this identification, we first show that \begin{equation} Z_{\operatorname{min}}^{R'-(s)s}(\Lc') = \PP(\vmin)^{R'-(s)s}(\mathcal{O}_{\PP(\vmin)}(1)_{\frac{\rho}{d}}) \label{startingequality}\end{equation}  holds.  To see this, first note that the line bundle underlying $\mathcal{L}'$ is $\mathcal{O}_{X}(1)|_{\zmin}$, while the line bundle on the right-hand side is $\mathcal{O}_{\PP(\vmin)}(1)$, so these coincide as $\zmin = \PP(\vmin)$. 
      It remains to show that the $R'$-equivariant structures are the same. On the left-hand side, the $R$-action is twisted by the character $\rho$ and then restricted to $R'$. This is equivalent to first restricting to $R'$ and twisting by $\rho$, which is the $R'$-equivariant structure on the right-hand side. Therefore \eqref{startingequality} holds.

Since $\mathbb{P}(V_{\min}) \subset  \mathbb{P}(V_{\min} \oplus k)$ is closed and $R'$-invariant, we have \begin{equation} \PP(\vmin)^{R'-(s)s}(\mathcal{O}_{\PP(\vmin)}(1)_{\frac{\rho}{d}}) = \PP( \vmin) \cap \PP(\vmin \oplus k)^{R'-(s)s}(\mathcal{O}_{\PP(\vmin \oplus k)}(1)_{\frac{\rho}{d}}).\label{intersection} \end{equation}  The proof for both the semistable and stable loci then follows from  \eqref{equality3}.
\end{proof}

We now turn to the case where $\omega_{\min} > 0$. 

\begin{proof}[Proof of \thref{zminssvminss} \eqref{case2}] 
    By \thref{zminandxomin}, we have that $\widehat{Z}_{\min} = \{ [0:1] \} \times \mathbb{P}(\vmin)$ is contained in the exceptional divisor $E \cong \mathbb{P}(V)$ for the blow-up $\psi : \hX \rightarrow X$ of $X$ at $[0:1]$. By definition $\widehat{\Lc}'$ is obtained by restricting $\widehat{\mathcal{L}}:=\psi^{\ast} \Lc'^{\otimes N} \otimes \mathcal{O}(-E)$ to the $R'$-action on $\widehat{Z}_{\min}$. As line bundles, we have $\psi^{\ast} \Lc'^{\otimes N}|_{\widehat{Z}_{\min}} \cong \mathcal{O}_{\mathbb{P}(\vmin)}$  and $\mathcal{O}(-E)|_{\widehat{Z}_{\min}} \cong 
    \mathcal{O}_{\mathbb{P}(\vmin)}(1)$, so that $\widehat{\mathcal{L}}' \cong  \mathcal{O}_{\mathbb{P}(\vmin)}(1)$. The $R'$-equivariant structure on $\psi^{\ast} \Lc'^{\otimes N}|_{\widehat{Z}_{\min}}$ is given by twisting the natural linearisation by $\frac{N}{d} {\rho}$ and the $R'$-equivariant on $\mathcal{O}(-E)|_{\widehat{Z}_{\min}}$ is induced by the linear action. Thus $\widehat{\mathcal{L}}' \cong  \mathcal{O}_{\mathbb{P}(\vmin)}(1)_{\frac{N}{d} \rho}$ as $R'$-linearisations.     
It follows then from the same argument used to prove \thref{zminssvminss} \eqref{case1} that $$ \widehat{Z}_{\operatorname{min}}^{R'-(s)s}(\widehat{\Lc}') = \PP(\vmin)^{R'-(s)s}(\mathcal{O}_{\PP(\vmin)}(1)_{\frac{N}{d} \rho}).$$ The desired descriptions of $\widehat{Z}_{\min}^{R'-(s)s}(\widehat{\Lc}')$ then follow in the same way as for \thref{zminssvminss} \eqref{case1}.    
\end{proof}

To prove \thref{zminssvminss} \eqref{case4} and \eqref{case3}, we will make use of the following result by Reichstein describing (semi)stability for reductive group actions under equivariant blow-ups \cite{Reichstein1989}. 

\begin{proposition}[Reichstein]\thlabel{Reichsteinprop} 
    Suppose that a reductive group $G$ acts on a projective variety $X$ with linearisation $\Lc$, and let $\phi: Y \to X$ denote the blow-up of $X$ along a $G$-invariant closed subvariety $C$ of $X$, with exceptional divisor $E$. For $N \in \mathbb{N}_{> 0}$, consider the linearisation $$\widehat{\Lc} : =  \phi^{\ast} \Lc^{\otimes N} \otimes \mathcal{O}(-E)$$ for the action of $G$ on $Y$. Then for $N$ sufficiently large, the linearisation $\widehat{\Lc}$ is ample, and for any $y \in Y$ and $x = \phi(y)$ we have: \begin{enumerate}[(i)]
        \item if $y$ is semistable, then $x$ is semistable; \label{unstable}
        \item if $x$ is stable, then $y$ is stable;  \label{stable}
        \item if $y$ is stable and $x \notin C$, then $x$ is stable; \label{stableawayfromC}
        \item \label{semistable} if $x$ is semistable and $x \notin C$, then $y$ is semistable if and only if $x \notin \widetilde{C}$, where $\widetilde{C}:= \pi^{-1}(\pi(C))$ denotes the saturation of $C$ with respect to the GIT quotient $\pi: X \dashrightarrow X \gitq G$.
        \setcounter{smallroman}{\value{enumi}}
    \end{enumerate}
    If in addition we assume that $X$ and $C$ are smooth at every point of $\widetilde{C}$, then we also have: \begin{enumerate}[(i)]
    \setcounter{enumi}{\value{smallroman}}
    \item   \label{semistable2} if $x$ is semistable and $x \in C$, then $y$ is semistable if and only if $y$ does not lie in the proper transform of $\widetilde{C}$, i.e.\ $y \notin \overline{\phi^{-1}(\widetilde{C} \setminus C) }$;
    \item  if $x$ is semistable and $x \in C$, then $y$ is not stable if and only if $y$ lies in the proper transform of  $X^{ss} \setminus X^s$.\label{strictlyss} 
    \end{enumerate}  
    Hence under this assumption $ \justss{Y}{G}(\widehat{L}) = \phi^{-1}(\justss{X}{G}(\Lc)) \setminus \overline{\phi^{-1}(\widetilde{C} \setminus C)}.$
\end{proposition}

We can now prove the desired results in the case where $\omega_{\min} = 0$ and $\rho$ is trivial. 

\begin{proof}[Proof of \thref{zminssvminss} \eqref{case4}]
     By \thref{zminandxomin} we have $\widehat{Z}_{\min} = \{([v:w],[v']) \in \widehat{X}  \ | \ [v'] \in \PP(\vmin) \},$ which is isomorphic to the blow-up of $\zmin=\PP(\vmin \oplus k)$ at $[0:1]$. Let $\phi: \widehat{Z}_{\min} \to \zmin$ denote the blow-down map; note that this is the restriction of $\psi$ to $\widehat{Z}_{\min}$. Applying \thref{Reichsteinprop}, since $X$ and $C:= [0:1]$ are both smooth, we have that \begin{equation} \widehat{Z}_{\min}^{R'-ss}(\widehat{\Lc}') = \phi^{-1} \left(  Z_{\min}^{R'-ss}(\Lc') \right)  \setminus \overline{\phi^{-1}(\widetilde{C} \setminus C )},  \label{Reichstein}
    \end{equation} where $\widetilde{C}$ denotes the saturation of $C$ with respect to the quotient map $\pi_{\rho}: \zmin \dashrightarrow \zmin \gitq_{\hspace{-2pt} \Lc' \hspace{2pt}} R'.$ Since $\rho$ is trivial, $0 \in \vmin^{R'-ss}(\rho) = V$ and so $[0:1] \in \zmin^{R'-ss}(\Lc')$ by the equality \eqref{equality2}. Therefore we have $ \widetilde{C} : = \pi_{\rho}^{-1}( \pi_{\rho} ([0:1])).$  Using the inclusion $V \hookrightarrow X$ given by $v \mapsto [v: 1]$, we obtain that $\pi_{\rho}^{-1}(\pi_{\rho}([0:1])) = \pi^{-1}(\pi(0)) = \mathcal{N}.$ It follows then from \eqref{Reichstein} that \begin{equation}  
    \widehat{Z}_{\min}^{R'-ss}(\widehat{\Lc'}) = \phi^{-1} \left(  Z_{\min}^{R'-ss}(\Lc') \right)  \setminus \overline{\phi^{-1}(\mathcal{N} \setminus \{[0:1]\})}. \label{descriptionofzminhatss} \end{equation}
We have that $\overline{\phi^{-1}(\mathcal{N} \setminus \{[0:1]\})} = \{([v:w],[v']) \in \widehat{Z}_{\min} \ | \ v' \in \mathcal{N} \}$ and also $\zmin = \mathbb{P}(\vmin \oplus k)$ by \thref{zminandxomin}.  As $\rho$ is trivial, we have $\zmin^{R'-ss}(\Lc') =  \vmin \sqcup \PP(\vmin \setminus \mathcal{N})$ by  \thref{equalityonV,equalityonPV} and our assumptions on $d$.  Hence $\phi^{-1} (  Z_{\min}^{R'-ss}(\Lc') ) = \{([v:w],[v']) \ | \ v' \notin  \mathcal{N} \text{ if $w=0$}\}.$  As a result we obtain the desired description $$\widehat{Z}_{\min}^{R'-ss}(\widehat{\Lc'}) =  \left\{ ([v:w],[v']) \in \widehat{Z}_{\min} \ | \ v' \notin \mathcal{N} \right\}.$$ 
    
    To prove the statement about stability, suppose that $\widehat{z} = ([v:w],[v']) \in \widehat{Z}_{\min}^{R'-ss}(\widehat{\Lc'}).$ It follows from \thref{Reichsteinprop} \eqref{unstable} that $z = \phi(\widehat{z})  = [v:w]$ must lie in $\zmin^{R'-ss}(\Lc')$. Suppose first that $z$ lies in the centre of the blow-up, i.e.\ that $[v:w]=[0:1]$. Then by \thref{Reichsteinprop} \eqref{strictlyss}  we have that $\widehat{z}$ is stable if and only if it does not lie in the proper transform of $\zmin^{R'-ss}(\Lc') \setminus \zmin^{R'-s}(\Lc')$. Using the fact that $\zmin^{R'-ss}(\Lc') = \vmin \sqcup \PP( \vmin \setminus \mathcal{N})$ and that $\zmin^{R'-s}(\Lc') = \vmin^{R'-s}(\rho) \sqcup \PP(\vmin^{R'-s}(\rho) \setminus \mathcal{N}),$ since $v=0$ by assumption it follows that $\widehat{z}$ is stable if and only if $v' \in \vmin^{R'-s}(\rho)$.  
    Now suppose that $z$ does not lie in the centre of the blow-up. Then by \thref{Reichsteinprop} and the fact that $z$ is semistable we have that $\widehat{z}$ is stable if and only if $z$ is stable. From the description of $\zmin^{R'-s}(\Lc')$ given in the above paragraph, we have that $z$ is stable if and only if $v' \in \vmin^{R'-s}(\rho)$.  The desired description of $\widehat{Z}_{\min}^{R'-s}(\widehat{\mathcal{L}}')$ follows. 
    \end{proof}

The final case to prove is when $\omega_{\min} = 0$ and $\rho$ is non-trivial. 

\begin{proof}[Proof of \thref{zminssvminss} \eqref{case3}] 
     Since $\rho$ is non-trivial, using the equality \eqref{equality2} and the fact that $\zmin = \PP( \vmin \oplus k)$ we can see that $[0:1] \notin \zmin^{R'-ss}(\Lc')$. Hence the saturation $\widetilde{C}$ of the center $C= \{[0:1]\}$ of the blow-up $\hX \to X$ is empty. Therefore by \thref{Reichsteinprop} we have \begin{equation} \widehat{Z}_{\min}^{R'-ss}(\widehat{\Lc'}) =   \phi^{-1} \left(  Z_{\min}^{R'-ss}(\Lc') \right). \label{secondcase} \end{equation} By \thref{zminandxomin} and \thref{equalityonV,equalityonPV} we have that $ \zmin^{R'-ss}(\Lc') = \vmin^{R'-ss}(\rho) \sqcup \PP(\vmin^{R'-ss}(\rho) \setminus \mathcal{N}).$ The desired description of $\widehat{Z}_{\min}^{R'-ss}(\widehat{\Lc}')$ follows easily.

     For stability, consider $\widehat{z} = ([v:w],[v']) \in \widehat{Z}_{\min}^{R'-ss}(\rho)$ and let $z := [v:w] = \phi(\widehat{z})$. The point $z$ must lie in $\zmin^{R'-ss}(\rho)$ since otherwise $\widehat{z}$ would be unstable by \thref{Reichsteinprop} \eqref{semistable}. Since $\rho$ is non-trivial, we know that $[0:1] \notin \zmin^{R'-ss}(\Lc')$ therefore $z$ does lie in the centre of the blow-up. Applying \thref{Reichsteinprop} we obtain that $\widehat{z}$ is stable if and only if $z$ lies in  $Z_{\min}^{R'-s}(\Lc')$, and the latter stable locus coincides with $\vmin^{R'-s}(\rho) \sqcup \PP( \vmin^{R'-s}(\rho) \setminus \mathcal{N})$ by \thref{zminandxomin} and \thref{equalityonV,equalityonPV}. The description of $\widehat{Z}_{\min}^{R'-s}(\widehat{\Lc}')$ follows. 
\end{proof}

This completes the proof of \thref{zminssvminss}.

\subsubsection{Achieving the conditions \ref{Usscondproj} and \ref{Rcondproj}} 

\thref{zminssvminss} yields conditions for the $H$-action on $V$ which ensure that \ref{Usscondproj} and \ref{Rcondproj} hold for the linearised $H$-action on $X$ and on $\hX$.

\begin{corollary} \thlabel{stabcondfromv} For a character $\rho$ of $R'$, if the linear $H$-action on $V$ satisfies  \begin{equation}  \Stab_U(v) = \{e\}  \begin{cases} 
   \text{ $\forall \ v \in \vmin^{R'-ss}(\rho) \setminus \mathcal{N}$}  & \text{if $\omega_{\min} \neq 0$, or if $\omega_{\min} = 0$ and $\rho$ is trivial}\\
   \text{ $\forall \ v \in \vmin^{R'-ss}(\rho)$} & \text{if $\omega_{\min} = 0$ \text{ and } $\rho$ is non-trivial,}
   \end{cases}
\tag*{$[U;ss]_{\mathrm{aff}}$} \label{Usscondaffine} \end{equation} then \ref{Usscondproj} holds for the  $H$-action on $X$ if $\omega_{\min} < 0$, and for $\hX$ if $\omega_{\min} \geq 0$.
\end{corollary}
\begin{proof} This follows directly from \thref{zminssvminss} and the fact that since unipotent groups have no non-trivial characters, the scalar equivalence class of a point $v$ is fixed by $U$ if and only if $v$ itself is fixed by $U$. 
\end{proof}
We omit the proof of the analogous result for the condition \ref{Rcondproj}, as it is straighforward.
\begin{corollary} 
    \thlabel{ss=sforzmin} 
 For a character $\rho$ of $R'$, if the linear $H$-action on $V$ satisfies 
    \begin{equation} \tag*{$[R]_{\mathrm{aff}}$} \label{Rcondaffine}  \begin{cases} 
   \vmin^{R'-ss}(\rho) \setminus \mathcal{N} \subseteq \vmin^{R'-s}(\rho)  & \text{if $\omega_{\min} \neq 0$, or if $\omega_{\min} = 0$ and $\rho$ is trivial}\\
   \vmin^{R'-ss}(\rho) = \vmin^{R'-s}(\rho) & \text{if $\omega_{\min} = 0$ \text{ and } $\rho$ is non-trivial,} 
   \end{cases} \end{equation}  
   Then \ref{Rcondproj} holds for the linearised $H$-action on $X$ if $\omega_{\min} < 0$ and on $\widehat{X}$ if $\omega_{\min} \geq 0$.   
\end{corollary}

\subsubsection{Quotient of the $\rho$-twisted $H$-semistable locus in $V$}

We can now state and prove the main result of this section: for an internally graded group $H$ acting linearly on a vector space $V$ satisfying \ref{Usscondaffine} and \ref{Rcondaffine}, we construct a quotient of the $\rho$-twisted semistable locus  (see \thref{definition:rhotwistedsemistablelocus}).

\begin{theorem}[Quotient of $\justss{V}{H}(\rho)$] \thlabel{Hquotient}
    Let $H$ act linearly on a vector space $V$ and let $\rho$ denote a character of $H$ that is trivial on $\lambda(\GG_m)$. If \ref{Usscondaffine} and \ref{Rcondaffine} hold for the $\hU$-action on $V$, then there exists a quasi-projective geometric $H$-quotient $\justss{V}{H}(\rho) \to V \gitq_{\hspace{-2pt} \rho \hspace{1pt}} H := \Qss{V}{H}(\rho) / H$, with  projective completion $X \gitq_{\hspace{-2pt} \Lc_+ \hspace{2pt}} H$ if $\omega_{\min} < 0$ and $\hX \gitq_{\hspace{-2pt} \widehat{\Lc}_+ } H$ if $\omega_{\min} \geq 0$. 
\end{theorem}

\begin{proof}
By \thref{stabcondfromv,ss=sforzmin},  the conditions \ref{Usscondaffine} and \ref{Rcondaffine} ensure that \ref{Usscondproj} and \ref{Rcondproj} hold for  the linear action of $H$ on $X$ if $\omega_{\min} <0$ and on $\hX$ if $\omega_{\min} \geq 0$. Thus when $\omega_{\min} <0$ (respectively $\omega_{\min} \geq 0$), there is a projective geometric quotient $\Qss{X}{H}(\Lc_+) \rightarrow X \gitq_{\hspace{-2pt} \Lc_+ } H$ (respectively $\Qss{\hX}{H}(\widehat{\Lc}_+) \rightarrow \hX \gitq_{\hspace{-2pt} \widehat{\Lc}_+ } H$). For $\omega_{\min} <0$, the locus $\justss{V}{H}(\rho):=V \cap \Qss{X}{H}(\Lc_+)$ is an open saturated subset of $\Qss{X}{H}(\Lc_+)$, and similarly for $\omega_{\min} \geq 0$, the locus $\justss{V}{H}(\rho):=V \cap \psi(\Qss{\hX}{H}(\widehat{\Lc}_+) \setminus E)$ can be identified via $\psi^{-1}$ as an open saturated subset of $\Qss{\hX}{H}(\widehat{\Lc}_+)$. By restricting the geometric quotients of $\Qss{X}{H}(\Lc_+)$ or $\Qss{\hX}{H}(\widehat{\Lc}_+)$ according to the sign of $\omega_{\min}$, we obtain a geometric $H$-quotient of $\Qss{V}{H}(\rho)$, with the desired projective completions.     
\end{proof}

\subsubsection{Explicit description of the $\rho$-twisted $H$-semistable locus in $V$}
We conclude Section \ref{subsec:affineinternal} with an explicit description of $\justss{V}{H}(\rho)$,  assuming that \ref{Rcondaffine} and \ref{Usscondaffine} are satisfied. Recall that $\mathcal{N}$ is the null cone for the action of $R'$ on $\vmin$.

\begin{theorem}[Description of $\justss{V}{H}(\rho)$] \thlabel{thmQsslocusforH}
    Let $H$ act linearly on a vector space $V$ and let $\rho$ be a character of $H$ that is trivial on the grading 1PS. If \ref{Usscondaffine} and \ref{Rcondaffine} hold, then we have $$ \justss{V}{H}(\rho) = \begin{cases} p_{\min}^{-1} \left( \vmin^{R'-ss}(\rho) \setminus \mathcal{N} \right)  & \text{if $\omega_{\min} \neq 0$} \\
    p_{\min}^{-1} \left( \vmin^{R'-ss}(\rho) \right) \setminus U \vmin  & \text{if $\omega_{\min} = 0$ and $\rho$ is non-trivial} \\
    p_{\min}^{-1} \left( \vmin \setminus \mathcal{N} \right) \setminus U \vmin  & \text{if $\omega_{\min} = 0$ and $\rho$ is trivial.}  
    \end{cases} $$
\end{theorem}

\begin{proof} 
   For $\omega_{\min} < 0$, we have  $\justss{V}{H}(\rho) := V \cap \Qss{X}{H}(\Lc_+) = V \cap p^{-1} ( \zmin^{R'-ss}(\Lc' )) \setminus U \zmin.$  
    By \thref{zminandxomin} we know that $U \zmin \subseteq \PP(V)$ as $\omega_{\min} < 0$, so that $U \zmin \cap V = \emptyset$. Therefore $\justss{V}{H}(\rho) = V \cap p^{-1}  ( Z_{\operatorname{min}}^{R'-ss}(\Lc') ).$ For $v \in V$, we have $p([v : 1]) = [p_{\min}(v) :0]$ as $\omega_{\min} < 0$. Since $\zmin^{R'-ss}(\Lc'_+) = \PP(\vmin^{R'-ss}(\rho) \setminus \mathcal{N})$ by \thref{zminssvminss} \eqref{case1}, we obtain that $p([v:1])$ lies in $\zmin^{R'-ss}(\Lc')$ if and only if $p_{\min}(v)$ lies in $\vmin^{R'-ss}(\rho) \setminus \mathcal{N}$.  
    
    Suppose now that $\omega_{\min} \geq 0$; then $\justss{V}{H}(\rho) := V \cap \psi (\Qss{\hX}{H}(\widehat{\Lc}_+) \setminus E)$ and $\Qss{\hX}{H}(\widehat{\Lc}_+) = \widehat{p}^{-1} ( \widehat{Z}_{\min}^{R'-ss}(\widehat{\Lc'} )) \setminus U \widehat{Z}_{\min} $.   If $\omega_{\min} > 0$, then $U \widehat{Z}_{\min}$ is contained in the exceptional divisor $E$ by \thref{zminandxomin}. Therefore $ \Qss{V}{H}(\rho) = V \cap \psi ( p^{-1} ( \widehat{Z}_{\min}^{R'-ss}(\widehat{\Lc'} ) ) \setminus  E ). $ By \thref{zminssvminss} \eqref{case2} we have that $\widehat{Z}_{\min}^{R'-ss}(\widehat{\Lc}_+) =  \PP (\vmin^{R'-ss}(\rho) \setminus \mathcal{N} ) \subseteq \PP(\vmin) \subseteq \PP(V)\cong E$.  For $v \in V$, suppose that $([v:1],[v]) \in \Qss{\hX}{H}(\widehat{\Lc}_+) \setminus E$. Then $\widehat{p}([v:1],[v]) =  ([0:1],[p_{\min}(v)])$ lies in $ \widehat{Z}_{\min}^{R'-ss}(\widehat{\Lc}')$ if and only if $p_{\min}(v)$ lies in $\vmin^{R'-ss}(\rho) \setminus \mathcal{N}$. Noting that if $p_{\min}(v) \notin \mathcal{N}$, then $([p_{\min}(v):1],[p_{\min}(v)]) \notin E$, the desired description of $\Qss{V}{H}(\rho)$ follows.

     If $\omega_{\min} = 0$, then $U \widehat{Z}_{\min}$ is no longer contained in $E$, so \begin{equation} \justss{V}{H}(\rho) = V \cap \psi \left( \widehat{p}^{\hspace{0.2em}-1} ( \widehat{Z}_{\min}^{R'-ss}(\widehat{\Lc'} ) ) \setminus ( U \widehat{Z}_{\min} \cup E)  \right). \label{Qssvh} \end{equation}  
     Hence $v \in \justss{V}{H}(\rho)$ if and only if $v \notin U \vmin$, $v \neq 0$ and $([v:1],[v]) \in \widehat{p}^{\hspace{-0.2em} -1} (\widehat{Z}_{\min}^{R'-ss}(\widehat{\Lc'}))$. We identify non-zero $v \in V$ with $([v:1],[v]) \in \hX \setminus E$ and determine whether $\widehat{p}([v:1],[v]) = ([p_{\min}(v):1],[p_{\min}(v)])$ lies in $\widehat{Z}_{\min}^{R'-ss}(\widehat{\Lc'})$ using  \thref{zminssvminss}. If $\rho$ is trivial, then $\widehat{p}([v:1],[v]) \in \widehat{Z}_{\min}^{R'-ss}(\widehat{\Lc}')= \{ ([v:w],[v']) \in \widehat{Z}_{\min} \ | \ v' \notin \mathcal{N}\}$ if and only if $p_{\min}(v) \notin \mathcal{N}$. Moreover, $\vmin = \vmin^{R'-ss}(\rho)$, which completes the description for $\rho= 0$. If $\rho$ is non-trivial, then $\widehat{p}([v:1],[v]) \in \widehat{Z}_{\min}^{R'-ss}(\widehat{\Lc}')= \{ ([v:w],[v']) \in \widehat{Z}_{\min} \ | \ v' \in \vmin^{R'-ss}(\rho), \text{ and } v' \notin \mathcal{N} \text{ if $w=0$} \} $ if and only if $p_{\min}(v) \in \vmin^{R'-ss}(\rho)$, which completes the description for $\rho\neq  0$.
    \end{proof}

\subsection{Externally graded linear algebraic group actions}
\label{subsec:externallygradedaffine}
In this section we explain how to construct a quotient for the linear action of a linear algebraic group $H$ on a vector space $V$ when $H$ does not contain a grading multiplicative group. We fix throughout this section a linear algebraic group $H = U \rtimes R$, which we do \emph{not} assume is internally graded.  The assumption required for the construction to work is that the action of $H$ extends to the linear action of a group $\widehat{H} = H \rtimes \GG_m$ where $\widehat{H}$ is internally graded by $\GG_m$. Under this assumption, a $\rho$-twisted $H$-semistable locus in $V$ can be defined analogously to \thref{definition:rhotwistedsemistablelocus}.

\begin{definition}[$\rho$-twisted semistable locus for the action of $H$ on $V$]
   Suppose that $H = U \rtimes R' $ acts linearly on a vector space $V$, and that the action extends to the linear action of $\widehat{H} = H \rtimes \GG_m$ where $\widehat{H}$ is internally graded by $\GG_m$. Fix a character $\rho$ of $H$ and suppose that \ref{Usscondaffine} and \ref{Rcondaffine} hold for the action of $H$ on $V$. The \emph{$\rho$-twisted semistable locus for the action of $H$ on $V$} is $$ \justss{V}{H}(\rho) : = \begin{cases} V \cap \justss{X}{H}(\Lc_+) & \text{if $\omega_{\min} < 0$} \\
    V \cap \psi(\justss{\hX}{H}(\widehat{\Lc}_+) \setminus E) & \text{if $\omega_{\min} \geq 0$}. 
    \end{cases}$$ 
\end{definition}

Note that $\justss{V}{H}(\rho)$ depends on the choice of extension of the action from $H$ to $\widehat{H}$.

\begin{corollary} \thlabel{Hexternallygraded}
   Suppose that $H = U \rtimes R'$ acts linearly on a vector space $V$, and that the action extends to the linear action of $\widehat{H} = H \rtimes \GG_m$ where $\widehat{H}$ is internally graded by $\GG_m$. Fix a character $\rho$ of $H$ and suppose that \ref{Usscondaffine} and \ref{Rcondaffine} hold for the action of $H$ on $V$. Then $$ \justss{V}{H}(\rho)  = \begin{cases} p_{\min}^{-1}(\vmin^{R'-ss}(\rho) \setminus \mathcal{N})  & \text{if $\omega_{\min} \neq 0$, or if $\omega_{\min} =0$ and $\rho$ is trivial} \\
    p_{\min}^{-1}(\vmin^{R'-ss}(\rho) )  & \text{if $\omega_{\min} = 0$ and $\rho$ is non-trivial,}
    \end{cases} $$ and there exists a quasi-projective geometric $H$-quotient $ \justss{V}{H}(\rho) \to V \gitq_{\hspace{-2pt} \rho \hspace{2pt} } H,$ with projective completion $(X \times \PP^1) \gitq \widehat{H}$ if $\omega_{\min} < 0$ and $(\hX \times \PP^1) \gitq \widehat{H}$ if $\omega_{\min} \geq 0$. 
\end{corollary}

\begin{proof}
  The same argument used to prove \thref{Hquotient} shows that if $\omega_{\min} < 0$ (respectively if $\omega_{\min} \geq 0)$, then the linearised action of $\widehat{H}$ on $X$ (respectively on $\hX$) satisfies \ref{Usscondproj} and \ref{Rcondproj}. Thus we can apply \thref{externallygraded} to obtain a quasi-projective geometric quotient for the $H$-action on $\Qss{X}{H}(\Lc)$ if $\omega_{\min} < 0$, and for $H$-action on $\Qss{\hX}{H}(\widehat{\Lc})$ if $\omega_{\min}  \geq 0$.  These quotients have projective completions $(X \times \PP^1) \gitq \widehat{H}$ and $(\hX \times \PP^1) \gitq \widehat{H}$ respectively. 
  
    If $\omega_{\min} < 0$, then using the inclusion $V \hookrightarrow X$ given by $v \mapsto [v:1]$, by restricting to $V$ we obtain a quasi-projective geometric quotient for the $H$-action on $V \cap \Qss{X}{H}(\rho)$ with projective completion $(X \times \PP^1 ) \gitq \widehat{H}$.  If $\omega_{\min} \geq 0$, then we can instead obtain a quasi-projective geometric quotient for the $H$-action on $V \cap \psi(\Qss{X}{H}(\widehat{\Lc}) \setminus E)$, with projective completion $(\hX \times \PP^1) \gitq \widehat{H}$. The description of both intersections follows analogously to the proof of \thref{thmQsslocusforH}. We prove below the case where $\omega_{\min} = 0$ and $\rho$ is non-trivial; the other cases follow in a similar way. 

    We wish to show that $V \cap \psi(\justss{\hX}{H}(\widehat{\Lc}_+) \setminus E) =  p_{\min}^{-1}(\vmin^{R'-ss}(\rho) ).$ By \thref{externallygraded} we have that $\justss{\hX}{H}(\widehat{\Lc}_+) = \widehat{p}^{\hspace{0.2em} -1}(\justss{\widehat{Z}_{\min}}{R'}(\widehat{\Lc}'))$, and by \thref{zminssvminss} \eqref{case3} that $\justss{\widehat{Z}_{\min}}{R'}(\widehat{\Lc}') =\{ ([v:w],[v']) \in \widehat{Z}_{\min} \ | \ v' \in \vmin^{R'-ss}(\rho),\text{ and } v' \notin \mathcal{N} \text{ if $w=0$} \}. $  A point $([v:w],[v']) \in \hX$ does not lie in $E$ if and only if $v \neq 0$ (in which case $[v'] = [v]$), and its image under $\psi$ lies in $V$ if and only if $w \neq 0$. Moreover, by the description of  $\justss{\widehat{Z}_{\min}}{R'}(\widehat{\Lc}')$ we have that $\widehat{p}_{\min}([v:1],[v]) = ([p_{\min}(v):1],[p_{\min}(v)])$ lies in $  \justss{\widehat{Z}_{\min}}{R'}(\widehat{\Lc}') $ if and only if $p_{\min}(v)$ lies in $\vmin^{R'-(s)s}(\rho)$. This gives the desired equality. 
\end{proof}

\subsection{From quotients of vector spaces to quotients of affine varieties} \label{subsec:fromvstoaffine} 
 We have seen how to construct quotients for the linear actions of graded unipotent groups, internally graded and externally graded groups respectively on vector spaces. 
 In this section we generalise these results to actions on affine varieties. For simplicity we treat only the internally graded case, of which the graded unipotent group case is a special case. The externally graded case follows in the same way.   

Suppose that an internally graded linear algebraic group $H= U \rtimes R$ with grading multiplicative group $\lambda$ acts on an affine variety $Y$. Then by \cite[Lem 1.1]{Kempf1978} there is an equivariant embedding $Y\subseteq V$ into an $H$-representation $V$. Let $\rho$ denote a character of $H$, which is trivial on $\lambda(\GG_m)$. Then provided the action of $H$ on $V$ satisfies \ref{Usscondaffine} and \ref{Rcondaffine}, we can construct a quotient for the $H$-action on an explicit open subset of $Y$, by restricting the $\rho$-twisted $H$-quotient of $V$ to $Y$.

\begin{proposition}  \thlabel{quotientaffinev}
    Suppose that an internally graded group $H$ acts linearly on an affine variety $Y$ with respect to an embedding $Y \subseteq V$ into a vector space $V$ such that the conditions \ref{Usscondaffine} and \ref{Rcondaffine} hold for the linear action of $H$ on $V$. Then the restriction of the quasi-projective geometric quotient $\justss{V}{H}(\rho) \to V \gitq_{\hspace{-2pt} \rho \hspace{2pt} } H$ to $Y$ is a quasi-projective geometric quotient for the action of $H$ on $$ \justss{Y}{H}(\rho) : = Y \cap \justss{V}{H}(\rho).$$ 
    If $\omegamin < 0$, this quotient has a projective completion given by the restriction to the closure $\overline{Y}$ of $Y$ in $X=\PP(V \oplus k)$ of the quotient $\Qss{X}{H}(\mathcal{L}_+) \to X \gitq_{\hspace{-2pt} \mathcal{L}_+ } H$, which is a geometric quotient for the action of $H$ on $\overline{Y} \cap \Qss{X}{H}(\mathcal{L}_+).$ If $\omegamin \geq 0$, this quotient has a projective completion given by the restriction to the closure $\overline{Y}$ of $\psi^{-1}(Y) \subset \hX$ of the quotient $\Qss{\hX}{H}(\widehat{\mathcal{L}}_+) \to \hX \gitq_{\hspace{-2pt} \widehat{\mathcal{L}}_+ } H$.
\end{proposition}

\begin{proof}
    We prove the result when $\omegamin < 0$, as the other case is analogous but instead uses the blow-up $\hX$. By construction, the geometric quotient $\justss{V}{H}(\rho) \to V \gitq_{\hspace{-2pt} \rho \hspace{2pt} } H $ is the restriction to $V \cap \Qss{X}{H}(\mathcal{L}_+)$ of the geometric quotient $q: \Qss{X}{H}(\mathcal{L}_+) \to X \gitq_{\hspace{-2pt} \mathcal{L}_+  } H.$ Hence we wish to show that the restriction of $q$ to $Y \cap \Qss{X}{H}(\mathcal{L}_+) = Y \cap \justss{V}{H}(\rho) = \justss{Y}{H}(\rho)$ is a geometric quotient. Now by \thref{restrictingHquotients}, since $\overline{Y} \subseteq X$ is a closed $H$-invariant subvariety, the restriction of $q$ to $\overline{Y} \cap \Qss{X}{H}(\mathcal{L}_+)$ is a geometric $H$-quotient. This is a projective completion of the image of $\justss{Y}{H}$ under $q$. Moreover, since $Y \cap \Qss{X}{H}(\mathcal{L}_+) \subseteq \overline{Y} \cap \Qss{X}{H}(\mathcal{L}_+)$ is an open saturated $H$-invariant subset, it follows that further restricting to $\justss{Y}{H}(\rho)$ also gives a geometric quotient.   
\end{proof}

We conclude this section by noting that the condition \ref{Usscondaffine}, which applies to the action on the ambient affine space, cannot easily be translated to conditions on unipotent stabilisers on $Y$, even if $V$ is the linear span of $Y$ in an affine space, as the following example shows.

\begin{example}
Suppose that $\hU = \mathbb{G}_a \rtimes \mathbb{G}_m $ acts  on $ Y = \mathbb{V}(yz) \subseteq V = \mathbb{A}^3$ via $(u,t) \cdot (x,y,z) = (tx + u(y+z), t^{-1} y, t^{-1} z)$. Since $R = \lambda(\GG_m)$, condition \ref{Usscondaffine} is equivalent to \ref{Ucondaffine}.  Then $\vmin = \mathbb{V}(x)$ and all points in $Y \cap (\vmin \setminus \{0\})$ have trivial unipotent stabiliser group. Yet this is not true for all points in $\vmin \setminus \{0\}$.  
\end{example}

\section{Application to representations of quivers with multiplicities}  \label{sec:repqwm} 

In this section we apply the results of Section \ref{sec:affinegitfromprojgitforH} to the classification problem for representations of quivers with multiplicities. Classical representations of a quiver over a field $k$ are representations of the quiver in the category of (finite-dimensional) $k$-vector spaces. Representation of quivers with multiplicities are a generalisation: they are representations of the quiver in the category of free finite rank modules over a truncated polynomial ring determined by the multiplicity \cite{Yamakawa2010,Ringel2011,Geiss2014,Wyss2017,Hausel2018,Vernet2023}.

\subsection{Definitions} 

A quiver $Q= (Q_0,Q_1,s,t)$ is a finite, connected, directed graph consisting of a set $Q_0$ of vertices, a set $Q_1$ of arrows, and source and target maps $s,t: Q_1 \to Q_0$ for the arrows. A multiplicity is a natural number $m \in \mathbb{N}$,  which determines a truncated polynomial ring
\[ A_m := k[\epsilon]/(\epsilon^{m +1}).\]

\begin{definition} 
A \emph{representation of $(Q, m)$} is a representation of $Q$ in the category of finitely generated $A_m$-modules, that is, $\Phi = (\Phi_v, v \in Q_0; \Phi(a): \Phi_{s(a)} \to \Phi_{t(a)}, a \in Q_1)$ consisting of finitely generated $A_m$-modules $\Phi_v$ for each vertex $v$ and $A_m$-module homomorphisms $\Phi(a)$ for each arrow $a \in A$. A representation $\Phi$ of $(Q,m)$ is \emph{locally free} if $\phi_v$ is a free finitely generated $A_m$-module for each $v \in Q_0$. The \emph{rank} of $\Phi$ is the vector $d =(d_v)_{v \in Q_0}$ such that $\Phi_v$ is a $A_m$-module of rank $d_v$ for all $v \in Q_0$. For simplicity we say that $\Phi$ is a representation of $(Q,m)$ of rank $k$. The \emph{trivial representation} is the unique representation with rank vector zero.
\end{definition} 

The special case of locally free representations where $m = 0$ and $A_0 = k$ recovers the usual notion of representations of a quiver in the category of $k$-vector spaces, in which case the rank vector $d$ is the dimension vector. The case $m = 1$ was first considered in \cite{Ringel2011}, while the general case was first considered in \cite{Fan2010}, and later on in \cite{Wyss2017,Hausel2018,Vernet2023}. Yamakawa \cite{Yamakawa2010} considers varying multiplicities at each vertex, but only considers representations over $k$ rather than representations over truncated polynomial rings.  

Representations of $(Q,m)$ form an abelian category $\mathcal{R}ep(Q,m)$, with subobjects and homomorphisms defined analogously to the classical case (see \cite{King1994}). Locally free representations form a subcategory. For $m' < m$, there is a ring homomorphism  $\tau^m_{m'}: A_m \twoheadrightarrow A_{m'}$ given by truncating which induces an $m'$-truncation functor $ \mathcal{R}ep(Q,m) \rightarrow \mathcal{R}ep(Q,m')$. There is also a ring homomorphism $k \rightarrow A_m$ inducing an extension of scalars functor $ \mathcal{R}ep(Q):=\mathcal{R}ep(Q,0) \rightarrow \mathcal{R}ep(Q,m)$. Observe that the composition of extending and then truncating is the identity.  Both the truncation to trivial multiplicity $m' = 0$ and its extension to a representation of $(Q,m)$ play a prominent role in defining stability for representations of quivers with multiplicities, so we give them a name. 

\begin{definition}[Classical and extended classical representations for a representation of $(Q,m)$] \thlabel{underlyingclassicalrep}
    The \emph{classical truncation} of a $(Q,{m})$-representation $\Phi$ of rank $d$,  denoted by $\phi$, is the representation of $Q$ of dimension $d$ corresponding to the $0$-truncation of $\Phi$. The \emph{extended classical truncation} of $\Phi$, denoted $\Phi^0$, is the representation of $(Q,m)$ of rank $d$ obtained by viewing $\phi$ as a representation of $(Q,m)$ via extension of scalars. 
\end{definition}

The following result will be needed in Section \ref{subsec:comparisonrudakov}.

\begin{lemma}\thlabel{lem truncsubrep}
    If $\Phi$ is a locally free $(Q,m)$-representation then the classical truncation of any subrepresentation of $\Phi$ is a subrepresentation of $\Phi^0$. 
\end{lemma}
\begin{proof}
The $0$-truncation functor is obtained by applying the functor $-\otimes_{A_m}A_m/(\epsilon)$ to the diagram in the category of $A_m$-modules that forms the representation. Applying this to the commutative diagram that exhibits  $\Psi\subset  \Phi$ as a subrepresentation, we obtain a commutative diagram of $k$-vector spaces. Moreover, if $x \in \ker(\Psi^0_v \rar \Phi^0_v)$ then $\epsilon^mx\in \ker(\Psi_v \rar \Phi_v)$, so that the truncated maps must also be injective for each $v \in Q_v$. 
\end{proof}

The fact that there is not a natural ring homomorphism $A_{m'} \to A_{m}$ for $m ' < m$ means it is not obvious how to define quiver representations with varying multiplicities at each vertex: at each vertex we have modules over different truncated polynomial rings and we want module homomorphisms for each arrow, but we can only relate $A_m$-modules for different values of $m$ using the truncation homomorphisms, which imposes certain inequalities on our multiplicities depending on the arrows.  If we only ask for $k$-linear morphisms for each arrow, rather than module morphisms, one can study representations over $k$ of $Q$ with varying multiplicities as in \cite{Yamakawa2010}.

\subsection{Parameter space and group action} 
In this section, we describe a parameter space for locally free representations of $(Q,{m})$ of rank ${d}$ admitting a group action whose orbits correspond to isomorphism classes of representations analogously to the classical construction.

For any locally free representation $\Phi$ of $(Q,{m})$ of rank ${d}$, by choosing a basis of the free $A_m$-modules $\Phi_v$ at each vertex, we can identify the morphisms $\Phi(a): \Phi_{s(a)} \to \Phi_{t(a)}$ with $d_{t(a)} \times d_{s(a)}$-matrices $\Phi_a$ with entries in $A_{m}$. Therefore $$ \Rep(Q,m,d) := \prod_{a \in Q_1} \Mat_{d_{t(a)} \times d_{s(a)}}(A_{m})$$ parametrises locally free rank $d$ representations of $(Q,{m})$. We let $\Rep(Q,d) : = \Rep(Q,d,0)$, which parameterises dimension $d$ representations of $Q$. For any tuple of matrices ${\Phi} = (\Phi_a)_{a \in Q_1} \in \Rep(Q,m,d)$, we also let $\Phi$ denote the associated locally free representation of $(Q,m)$.
The redundancy coming from the choice of basis vectors is captured by the action of the product group $$ G_{m,d} := \prod_{v \in Q_0} \GL_{d_v}( A_m ).$$ Given ${g} = (g_v)_{v \in Q_0} \in G_{m,d},$ and ${\Phi} = (\Phi_a)_{a \in Q_1} \in \Rep(Q,m,d)$, the action is given by \begin{equation} {g} \cdot {\Phi}  = \left( g_{t(a)} \Phi_a g_{s(a)}^{-1} \right)_{a \in Q_1}. \label{action}
\end{equation}

\begin{remark}[Link with jet spaces]
    The varieties $\Rep(Q,m,d)$ and $G_{m,d}$ are the $(m+1)$-th jet spaces of $\Rep(Q,d)$ and $G_d : = G_{0,d}$ respectively, and the action of $G_{m,d}$ on $\Rep(Q,m,d)$ is the natural extension to jet spaces of the classical action of $G_d$ on $\Rep(Q,d)$, see \cite[Rk 2.1.1]{Vernet2023} and \cite[Ch 3]{Chambert2018}. 
\end{remark}

Exactly as in the classical case, we obtain the following result.

\begin{proposition} \thlabel{setup}
Isomorphism classes of locally free representations of $(Q,m)$ of rank $d$ are in bijection with $G_{m,d}$-orbits in $\Rep(Q,m,d)$. Moreover, for ${\Phi} \in \Rep(Q,m,d)$, we have an isomorphism $\Aut_{(Q,m)}( {\Phi} )\cong \Stab_{G_{m,d}}({\Phi}).$
\end{proposition}

Henceforth all representations will be assumed to be locally free.

\subsection{Matrices with coefficients in truncated polynomial rings}

To construct a quotient for the $G_{m,d}$-action on $\Rep(Q,m,d)$, 
we first describe how matrices with entries in $A_m = k[\epsilon]/(\epsilon^{m+1})$ can be embedded in larger matrices with entries in $k$ in order to describe the structure of $G_{m,d}$.

For each $l,n \in \mathbb{N}$, we have an embedding 
\begin{equation}\label{embedd matrices truncated poly coeffs} \begin{array}{rcc} \Mat_{l \times n}(A_m) & \hookrightarrow &\Mat_{(m+1) l \times (m+1) n}(k) \\ 
& & \\
\Phi = \Phi^0 + \Phi^1 \epsilon + \cdots + \Phi^{m} \epsilon^{m} & \mapsto  &\begin{pmatrix}
 \Phi^0 & \Phi^1 & \cdots &  \Phi^{m} \\
   0 &  \Phi^0 &  \cdots &  \Phi^{m-1}  \\
  \vdots & \ddots  & \ddots  & \vdots \\

 0 &   \cdots & 0 &   \Phi^0
\end{pmatrix}
\end{array}
\end{equation}
 where each $\Phi^i \in \Mat_{l \times n}(k)$. Moreover, these embeddings are compatible with matrix multiplication. In particular, we have a group homomorphism 
\begin{equation}\label{embedding quiver with mult gp} \GL_n(A_m) \hookrightarrow \GL_{(m+1)n}(k),\end{equation}
thus $g = g^0 + g^1 \epsilon + \cdots + g^{m} \epsilon^{m} \in \Mat_{n \times n}(A_m)$ is invertible if and only if $g^0$ is invertible. Moreover, we see that $\GL_n(A_m)$ is non-reductive as soon as $m >0$. Its unipotent radical consists of elements of the form $\operatorname{Id} + u^1 \epsilon + \cdots u^{m} \epsilon^{m}$ with $u^i \in \Mat_{n \times n}(k)$, while a Levi subgroup is given by elements of the form $g^0$ for $g^0 \in \GL_n(k)$.

The group $\GL_n(A_m)$ does not contain a grading multiplicative subgroup, but we can use the embedding $\GL_n(A_m) \hookrightarrow \GL_{(m+1)n}(k)$ to find an external grading as follows.  Let $\lambda_{m,n} : \GG_{m}(k) \to \GL_{(m+1)n}(k)$ be the 1PS given by 
\begin{equation}\label{lambda grading for quivers with mult}
\lambda_{m,n}(t) : = \begin{pmatrix} t^{m} \operatorname{Id} & 0 & \cdots  & 0 \\
0 & t^{m-1} \operatorname{Id} &  & 0 \\
\vdots &  &  \ddots &   \\
0 & 0 &     & t^0 \operatorname{Id}
\end{pmatrix},
\end{equation}
where $\operatorname{Id}$ is the $n \times n$ identity matrix with entries in $k$. Then $\lambda_{m,n}$ acts with strictly positive weights on the Lie algebra of the unipotent radical of $\GL_n(A_m)$ via conjugation, and moreover commutes with the centre of the Levi subgroup of $\GL_n(A_m)$. Hence $\GL_n(A_m) \rtimes\lambda_{m,n}(\GG_m)$ is graded by $\lambda_{m,n}$.

\subsection{Constructing the quotient} \label{subsec:constructingquotient} 

In this section we describe how to construct a quotient of the $G_{m,d}$-action on $\Rep(Q,m,d)$. As in the classical case, there is a diagonal subgroup acting trivially. Let $\Delta_m$ be the image of the diagonal embedding
\[ \begin{array}{rcl} \GL_1(A_m) 
\cong \GG_m \times \GG_a^m & \hookrightarrow &G_{m,d} = \prod_{v \in Q_0} \GL_{d_v}(A_m) \\ g^0 + g^1\epsilon + \dots + g^m \epsilon^m & \mapsto & (g^0 I_{d_v} + g^1 I_{d_v} \epsilon + \cdots g^m I_{d_v} \epsilon^m)_{v \in Q_0}. \end{array} \]
Then $\Delta_m \subseteq G_{m,d}$ fixes any representation in $\Rep(Q,m)$. Consequently, one can equivalently construct a quotient for the action of $H :=G_{m,d}/\Delta_m $ on $\Rep(Q,m,d)$.

Taking a product of embeddings of the form \eqref{embedding quiver with mult gp}, we obtain an injective group homomorphism
\begin{equation}\label{embedd quiver group}
G_{m,d} \hookrightarrow \prod_{v \in Q_0} \GL_{(m+1) d_v}(k).
\end{equation}
Write $H = U \rtimes R$ where $U$ is the unipotent radical and $R = \prod_{v \in Q_0} \GL_{d_v}(k) /\Delta_0$ is the Levi factor. We use the 1PSs $\lambda_{m,d_v}(t) $ of $ \GL_{(m+1) d_v}(k)$ defined in  \eqref{lambda grading for quivers with mult} to externally grade $H$ as follows.

\begin{proposition}[Externally grading 1PS for $H$]
The 1PS $\lambda$ of $\prod_{v \in Q_0} \GL_{(m+1) d_v}(k)$ defined by $$ \lambda(t) : = (\lambda_{m,d_v}(t))_{v \in Q_0}$$ positively grades the unipotent radical $U$ of $H$, so that $\widehat{H} : = H \rtimes \lambda(\GG_m)$ is internally graded by $\lambda$. 
\end{proposition}

To be able to apply the results of Section \ref{subsec:externallygradedaffine} to the action of $H = G_{m,d}/ \Delta_m $ on $\Rep(Q,m,d)$, we must also show that the action of $H$ on $\Rep(Q,m,d)$ extends to an action of $\widehat{H}$. For this, we note that via the embeddings \eqref{embedd matrices truncated poly coeffs} we obtain an embedding
\[ \Rep(Q,m,d) \hookrightarrow W: = \prod_{a \in Q_1} \Mat_{(m+1) d_{t(a)} \times (m+1) d_{s(a)}}(k)\]
which is equivariant with respect to the homomorphism \eqref{embedd quiver group} and the natural conjugation actions. 
Moreover, $\Rep(Q,m,d) \subset W$ is preserved by the conjugation action of $\lambda(\GG_m)$. Indeed, writing $\Phi_a = \Phi_a^0 + \cdots + \Phi_a^m$ for $a \in Q_1$, a direct calculation gives \begin{equation} \lambda(t) \cdot {\Phi} = (\Phi_a^0 + t \Phi^1_a  \epsilon + t^2 \Phi^2_a \epsilon^2 + \cdots + t^{m} \Phi^m_a \epsilon^{m})_{a \in Q_1}. \label{gradingGMaction}
\end{equation} In summary, we deduce the following result.

\begin{proposition}
    The linear action of $H$ on $\Rep(Q,m,d)$ extends to a linear action of $\widehat{H}= H \rtimes \lambda(\GG_m)$, with $$(h, \lambda(t)) \cdot {\Phi} := {g} \cdot (\lambda(t) \cdot {\Phi})$$ for any lift  ${g} \in G_{m,d}$ of $h \in H = G_{m,d} / \Delta_m$.  
\end{proposition}

We can thus apply \thref{Hexternallygraded}  to obtain the following result.

\begin{proposition} \thlabel{quotient} 
    Fix a character $\rho$ of the Levi $R$ of $H$. If the $\widehat{H}$-action on $V= \Rep(Q,m,d)$ has the property that \ref{Rcondaffine} and \ref{Usscondaffine} hold, then there exists a quasi-projective geometric quotient for the action of $H$ on $$ \Qss{V}{H}(\rho)  = \begin{cases} 
    p_{\min}^{-1} \left( \vmin \setminus \mathcal{N} \right)  & \text{if $\rho$ is trivial} \\
    p_{\min}^{-1} \left( \vmin^{R-ss}(\rho) \right)   & \text{if $\rho$ is non-trivial.}
    \end{cases} $$
\end{proposition}

\begin{proof}
    This follows from \thref{Hexternallygraded} as we see $\omega_{\min} = 0$ from the $\lambda(\GG_m)$-action described at \eqref{gradingGMaction}.
\end{proof}

\subsection{Stability and moduli space} \label{subsec:stabandms}

\thref{quotient} is only useful if we can interpret in a moduli-theoretic way, i.e.\ in terms of properties intrinsic to representations of quivers with multiplicities, the conditions \ref{Rcondaffine} and \ref{Usscondaffine}, as well as the locus $\Qss{V}{H}(\rho).$ As a first step, we describe $p_{\min}: \vmino \to \vmin$. 

\begin{lemma}[Moduli-theoretic interpretation of $p_{\min}: \vmino \to \vmin$] \thlabel{modinterppmin}
   For ${\Phi} \in V= \Rep(Q,m,d)$, the retraction $p_{\min}(\Phi)$ coincides with the extended classical truncation ${\Phi}^0$ of $\Phi$. Hence ${\Phi}$ lies in $\vmino$ if and only ${\Phi}^0$ is non-zero. Moreover, there is an isomorphism \begin{equation} \vmin \cong \Rep(Q,d). \label{vmin}
    \end{equation}  
\end{lemma}

\begin{proof} 
By equation \eqref{gradingGMaction}, the minimal weight for the $\lambda(\GG_m)$-action on $V$ is $\omega_{\min} =0$ and thus $$\vmin = \left\{\Phi \in V \ | \ \Phi_a^i = 0 \text{ for all $i > 0$ and all $a \in Q_1$} \right\} = \{ \Phi \in V \ | \ \Phi = \Phi^0 \} \cong \Rep(Q,d).$$ From this description we see that the truncation map $V = \Rep(Q,m,d) \to \Rep(Q,d)$ coincides with $p_{\min}$ and $ \vmino = \left\{ \Phi \in V \ | \ \Phi_a^0 \neq 0 \text{ for some $a \in Q_1$} \right\}$, which completes the proof.
\end{proof}

We recall some standard definitions for representations of $Q$ (without multiplicities). A \emph{stability parameter} for $d$-dimensional representations of $Q$ is a tuple $\rho = (\rho_v)_{v \in Q_0}$ of integers satisfying $\sum_{v \in Q_0} \rho_v d_v = 0$; this defines a character $\rho$ of the fixed Levi subgroup $R$ of $H$ with $\rho(\Delta_1) = 1$ as in \cite{King1994}. Moreover King defines a $d$-dimensional representation $\phi$ of $Q$ to be $\rho$-\emph{semistable} (respectively $\rho$-\emph{stable}) if $\rho(\phi'):= \sum_{v \in Q_0} \rho_v \dim \phi'_v \geq 0$ (respectively \ $\rho(\phi') >0$) for all proper subrepresentations $\phi' \subseteq \phi$. A representation of $Q$ of dimension $d$ is \emph{nilpotent} if its restriction to every cycle in $Q$ is a nilpotent endomorphism. Equivalently, $\phi$ is nilpotent if there is a natural number $N$ such that $\phi$ restricted to any length $N$ path is zero. In particular, if $Q$ is acyclic (without oriented cycles), all representations are nilpotent. Recall that given a representation $\Phi$ of $(Q,m)$, we denote by $\phi$ its classical truncation (see \thref{underlyingclassicalrep}). 

\begin{lemma}[Moduli-theoretic interpretation of \ref{Rcondaffine}] \thlabel{modinterpR}
    For ${\Phi} \in \vmin \cong \Rep(Q,d)$, we have: \begin{enumerate}[(i)]
    \item ${\Phi} \in\vmin^{R-(s)s}(\rho)$ if and only if ${\phi}$ is $\rho$-(semi)stable; \label{part1}
    \item ${\Phi} \in \mathcal{N}$ if and only if ${\phi}$ is nilpotent; \label{part2} 
    \item if $\rho= {0}$ then ${\Phi}$ is $\rho$-stable if and only if ${\phi}$ is irreducible (has no proper subrepresentations).  \label{part3}
    \end{enumerate} As a result, the condition \ref{Rcondaffine} holds for the $\widehat{H}$-action on $V$ twisted by $\rho$ if and only if: \begin{enumerate}
    \item[\emph{Case}] $\rho \neq 0$: $\rho$-semistability coincides with $\rho$-stability for representations of $Q$ of dimension $d$,
    \item[\emph{Case}] $\rho = 0$: any non-simple representation of $Q$ of dimension $d$ is nilpotent. 
    \end{enumerate} 
\end{lemma}

\begin{proof}
   Using the isomorphism \eqref{vmin} and the fact that $R = \prod_{v \in Q_0} \GL_{d_v}(k)/\Delta_0$, \eqref{part1} is \cite[Prop 3.1]{King1994}. For \eqref{part2}, recall that if $\pi: \vmin \to \vmin \gitq R$ is the affine GIT quotient, the null cone $\mathcal{N}:= \pi^{-1}(\overline{0})$ is the locus of points where all non-constant invariants vanish. Since the $R$-invariants on $\vmin \cong \Rep(Q,d)$ are generated by traces along oriented cycles  \cite{Lebruyn1990}, the null cone precisely consists of nilpotent quiver representations. Statement \eqref{part3} follows immediately from the definition of $\rho$-stability. The final statement follows immediately from these claims.
\end{proof}

We now give a moduli-theoretic interpretation of the assumption \ref{Usscondaffine} in \thref{quotient}. 

\begin{lemma}[Moduli-theoretic interpretation of \ref{Usscondaffine}] \thlabel{modinterpUss}
    The condition \ref{Usscondaffine} holds for the action of $\widehat{H}$ on $V$ twisted by $\rho$ if and only if every rank $d$ representation ${\Phi}$ of $(Q,m)$ for which $\phi$ is $\rho$-semistable has the property that $\Aut_{(Q,m)}(\Phi^0) / \Delta_m$ is reductive.  
\end{lemma}

\begin{proof} 
This follows from \thref{setup} and the fact that the unipotent radical of $\Aut_{(Q,m)}(\Phi^0)$ is isomorphic to $\Stab_{U_{m,d}}(\Phi^0)$ under the isomorphism of \thref{setup}, by \cite[Prop 2.23]{Hoskins2021}. 
\end{proof} 

The trivial representation $\Phi$ of $(Q,m)$ satisfies $\Phi=\Phi^0$ and $\phi = 0$. If $\rho =0$  then $\phi$ is $\rho$-semistable as a representation of $Q$. However the automorphism group of $\Phi^0$ is all of $G_{d,m} / \Delta_m$, which is not reductive if $m > 0$. Therefore \ref{Usscondaffine} can never hold when $\rho =0$ and $m >0$. For this reason we assume from here on that $\rho \neq 0$. 

Given the above moduli-theoretic descriptions of the conditions in \thref{quotient}, we formulate the following notion of $\rho$-stability for representations of quivers with multiplicities coming from the NRGIT notion of stability. We compare it to a Rudakov notion of stability in Section \ref{subsec:comparisonrudakov}.

\begin{definition}[Stability for representations of quivers with multiplicities] \thlabel{stabwithmult}
  Fix a quiver with multiplicity $(Q,m)$, rank vector $d$, and a non-zero stability parameter $\rho = (\rho_v)_{v \in Q_0}$. Then a representation ${\Phi}$ of $(Q, m)$ of rank $d$ is \emph{NRGIT $\rho$-stable} if: \begin{enumerate}[(i)]
  \item its classical truncation ${\phi}$ is $\rho$-stable; \label{firstcond}
  \item its \emph{non-trivial automorphism group} $\Aut_{(Q,m)}({\Phi}^0) / \Delta_m$ is reductive.  \label{secondcond} 
  \end{enumerate}
\end{definition}

Based on the above, \thref{quotient} can be reformulated in the following way.

\begin{theorem} \thlabel{maintheoremquivers}
   Fix a quiver with multiplicity $(Q,m)$, rank vector $d$, and a non-zero stability parameter $\rho = (\rho_v)_{v \in Q_0}$, and suppose that:  \begin{enumerate}[(i)]
       \item  any $\rho$-semistable $d$-dimensional representation of $Q$ is $\rho$-stable;
       \label{assumption1}
       \item any $\rho$-semistable $d$-dimensional representation of $Q$, viewed as a representation of $(Q,m)$ via extension of scalars, has a reductive non-trivial automorphism group. \label{assumption2}
    \end{enumerate} 
    Then there exists a quasi-projective coarse moduli space $\mathcal{M}^{\rho-s}(Q,m,d)$ for NRGIT $\rho$-stable representations of $(Q,m)$ of rank $d$, which moreover comes with an explicit projective completion. 
\end{theorem}

\begin{proof}
    By \thref{modinterpR,modinterpUss}, assumptions \eqref{assumption1} and \eqref{assumption2} are equivalent to \ref{Rcondaffine} and \ref{Usscondaffine} respectively. Thus by \thref{quotient} there exists a quasi-projective geometric quotient for the $H$-action on $\Qss{V}{H}(\rho)$ with projective completion $(\hX \times \PP^1) \gitq \widehat{H}$. Since $\rho \neq 0$, we have $\Qss{V}{H}(\rho) = p_{\min}^{-1}(\vmin^{R-ss}(\rho))$ by \thref{quotient}. Then combining \thref{modinterppmin,modinterpR,modinterpUss} we see that ${\Phi}$ lies in $\Qss{V}{H}(\rho)$ if and only if ${\phi}$ is $\rho$-stable. It follows then from \thref{setup} that the geometric quotient $\Qss{V}{H}(\rho)$ is a coarse moduli space for $\rho$-stable rank $d$ representations of $(Q,m)$. 
\end{proof}

It is natural to ask when the assumptions of \thref{maintheoremquivers} are satisfied; Section \ref{subsec:toric} will show that these assumptions are satisfied for toric representations when the stability condition is generic. Before turning to the toric case, we first compare the notion of NRGIT $\rho$-stability from \thref{stabwithmult} to another natural definition of $\rho$-stability.

\subsection{Comparison with Rudakov $\rho$-stability} \label{subsec:comparisonrudakov} For representations of quivers with multiplicities, there is another natural definition of stability, closer to King's definition. This definition is obtained from Rudakov's general formulation of stability in an abelian category, involving verifying an inequality for all subrepresentations.

\begin{definition}[Rudakov $\rho$-stability]
Fix a stability condition $\rho$. A representation $\Phi$ of $(Q,m)$ is \emph{Rudakov $\rho$-stable} if for all proper subrepresentations $\Phi'$ of $\Phi$, the inequality $\rho(\Phi') >0$ holds.
\end{definition} 

We now compare Rudakov $\rho$-stability to NRGIT $\rho$-stability.

\begin{proposition}[NRGIT $\rho$-stability implies Rudakov $\rho$-stability] \thlabel{lem comparing two stab defs for quivers}
   If a representation $\Phi$ of $(Q,m)$ has the property that its classical truncation $\phi$ is $\rho$-stable, then $\Phi$ is Rudakov $\rho$-stable. In particular, if $\Phi$ is NRGIT $\rho$-stable then $\Phi$ is Rudakov $\rho$-stable. 
\end{proposition}

\begin{proof}
By \thref{lem truncsubrep}, if $\Phi$ is not Rudakov $\rho$-stable and destabilised by $\Psi \subseteq \Phi$, then the classical truncation $\Psi^0 \subseteq \Phi^0$ is a subrepresentation destabilising $\Phi^0$.  
\end{proof}

The converse implication is not necessarily true, we will see a counterexample  in Section \ref{subsubsec:workedexample}. This same example can also be used to show that although the $\rho$-stable representations admit a coarse mdoduli space, the Rudakov $\rho$-stable representations may not. Thus NRGIT should be viewed as providing a stronger stability condition than that of Rudakov $\rho$-stability, having the property that a coarse moduli space can be constructed for the smaller stable locus.

\subsection{Toric representations} \label{subsec:toric} 
A \emph{toric representation (with multiplicities)} is a locally free representation of $(Q,m)$ with rank vector ${d}=(1,\hdots, 1)$, as studied in \cite{Wyss2017,Hausel2018,Vernet2023}. 
In this section we show that if $\rho$ is generic (see \thref{genericstab}), then the assumptions of \thref{quotient} hold for the action of $\widehat{H}$ on $V = \Rep(Q,m,d)$ twisted by $\rho$, giving a coarse moduli space for $\rho$-stable  representations of $(Q,m)$ of rank $d$. We also include a simple worked example. A natural question is whether  

\subsubsection{Characterising $\rho$-stability and constructing a moduli space for toric representations}

To describe when $\Aut_{(Q,m)}({\Phi^0})/ \Delta_m$ is reductive for representations $\Phi = \Phi^0$ (the only case under which we will interpret this assumption), we introduce the support quiver of a representation (defined as a subquiver of $Q$), which is closely related to the invariant introduced in \cite[$\S$7.1]{Hausel2018}. 

\begin{definition}[Graph of a representation]
    The \emph{support quiver} of a representation ${\Phi}$ of $(Q,m)$ is the subquiver $\Gamma({\Phi})$ of $Q$ with vertices $Q_0$ and arrows given by those $a \in Q_1$ such that the classical truncation $\phi_a$ is non-zero. 
\end{definition}

Note that $\Gamma(\Phi)  = \Gamma(\Phi^0) = \Gamma(\phi) $. 

\begin{proposition} \thlabel{reinterpretation}
    The support quiver of a toric representation ${\Phi}$ of $(Q, m)$ is connected if and only if  $\Aut_{(Q,m)} ( {\Phi^0}) / \Delta_m$ is reductive.
\end{proposition}

\begin{proof} As seen in the proof of \thref{modinterpUss}, the unipotent radical of $\Aut_{(Q,m)}(\Phi^0)$ is isomorphic to $\Stab_{U_{m,d}}(\Phi^0)$. Therefore $\Aut_{(Q,m)} ( {\Phi^0}) / \Delta_m$ is reductive if and only if $\Stab_{U_{m,d}}(\Phi^0) \subseteq \Delta_m$.      Suppose first that $\Gamma({\Phi})$ is connected, and that $ {u} \cdot {\Phi^0} = {\Phi^0}$ for ${u} \in U_{m,d}$. Then for each arrow $a$ in $\Gamma({\Phi}) = \Gamma(\Phi^0)$ we have an equality of $1 \times 1$ matrices (with coefficients in $A_m$) given by $$ u_{t(a)} \Phi^0_a u_{s(a)}^{-1}  = \Phi^0_a.$$ Since $\Phi^0_a \neq 0$ as $a$ is an edge of $\Gamma({\Phi^0})$, it follows that $u_{t(a)}  = u_{s(a)}$ for each edge of $\Gamma({\Phi^0})$. As the quiver $\Gamma(\Phi^0)$ is connected, we obtain that $u_v = u_{v'}$ for all $v,v' \in Q_0$, so that ${u} \in \Delta_m$ as required.

    To prove the converse, suppose that $\Gamma(\Phi)$ is not connected. Let $\Gamma_1$ be a connected component of $\Gamma(\Phi)$, and $\Gamma_2$ be the union of the remaining connected components, which by assumption is nonempty. Choose two distinct elements $u_1,u_2$ of the unipotent radical of $\GL_1(A_m)$, and let $u =(u_v)_{v \in Q_0} \in U_{m,d}$ be given by \[ u_v = \begin{cases}  u_i \quad\text{if $v$ is a vertex of some $\Gamma_i$ }\\  0  \quad \text{ else}.\end{cases} \] Note that since $u_1 \neq u_2$, we have $u \notin \Delta_m$. Now for each edge $a \in Q_1$ with both $s(a), t(a)$ lying in the same $\Gamma_i$, the equation \[u_{t(a)} \Phi_a^0 u_{s(a)}^{-1} =  \Phi_a^0 \] holds since $u_{t(a)} = u_{s(a)} = u_i$. For all other edges $a \in Q_1$, we have $a \notin \Gamma(\Phi)$ so $\Phi_a^0 =0$. Therefore the above equation also holds holds for such edges. Hence $u \in \Stab_{U_{m,d}}(\Phi^0)$, which gives $\Stab_{U_{m,d}}(\Phi^0) \nsubseteq \Delta_m$. 
\end{proof}

From \thref{reinterpretation} and \thref{stabwithmult} we obtain the following corollary. 
\begin{corollary}[$\rho$-stability for toric representations]
    A toric representation $\Phi$ of $(Q,m)$ is $\rho$-stable if and only if $\phi$ is $\rho$-stable and $\Gamma(\Phi)$ is connected. 
\end{corollary}

We now introduce the notion of genericity of a stability parameter for a toric dimension vector, for which the notion of $\rho$-stability is even simpler.

\begin{definition}[Generic toric stability parameter] \thlabel{genericstab}
    A stability parameter $\rho = (\rho_v)_{v \in Q_0}$ is \emph{generic} if $\sum_{v \in Q_0} \rho_v = 0$ but $\sum_{v \in Q_0'} \rho_v \neq 0$ all non-empty subsets $Q_0' \subsetneq Q_0$. 
\end{definition}

 For a generic stability parameter, semistability and stability coincide for toric representations. For an arbitrary dimension vector $d$, there is a similar notion guaranteeing semistability and stability coincide for all $d$-dimensional representations of $Q$ (\cite[Definition 2.1.2]{Crawley2004}).

\begin{proposition}[$\rho$-stability for generic $\rho$] \thlabel{example}
Suppose that $d$ is toric and $\rho$ is generic. Then any $\rho$-semistable representation of $Q$ of dimension $d$ is $\rho$-stable, and moreover its support quiver is connected. Therefore a representation of $(Q,m)$ of rank $d$ is $\rho$-stable if and only if its classical truncation is $\rho$-stable as a representation of $Q$. 
\end{proposition} 

\begin{proof}
    Suppose $\phi$ is a $\rho$-semistable representation of $Q$ of dimension $d$ and for a contradiction suppose that $\Gamma({\phi})$ has several connected components $\Gamma_1, \hdots, \Gamma_n$.
    Let $\{v_1^i, \hdots, v_{m_i}^1\}$ denote the vertices of each $\Gamma_i$. Then $\{v_1^1, \hdots, v_{m_1}^1, v_1^2, \hdots, v_{m_n}^n\} = Q_0$. As $d$ is toric and by assumption we have $\sum_{i,j} \rho_{v_j^i} = 0$, genericity of $\rho$ means there can be no $i \in \{1, \hdots , n\}$ such that $\sum_{j=1}^{m_i} \rho_{v_j^i} = 0$. Therefore there exists an $i$ such that $\sum_{j=1}^{m_i} \rho_{v_j^i} < 0$. If ${\phi}_i$ is the subprepresentation of ${\phi}$ supported on $\Gamma_i$ (that is obtained by setting all homomorphisms except those indexed by edges in $\Gamma_i$ to zero), then $\rho(\phi_i)= \sum_{j=1}^{m_i} \rho_{v_j^i}<0$ contradicts $\rho$-semistability of $\phi$. Hence $\Gamma(\phi)$ must be connected.   
    The final claim follow by definition of $\rho$-stability for representations of $(Q,m)$, where the second condition in \thref{stabwithmult} is equivalent to the support quiver of the representation being connected by \thref{reinterpretation}, which we have just seen is automatic for $\rho$-semistable representations of $Q$.
\end{proof}

\begin{theorem}[$\rho$-stability and coarse moduli spaces for toric representations] \thlabel{toric1}
    Fix a quiver with multiplicity $(Q,m)$, a toric rank vector ${d} = (1,\hdots, 1)$ and a generic stability parameter $\rho$. Then there exists a quasi-projective coarse moduli space for toric representations of $(Q,m)$ whose classical truncation is $\rho$-stable. Moreover, this moduli space admits an explicit projective completion. 
\end{theorem} 

\begin{proof} 
By genericity of $\rho$ and \thref{example}, both conditions of \thref{maintheoremquivers} are satisfied, so we obtain a a quasi-projective coarse moduli space for toric representations of $(Q,m)$ which are $\rho$-stable, or equivalently their classical trunction is $\rho$-stable (or $\rho$-semistable) by \thref{example}.
 \end{proof} 

\begin{remark}[Link with \cite{Wyss2017,Hausel2018,Vernet2023} and future work]
Toric representations were first studied by Wyss \cite{Wyss2017} and Hausel, Letellier and Rodriguez-Villegas \cite{Hausel2018}, showing that the count $A_{Q,d,m}(q)$ when $d = (1,\hdots, 1)$ of isomorphism classes of absolutely indecomposable toric representations of $(Q,m)$ over the finite field $\mathbb{F}_q$ behaves polynomially in $q$ and conjecturing non-negativity of the coefficients, which was recently proved by Vernet \cite{Vernet2023}. This positivity result for toric representations of a quiver with multiplicities generalises the corresponding result for classical representations and arbitrary dimension vector, conjectured by Kac \cite{Kac80} and proved by Crawley-Boevey and Van den Bergh \cite{Crawley2004} for indivisible dimension vectors and by Hausel, Lettelier and Rodriguez-Villegas \cite{Hausel2013} in general. Both proofs interpret the Kac polynomial in terms of the cohomology of an associated quiver variety (parametrising representations of the preprojective algebra of the doubled quiver). Davison \cite{Davison2023} reproved this result using the cohomology of an associated preprojective stack instead. Vernet's proof in the toric case with multiplicities also uses preprojective stacks. In future work with Gergely B\'{e}rczi we will address the question of whether this result can also be proved analogously to the original proofs \cite{Crawley2004,Hausel2013}, using the moduli spaces for toric representations with multiplicities, which we have constructed in this section.     
\end{remark}

\subsubsection{A complete worked example} \label{subsubsec:workedexample}

Consider the quiver \[Q = 
\begin{tikzcd}
v_1  \arrow[r,bend right,"a_1"]   & v_2 \arrow[l,bend right,"a_2"]
\end{tikzcd}
\] with dimension vector $d= (1,1)$ and multiplicity vector $m  = (1,1)$. Let $\rho = (n, -n)$ with $n >0$, so that $\rho$ is generic. By \thref{toric1} there is a quasi-projective coarse moduli space $\mathcal{M}^{\rho-s}(Q, m,d)$ for toric representations of $(Q, m)$ whose classical truncation is $\rho$-stable. A representation $(\alpha_1, \alpha_2)$ of $Q$ with dimension vector $d$ is $\rho$-stable if and only if $\alpha_2 \neq 0$. Therefore a representation $(\Phi_1,\Phi_2)$ of $(Q,m)$ of rank $d$ is $\rho$-semistable if and only if its classical truncation $(\alpha_1,\alpha_2)$ satisfies $\alpha_2 \neq 0$.

We can construct the geometric quotient  $\mathcal{M}^{\rho-s}(Q, m,d)=\Qss{V}{H}(\rho)/H$ by hand in this example. We view $V = \Mat_{1 \times 1}(k[\epsilon]/\epsilon^2) \times \Mat_{1 \times 1}(k[\epsilon]/\epsilon^2)$ as the vector space of pairs of matrices $$ (\Phi_1,\Phi_2) = \left( \begin{pmatrix} \alpha_1 & \beta_1 \\ 0 & \alpha_1 \end{pmatrix}, \begin{pmatrix} \alpha_2 & \beta_2 \\ 0 & \alpha_2 \end{pmatrix} \right) $$ where $\alpha_1,\beta_1,\alpha_2,\beta_2 \in k$. The group $H$ can be identified with the subgroup of $\GL_2(k)$ consisting of matrices of the form $$ \begin{pmatrix} 
t & u \\
0 & t
\end{pmatrix},$$ with $h \in H$ acting on $(\Phi_1,\Phi_2) \in V$ via $ h \cdot (\Phi_1,\Phi_2) = (h \Phi_1, h^{-1} \Phi_2 ).$  Consider the map $$q: \Qss{V}{H}(\rho) \to \mathbb{A}^2, \quad \left( \begin{pmatrix} \alpha_1 & \beta_1 \\
0 & \alpha_1 \end{pmatrix}, \begin{pmatrix} \alpha_2 & \beta_2 \\
0 & \alpha_2 \end{pmatrix} \right) \mapsto (\alpha_1 \alpha_2, \alpha_2 \beta_1 +\alpha_1 \beta_2).$$ By direct calculation one can show that the map $q$ is $H$-invariant, surjective and has fibres given by $H$-orbits. Since $\Qss{V}{H}(\rho)$ is irreducible and $\mathbb{A}^2$ is normal, it follows from \cite[Cor 25.3.4]{Tauvel2005} that $q$ is a geometric quotient for the action of $H$ on $\Qss{V}{H}(\rho)$. 

This example can also be used to show that the converse to \thref{lem comparing two stab defs for quivers} may not hold. In this example we have that $\Phi= (\Phi_1,\Phi_2)$ is Rudakov $\rho$-stable if and only if $\Phi_2 \neq 0$. Thus  the locus of Rudakov $\rho$-stable representations in $V$ is strictly larger than the locus of $\rho$-stable representations. We may ask whether this larger locus admits a geometric $H$-quotient. The quotient map $q: \Qss{V}{H}(\rho) \to \mathbb{A}^2$ constructed in Section \ref{subsubsec:workedexample}  extends to a well-defined $H$-invariant morphism from the larger locus of Rudakov $\rho$-stable representations, but it is not a geometric quotient on this larger locus. One way to see this is that while $q$ is invariant under switching the two matrices, if $\Phi_1 \neq \Phi_2$ then $(\Phi_1,\Phi_2)$ and $(\Phi_1,\Phi_2)$ only lie in the same $H$-orbit if $(\Phi_1,\Phi_2) \in \Qss{V}{H}$. In fact the locus of Rudakov $\rho$-stable representations does not admit a geometric quotient for the action of $H$. This can be seen from the fact that these representations do not all have the same stabiliser dimension, and therefore a geometric quotient cannot exist \cite[\S 25.3.5]{Tauvel2005}.

\bibliographystyle{alpha}

\end{document}